\documentclass[a4paper,reqno]{amsart}
\usepackage{amscd,amsthm,amsmath,amssymb,mathtools}
\usepackage{upgreek,bm}
\usepackage{verbatim}
\usepackage{xcolor}
\usepackage{hyperref}
\usepackage{url}
\usepackage{enumerate,enumitem}

\usepackage{graphicx}
\usepackage{epstopdf,epsfig,subfigure}
\usepackage{curve2e}

\usepackage{array}
\usepackage[numbers,sort&compress]{natbib}

\theoremstyle{plain}
 \newtheorem*{mainresult}{Main Result}

\newtheorem{theorem}{Theorem}[section]

\newtheorem{lemma}[theorem]{Lemma}

\newtheorem{assumption}[theorem]{Assumption}

\theoremstyle{definition}

\newtheorem{remark}[theorem]{Remark}

\numberwithin{equation}{section}

\usepackage{geometry}

\newcommand{\linspan}{\mathop{\rm span}\nolimits}

\newcommand{\rest}{\left.\kern-2\nulldelimiterspace\right|_}
\newcommand{\norm}[2]{\left|#1\right|_{#2}}
\newcommand{\dnorm}[2]{\left\|#1\right\|_{#2}}

\newcommand{\Id}{{\mathbf1}}
\newcommand{\indf}{1}

\newcommand{\ex}{\mathrm{e}}
\newcommand{\p}{\partial}
\newcommand{\ed}{\mathrm d}

\newcommand{\deltafun}{\bm\updelta}

\newcommand*{\Bigcdot}{\raisebox{-.25ex}{\scalebox{1.25}{$\cdot$}}}

\newcommand{\clA}{{\mathcal A}}

\newcommand{\clC}{{\mathcal C}}
\newcommand{\clD}{{\mathcal D}}
\newcommand{\clE}{{\mathcal E}}

\newcommand{\clG}{{\mathcal G}}

\newcommand{\clL}{{\mathcal L}}
\newcommand{\clM}{{\mathcal M}}
\newcommand{\clN}{{\mathcal N}}

\newcommand{\clU}{{\mathcal U}}

\newcommand{\clW}{{\mathcal W}}
\newcommand{\clX}{{\mathcal X}}

\newcommand{\clZ}{{\mathcal Z}}

\newcommand{\bbN}{{\mathbb N}}

\newcommand{\bbR}{{\mathbb R}}

\newcommand{\bbT}{{\mathbb T}}

\newcommand{\bfE}{{\mathbf E}}

\newcommand{\fkA}{{\mathfrak A}}

\newcommand{\fkE}{{\mathfrak E}}

\newcommand{\fkG}{{\mathfrak G}}

\newcommand{\fkI}{{\mathfrak I}}

\newcommand{\fkM}{{\mathfrak M}}
\newcommand{\fkN}{{\mathfrak N}}

\newcommand{\fkP}{{\mathfrak P}}

\newcommand{\fkX}{{\mathfrak X}}

\newcommand{\fkZ}{{\mathfrak Z}}

\newcommand{\rmD}{{\mathrm D}}

\newcommand{\bfy}{{\mathbf y}}

\newcommand{\rmd}{{\mathrm d}}

\newcommand{\rmf}{{\mathrm f}}

\newcommand{\fka}{{\mathfrak a}}

\newcommand{\fki}{{\mathfrak i}}

\newcommand{\fkm}{{\mathfrak m}}

\newcommand{\fkr}{{\mathfrak r}}
\newcommand{\fks}{{\mathfrak s}}

\newcommand{\fkw}{{\mathfrak w}}

\newcommand{\tte}{{\mathtt e}}

\newcommand{\ttr}{{\mathtt r}}

\newcommand{\ovlineC}[1]{\overline C_{\left[#1\right]}}

\definecolor{DarkBlue}{rgb}{0,0.08,0.45}
\definecolor{DarkRed}{rgb}{.65,0,0}
\definecolor{applegreen}{rgb}{0.55, 0.71, 0.0}

\newcounter{mymac@matlab}
\newcommand{\matlab}{MATLAB%
   \ifnum\value{mymac@matlab}<1%
   \textregistered%
   \setcounter{mymac@matlab}{1}%
   \fi%
  }

\begin{document}
\title{Dynamical observers for parabolic equations with spatial point measurements}
\author{S\'ergio S.~Rodrigues$^{\tt1,*}$}
\author{Dagmawi A. Seifu$^{\tt1}$}
\thanks{
\vspace{-1em}\newline\noindent
{\sc MSC2020}: 93C20, 93C50, 93B51, 93E10.%\blue\quad\hfill\today.\black\quad
\newline\noindent
$^{\tt1}$ Johann Radon Institute for Computational and Applied Mathematics,
  \"OAW, Altenbergerstr. 69, 4040 Linz, Austria.
  \newline\noindent
  $^{*}$ Corresponding author.
\newline\noindent
{\sc Emails}:
{\small\tt   sergio.rodrigues@ricam.oeaw.ac.at, dagmawi.seifu@ricam.oeaw.ac.at}
 }

\begin{abstract}
An exponential Luenberger dynamical observer is proposed to estimate  the state of a general class of nonautonomous semilinear parabolic equations. The result can be applied to the case where the output is given by state measurements taken at a finite number of spatial points, that is, to the case where our sensors are a finite number of delta distributions. The output injection operator is explicit and the derivation of the main result involves the decomposition of the state space into a direct sum of two oblique components depending on the set of  sensors. Simulations are presented as an application to the Kuramoto--Sivashinsky models for flame propagation and fluid flow.
\end{abstract}

\maketitle

\pagestyle{myheadings} \thispagestyle{plain} \markboth{\sc S. S. Rodrigues and D.A. Seifu}
{\sc  Observer design for delta distribution sensors}

%%%%%%%%%%%%%%%%%%%%%%%%%%%%%%%%%%%%%%%%%%%%%%%%%%%%%%%%%%%%%%%%%%%%%%%%%%%%%%%

%%%%%%%%%%%%%%%%%%%%%%%%%%%%%%%%%%%%%
%%%%%%%%%%%%%%%%%%%%%%%%%%%%%%%%%%%%%
\section{Introduction}
We address the design of an observer for general semilinear parabolic-like equations. As an example of application, we shall consider the Kuramoto--Sivashinsky equation
\begin{subequations}\label{sys-KS-flame}
  \begin{align}
 &\tfrac{\p}{\p t}y_\ttr+\nu_2 \Delta^2 y_\ttr+\nu_1 \Delta  y_\ttr+\nu_0\tfrac{1}{2}\norm{\nabla y_\ttr}{\bbR^d}^2=f,\qquad\clG y_\ttr\rest{\p\Omega}=g,\\
 &w=\clZ_{S} y_\ttr,
\end{align}
which is a model for flame propagation. The state~$y_\ttr= y_\ttr(t,x)\in\bbR$ is defined for~$(t,x)\in(0,+\infty)\times\Omega$, in a bounded spatial domain~$\Omega\subset\bbR^d$, $d\in\{1,2,3\}$.
Above~$\nu_i>0$, $i\in\{0,1,2\}$, are positive constants; $f=f(t,x)$ and~$g=g(t,x)$ are external forces; the operator~$\clG$ imposes the boundary conditions on the boundary~$\p\Omega$ of~$\Omega$.

The initial state~$y_\ttr(0,x)\in W^{2,2}(\Omega)$ is assumed to be unknown and our goal is to obtain an estimate of the state~$y_\ttr(t,x)$, for time~$t>0$. For that we shall use the vector output~$w=\clZ_{S} y_\ttr\in \bbR^{S_\sigma}$ of a finite number~$S_\sigma$ of measurements. We are particularly interested in the case where the state measurements are taken at a finite subset of spatial points~$\{x^{S,1},x^{S,2},\dots,x^{S,S_\sigma}\}\subset\Omega$, where~$S_\sigma$ is a positive integer, that is,
\begin{equation}\label{sys-Intro-o}%
  w(t)=\clZ_{S} y_\ttr(t,\Bigcdot)\coloneqq\begin{bmatrix}
           y_\ttr (t, x^{S,1}) \\  y_\ttr(t,x^{S,2}) \\ \vdots \\ y_\ttr(t,x^{S,S_\sigma})
         \end{bmatrix}\in\bbR^{S_\sigma\times 1}.
\end{equation}
\end{subequations}
\begin{remark}
The reason why we denote the points where the measurements are taken as~$x^{S,j}$ is that the detectability result will hold for a large enough number of measurements. Thus, it will be convenient to consider a sequence of output  operators~$(\clZ_{S})_{S\in\bbN_+}$, corresponding to measurements at a set of points~$\{x^{S,j}\mid 1\le j\le S_\sigma\}$. Namely, for a fixed~$S$, we will have~$S_\sigma=\sigma(S)$ sensors, for a strictly increasing function~$\sigma\colon\bbN\to\bbN$, and the location of each sensor will be defined by the pair~$(S,j)$.
\end{remark}

To obtain a state estimate~$y_\tte$ for~$y_\ttr$ we design a Luenberger observer as
\begin{subequations}\label{sys-KS-flame-obs}
  \begin{align}
 &\tfrac{\p}{\p t}y_\tte+\nu_2 \Delta^2y_\tte+\nu_1 \Delta y_\tte+\nu_0\tfrac{1}{2}\norm{\nabla y_\tte}{\bbR^d}^2=f+\fkI_{S}^{[\lambda,\Lambda]}(\clZ_{S}y_\tte-w),\qquad\clG y_\tte\rest{\p\Omega}=g, \\ & y_\tte(0,x)= y_{\tte0}\in W^{2,2}(\Omega),
\end{align}
\end{subequations}
for a suitable output injection operator~$\fkI_{S}^{[\lambda,\Lambda]}\colon\bbR^{S_\sigma}\to L^2(\Omega)$, so that
\begin{align}\label{goal} 
 \norm{y_\tte(t,\Bigcdot)-y_\ttr(t,\Bigcdot)}{W^{2,2}(\Omega)}\le \varrho\ex^{-\mu (t-s)}\norm{y_\tte(s,\Bigcdot)-y_\ttr(s,\Bigcdot)}{W^{2,2}(\Omega)},
 \quad\mbox{for}\quad t\ge s\ge0.
\end{align}

\begin{remark}(On notation).
Note that we use  the subscript~$\ttr$ to denote the {\em reference}~$y_\ttr$  targeted  state, and we use the subscript~$\tte$ to denote its {\em estimate}~$y_\tte$ provided by the observer. A common notation in the literature is to denote the  targeted state by a simple alphabetical letter as~$y$ and its estimate by~$\widehat y$ (cf.~\cite{JadachowskiMeurerKugi15,SmyshlyaevKrstic05,AguiarHesp09,KatzFridman22}). The reason we do not use the circumflex accent~``$\widehat{\;\;}$'' for the estimate is to avoid potential confusion with some references where this accent is used to denote targeted trajectories; see~\cite{BarRodShi11,Rod21-amo} in the context of stabilizability. Though that context  is conceptually different from the context of detectability/observer design addressed in this manuscript, these two contexts are combined in output-based stabilization feedback control problem applications.  
\end{remark}

For the error~$z\coloneqq y_\tte-y_\ttr$ we find the dynamics
\begin{subequations}\label{sys-KS-flame-z} 
\begin{align}
 &\tfrac{\p}{\p t}z+\nu_2 \Delta^2z+\nu_1 \Delta z+\nu_0(\nabla y_\ttr,\nabla z)_{\bbR^d}+\nu_0\tfrac{1}{2}\norm{\nabla z}{\bbR^d}^2=\fkI_{S}^{[\lambda,\Lambda]}\clZ_{S}z,\\
&z(0)=z_0\coloneqq y_\tte(0)- y_\ttr(0),
\end{align} 
and the goal in~\eqref{goal} reads
\end{subequations}
\begin{align}\label{goal-error} 
 \norm{z(t,\Bigcdot)}{W^{2,2}(\Omega)}\le \varrho\ex^{-\mu (t-s)}\norm{z(s,\Bigcdot)}{W^{2,2}(\Omega)},
 \quad\mbox{for}\quad t\ge s\ge0.
\end{align}

Let us denote the set of points where the measurements are taken as
 \begin{equation}\label{WS-Intro}
 W_S\coloneqq\{x^{S,j}\mid 1\le j\le S_\sigma\}\subset\Omega.
  \end{equation}

When applying the main result of this manuscript to the concrete model~\eqref{sys-KS-flame}, under appropriate boundary conditions, we can take
\begin{align}\label{sys-haty-o-Inj}
 &\fkI_{S}^{[\lambda,\Lambda]}\coloneqq-\lambda \nu_2^{-1}(-\Delta+\Id)^{-2}\clZ_S^*\Lambda \quad
 \mbox{with}\quad \Lambda\in\bbR^{S_\sigma\times S_\sigma},\quad{\rm eig}(\Lambda+\Lambda^\top,1)=1,
 \end{align}
  where~$\clZ_S^*$ stands for the adjoint of~$\clZ_S$ and~${\rm eig}(\Lambda+\Lambda^\top,1)$ stands for the smallest eigenvalue of~$\Lambda+\Lambda^\top$. Then, the main result reads as follows.
\begin{mainresult}
There exists a sequence~$(W_S)_{S\in\bbN_+}$ of subsets as in~\eqref{WS-Intro} such that for any given~$R>0$, $\varrho>1$, and~$\mu>0$, there is a large
 enough~$S\in\bbN_+$, for which we can find a large enough~$\lambda>0$ such that: for all
 initial error satisfying~$\norm{z(0,\Bigcdot)}{W^{2,2}(\Omega)}\le R$, it follows that the
 corresponding solution of~\eqref{sys-KS-flame-obs},
 with the output injection operator as in~\eqref{sys-haty-o-Inj}, satisfies~\eqref{goal-error}.
 \end{mainresult}

Explicit locations for the points in~$W_S$ will be given later on.

Though the initial state~$y_\ttr(0,x)$ is unavailable, the choice of~$y_{\tte0}=y_\tte(0,x)$ is at our disposal,
for example, we can choose~$y_{\tte0}$ as an initial guess we might have for~$y_\ttr(0,x)$.

Note that Main Result is semiglobal, that is,  the error will converge exponentially to zero, for {\em arbitrary large initial errors}~$\norm{z_0}{V}\le R$, 
with an {\em arbitrarily large exponential rate}~$\mu>0$, and
{\em arbitrary small transient bound}~$\varrho>1$, provided we take a large enough number of sensors~$S_\sigma$, depending on~$(R,\mu,\varrho)$, and
a large enough~$\lambda$, depending on~$(R,\mu,\varrho,S_\sigma,\Lambda)$.
We would like to mention that, though in particular settings we can take the ``optimal'' transient bound~$\varrho=1$, this may be not  always possible. The case~$\varrho=1$
is interesting as it implies that the error norm is {\em strictly} decreasing. Also, in theory we can take an arbitrary~$\Lambda$ satisfying~${\rm eig}(\Lambda+\Lambda^\top,1)$, however in practical applications the choice of~$\Lambda$ can play a crucial role in the performance of the observer.

%%%%%%%%%%%%%%%%%%%%%%%%%%%%%
%%%%%%%%%%%%%%%%%%%%%%%%%%%%%
\subsection{A more concrete form for the injection operator}\label{sS:Inj-as-deltas}
Let~$\clZ_S$ be as in~\eqref{sys-Intro-o} and let~$v\in\bbR^{S_\sigma\times 1}$ and~$h\in W^{2,2}(\Omega)$. Then,  we find
 \[
 v^\top\clZ_Sh={\textstyle\sum\limits_{i=1}^{S_\sigma}}v_{i,1}(\clZ_Sh)_{i,1}
 ={\textstyle\sum\limits_{i=1}^{S_\sigma}}v_{i,1}h(x^{S,i})=\langle{\textstyle\sum\limits_{i=1}^{S_\sigma}}v_{i,1}\deltafun_{x^{S,i}},h\rangle_{W^{-2,2}(\Omega),W^{2,2}(\Omega)}
 \]
and we can write $\clZ_S^*v={\textstyle\sum\limits_{i=1}^{S_\sigma}}v_{i,1}\deltafun_{x^{S,i}}$, which allows us to write
the injection operator in~\eqref{sys-haty-o-Inj}, for a given output~$w\in\bbR^{S_\sigma\times 1}$ and with~$A\coloneqq\nu_2(-\Delta+\Id)^{2}$, as
\begin{align}\label{sys-haty-o-Inj-delta}
 &\fkI_{S}^{[\lambda,\Lambda]}w=-\lambda A^{-1}\clZ_S^*\Lambda w=-\lambda A^{-1}{\textstyle\sum\limits_{i=1}^{S_\sigma}} (\Lambda w)_{i,1}\deltafun_{x^{S,i}}=-\lambda {\textstyle\sum\limits_{i=1}^{S_\sigma}} (\Lambda w)_{i,1} A^{-1}\deltafun_{x^{S,i}}.
 \end{align}
That is, the forcing injected by~$\fkI_{S}^{[\lambda,\Lambda]}$ into the observer dynamics  is a linear combination of the functions~$A^{-1}\deltafun_{x^{S,i}}\in  L^2(\Omega)$, where~$\deltafun_{x^{S,i}}$ is the delta distribution located at the spatial point~$x^{S,i}\in\Omega$, that is, $\langle\deltafun_{x^{S,i}}, h\rangle_{W^{-2,2}(\Omega),W^{2,2}(\Omega)}
 \coloneqq h(x^{S,i})$.

\begin{remark}
We do not know whether both existence and uniqueness of weak solutions hold true for model~\eqref{sys-KS-flame-obs}, for external forces as delta distributions as~$\deltafun_{x^{S,i}}\in W^{-2,2}(\Omega)$. This is the reason we consider the stabilization of strong solutions, and to have such solutions we need more regular external forces, hence we take the more regular functions~$A^{-1}\deltafun_{x^{S,i}}\in L^2(\Omega)$ in~\eqref{sys-haty-o-Inj-delta}. 
\end{remark}

%%%%%%%%%%%%%%%%%%%%%%%%%%%%%
%%%%%%%%%%%%%%%%%%%%%%%%%%%%%
\subsection{Motivation and literature}\label{sS:motivation}
Recovering the state of a given system from the output of a set of measurements is an interesting subject on its own (cf. the data assimilation results in~\cite{AzouaniOlsonTiti14,OlsonTiti03,MarkowichTitiTrabelsi16}). They also play a paramount role  in control applications as in the implementation of output based stabilizing feedback controls (cf. the discussion in~\cite[sect.~1.1]{Rod21-jnls}). See also~\cite{KangFridman19,BuchotRaymondTiago15,Rod21-aut}.

We cannot expect
that the state~$y_\ttr(t,\Bigcdot)$ living in an infinite-dimensional space of functions~$W^{2,2}(\Omega)$ can be reconstructed from the
finite-dimensional vector~$w (t)=\clZ_{S}y_\ttr(t,\Bigcdot)$ at a fixed time~$t$, this is why we look for
a {\em dynamical} observer as~\eqref{sys-KS-flame-obs}, in order to construct an estimate~$y_\tte(t,\Bigcdot)$ for~$y_\ttr(t,\Bigcdot)$, 
which will be improving as time increases.

In this manuscript we consider a general class of nonatonomous semilinear dynamical systems.
We are particularly interested in the case the output is given by measurements of the state  at a finite number of spatial points. In~\cite[sect.~1.1]{Rod21-jnls} an observer is presented for an analogous class of dynamical systems  in the case the output is given by measurements of the average of the state  at a finite number of spatial subdomains.

Both average and point measurements are covered by the method in~\cite{AzouaniOlsonTiti14} corresponding to so called
determining volumes/averages and determining nodes/points. In the case of average measurements, in~\cite[Rem.~3.13]{Rod21-jnls} a difference is pointed out between a result of the approaches followed in~\cite{Rod21-jnls} and~\cite{AzouaniOlsonTiti14}, namely, concerning the (order of the) choice of the number of sensors~$S_\sigma$ and the gain parameter~$\lambda$ in the output injection operator. The approach is this manuscript is closer to that in~\cite{Rod21-jnls} and an analogue difference holds in the case of point measurements. In our approach we look for a monotone output injection operator, while in~\cite{AzouaniOlsonTiti14} the monotonicity is not required. Our approach is proven to work for a more general class of systems, in particular, the solutions of the free dynamics in~\cite{AzouaniOlsonTiti14} are assumed to be globally defined  in time, while in our approach a general  free dynamics solution is allowed to blow-up in finite-time, only the targeted solution~$y_\ttr$ is assumed to be globally defined  in time. In particular, the output injection operator should be constructed so that the solution~$y_\tte$ of the observer (or, equivalently,  the solution~$z=y_\tte-y_\ttr$ of the error dynamics) is globally defined in time (i.e., so that finite-time blow up does not occur and~\eqref{goal} makes sense for all~$t>0$).

We underline that the output injection operator we propose here is explicitly given as well as those as in~\cite{AzouaniOlsonTiti14}. Thus these observers will be able to give us an estimate of the state in real time, which can play an important role in applications, for example, in the performance of feedback stabilizing controls based on the state estimate provided by the observer. 

At this point, we refer the reader also to the works~\cite{MeurerKugi09,Meurer13,JadachowskiMeurerKugi13} for one-dimensional parabolic equations, $d=1$, by using a different popular method involving the
nontrivial backstepping and Cole-Hopf transformations.

For autonomous dynamics, we can often derive detectability results (i.e., the existence of exponential observers) from the spectral properties of the time-independent operators defining the dynamics~\cite{RamdaniTucsnakValein16}. Spectral properties are likely not appropriate to tackle the nonautonomous case~\cite{Wu74}. Thus instead of departing from spectral properties, the proof of the results departs from standard energy estimates, that is, the explicit output injection is constructed in such a way that from such estimates we are able to conclude the exponential decrease of the norm of the error.

%%%%%%%%%%%%%%%%%%%%%%%%%%%%%
%%%%%%%%%%%%%%%%%%%%%%%%%%%%%
 \subsection{Abstract formulation}\label{sS:abstr-form}
 We shall prove the results for a general class of evolution parabolic-like equations as
 \begin{align}\label{sys-y} 
 \dot y_\ttr +Ay_\ttr+A_{\rm rc}y_\ttr +\clN(y_\ttr)=f,\qquad w=\clZ_{S} y_\ttr,
\end{align} 
under general assumptions on the plant
operators~$A$, $A_{\rm rc}$, $\clN$, and on our targeted state~$y_\ttr$. The satisfiability of these assumptions shall be checked for the concrete Kuramoto--Sivashinsky model~\eqref{sys-KS-flame}. 
An estimate for the state~$y_\ttr$ of system~\eqref{sys-y} shall be given by an abstract Luenberger observer as 
 \begin{align}\label{sys-haty} 
 &\dot{y_\tte} +Ay_\tte+A_{\rm rc}y_\tte +\clN(y_\tte)=f-\lambda A^{-1}\overline\fkI_{S}(\clZ_Sy_\tte-w),\qquad y_\tte(0)=y_{\tte0},
\end{align}
for an output operator~$\clZ_S\colon \rmD(A)\to\bbR^{S_\sigma\times1}$ obtained from a finite number~${S_\sigma}$ of sensors
\begin{equation}\label{sens-V'} 
\clZ_Sv\coloneqq\begin{bmatrix}\fkw^{S,1}v\\ \fkw^{S,2}v\\ \vdots\\ \fkw^{S,S_\sigma}v\end{bmatrix},\quad \fkw^{S,j}\in \rmD(A)',\quad \fkw^{S,j}v\coloneqq\langle\fkw^{S,j},v\rangle_{\rmD(A)',\rmD(A)},\qquad 1\le j\le S_\sigma.
\end{equation}
 where~$\rmD(A)=\{h\in H\mid Ah\in H\}$ is the domain of~$A$ in a  pivot Hilbert space~$H=H'$.
 
 The result will follow for large enough~$S_\sigma$ and~$\lambda$, and for {\em an appropriate set} of sensors, which in some concrete applications/examples can be translated to {\em an appropriate location} of the sensors.

%%%%%%%%%%%%%%%%%%%%%%%%%%%%%%%%%%%%%%%%%%%%%%%%%%%%%%%%%%%%%%%%%%%%%%%%%%
\subsection{Contents and notation}
In Section~\ref{S:assumptions} we present the assumptions we require for the
dynamics plant operators and for all the ``parameters'' involved in
the output injection operator. 
In Section~\ref{S:exp_error} we prove that under such assumptions the error norm
of the observer estimate decreases
exponentially to zero.  The satisfiability of the assumptions are shown in Section~\ref{S:KS-satisfAssumpt} for the the model~\eqref{sys-KS-flame} for flame propagation. Numerical simulations showing the exponential stability of the error
dynamics are presented in Section~\ref{S:simul-flame} for the same model. Since many results in the literature address the Kuramoto--Sivashinsky model~\eqref{sys-KS-flame} for fluid flow, in Section~\ref{S:simul-fluid} we also present simulations showing the performance of the proposed observer for this model.

\smallskip

Concerning notation, we write~$\bbR$ and~$\bbN$ for the sets of real numbers and nonnegative
integers, respectively, and we set $\bbR_+\coloneqq(0,+\infty)$,  $\overline\bbR_+\coloneqq[0,+\infty)$
and~$\bbN_+\coloneqq\bbN\setminus\{0\}$.

Given two Banach spaces~$X$ and~$Y$, if the inclusion
$X\subseteq Y$ is continuous, we write $X\xhookrightarrow{} Y$. We write
$X\xhookrightarrow{\rm d} Y$, respectively $X\xhookrightarrow{\rm c} Y$,
if the inclusion is also dense, respectively compact.

The space of continuous linear mappings from~$X$ into~$Y$ is denoted by~$\clL(X,Y)$. In case~$X=Y$ we 
write~$\clL(X)\coloneqq\clL(X,X)$.
The continuous dual of~$X$ is denoted~$X'\coloneqq\clL(X,\bbR)$.
The adjoint of an operator $L\in\clL(X,Y)$ will be denoted $L^*\in\clL(Y',X')$, and its kernel (i.e., null space)
will be denoted $\ker L\coloneqq \{v\in X\mid Lv=0\}$.

We follow the usual notation~$L^2(\Omega)$ for the Lebesgue space of square integrable functions, and~$W^{k,2}(\Omega)=\{h\in L^2(\Omega)\mid \nabla^ih\in L^2(\Omega),\; 0\le i\le k\}$ for the Sobolev spaces; where~$\nabla^ih$ stands for the partial derivatives of order~$i$, $k\in\bbN$.

The space of continuous functions from~$X$ into~$Y$ is denoted by~$\clC(X,Y)$.
The space of increasing functions in~$\clC(\overline\bbR_+,\bbR)$ vanishing at~$0$ is denoted by
\begin{equation}\notag
 \clC_{0,\iota}(\overline\bbR_+,\bbR) \coloneqq \{\fki\in \clC(\overline\bbR_+,\bbR)\mid
 \; \mathfrak i(0)=0,\quad\!\!\mbox{and}\quad\!\!
 \mathfrak i(\varkappa_2)\ge\mathfrak i(\varkappa_1)\;\mbox{ if }\; \varkappa_2\ge \varkappa_1\ge0\}.
\end{equation}

We also denote the vector subspace~$\clC_{\rm b, \iota}(X, Y)\subset \clC(X,Y)$ by  
\begin{equation}\notag
 \clC_{\rm b, \iota}(X, Y)\coloneqq
 \left\{f\in \clC(X,Y) \mid \exists\mathfrak i\in \clC_{0,\iota}(\overline{\bbR_0},\bbR)\;\forall x\in X:\;
\norm{f(x)}{Y}\le \mathfrak i (\norm{x}{X})
\right\}.
\end{equation}

The orthogonal complement to a given subset~$B\subset H$ of a Hilbert space~$H$,
with scalar product~$(\Bigcdot,\Bigcdot)_H$,  is 
denoted~$B^\perp\coloneqq\{h\in H\mid (h,s)_H=0\mbox{ for all }s\in B\}$.

Given two closed subspaces~$F\subseteq H$ and~$G\subseteq H$ of the
Hilbert space~$H$, we write~$H=F\oplus G$ in case we know that both~$H=F+ G$ and~$F\cap G=\{0\}$.
If~$H=F\oplus G$, we denote by~$P_F^G\in\clL(H,F)$
the oblique projection in~$H$ onto~$F$ along~$G$. That is, writing $h\in H$ as $h=h_F+h_G$
with~$(h_F,h_G)\in F\times G$, we have~$P_F^Gh\coloneqq h_F$.
The orthogonal projection in~$H$ onto~$F$ is denoted by~$P_F\in\clL(H,F)$.
Notice that~$P_F= P_F^{F^\perp}$.

By~$\bbR^{m\times n}$ we denote the space of matrices with~$m$ rows and~$n$ columns, and with real entries.  The entry of~$\clM\in\bbR^{m\times n}$ in the~$i$th row and~$j$th column is denoted by~$\clM_{(i,j)}\in\bbR$; we write~$\clM^\top\in\bbR^{n\times m}$ for the transpose of~$\clM$. The inverse of~$\clM\in\bbR^{m\times m}$, if it exists,  is denoted by~$\clM^{-1}\in\bbR^{m\times m}$;  for simplicity, we also denote~$\clM^{-\top}\coloneqq(\clM^{\top})^{-1}=(\clM^{-1})^{\top}$.
For a symmetric matrix~$\clM\in\bbR^{m\times m}$, $\clM=\clM^\top$,
we denote the sequence of its increasing eigenvalues as
\[
{\rm eig}(\clM,1)\le {\rm eig}(\clM,2)\le\dots\le {\rm eig}(\clM,m),\qquad {\rm eig}(\clM,i)\in\bbR,\quad 1\le i\le m.
\]
We say that a symmetric matrix~$\clM\in\bbR^{m\times m}$ is positive definite if~${\rm eig}(\clM,1)>0$.

Given~$a=(a_1,a_2,\dots,a_n)\in\overline\bbR_+^n$, $n\in\bbN_+$, we
denote~$\|a\|\coloneqq\max\limits_{1\le j\le n} a_j$.

By
$\ovlineC{a_1,a_2,\dots,a_n}$ we denote a nonnegative function that
increases in each of its nonnegative arguments~$a_i$, $1\le i\le n$. To shorten the notation, for~$a\in\overline\bbR_+^n$ and~$b\in\overline\bbR_+^m$ we shall denote~$\ovlineC{a}\coloneqq\ovlineC{a_1,a_2,\dots,a_n}$ and~$\ovlineC{a;b_{1},...,b_{m}}\coloneqq\ovlineC{a_1,\dots,a_n,b_{1},...,b_{m}}$.

Finally, $C,\,C_i$, $i=0,\,1,\,\dots$, stand for unessential positive constants.

%%%%%%%%%%%%%%%%%%%%%%%%%%%%%%%%%%%%
%%%%%%%%%%%%%%%%%%%%%%%%%%%%%%%%%%%%
\section{Assumptions}\label{S:assumptions}
Let an estimate for the targeted state~$y_\ttr$ of the nominal system~\eqref{sys-y} be given by a Luenberger observer as~\eqref{sys-haty}. 
Then, for the error~$z\coloneqq y_\tte-y_\ttr$ we find the dynamics
\begin{subequations}\label{sys-z} 
\begin{align}
 &\dot{z} +Az+A_{\rm rc}z +\fkN_{y_\ttr}z=-\lambda A^{-1}\overline\fkI_{S}\clZ_Sz,\qquad z(0)=z_0\coloneqq y_\tte(0)- y_\ttr(0),\\
 \mbox{with}\quad&\fkN_{y_\ttr}(z)\coloneqq \clN(t,z+y_\ttr)-\clN(t,y_\ttr).
\end{align} 
\end{subequations}

Hereafter, all Hilbert spaces are real and separable. 
We set a pivot Hilbert space~$H$, that is, identified with its continuous dual, $H'=H$, and we set another Hilbert space~$V\subset H$.
\begin{assumption}\label{A:A0sp}
 $A\in\clL(V,V')$ is symmetric and $(v,w)\mapsto\langle Av,w\rangle_{V',V}$ is a complete scalar product in~$V.$
\end{assumption}

From now on, we suppose that~$V$ is endowed with the scalar product~$(v,w)_V\coloneqq\langle Av,w\rangle_{V',V}$,
which still makes~$V$ a Hilbert space.
Necessarily, $A\colon V\to V'$ is an isometry.
\begin{assumption}\label{A:A0cdc}
The inclusion $V\subseteq H$ is dense, continuous, and compact.
\end{assumption}

Necessarily, we have that
\[
 \langle v,w\rangle_{V',V}=(v,w)_{H},\quad\mbox{for all }(v,w)\in H\times V,
\]
and also that the operator $A$ is densely defined in~$H$, with domain $\rmD(A)$ satisfying
\[
\rmD(A)\xhookrightarrow{\rm d,\,c} V\xhookrightarrow{\rm d,\,c} H\xhookrightarrow{\rm d,\,c} V'
\xhookrightarrow{\rm d,\,c}\rmD(A)'.
\]
Further,~$A$ has a compact inverse~$A^{-1}\colon H\to H$,
and we can find a nondecreasing
system of (repeated accordingly to their multiplicity) eigenvalues
$(\alpha_n)_{n\in\bbN_0}$ and a corresponding complete basis of
eigenfunctions $(e_n)_{n\in\bbN_0}$:
\[\label{eigfeigv}
0<\alpha_1\le\alpha_2\le\dots\le\alpha_n\le\alpha_{n+1}\to+\infty \quad\mbox{and}\quad Ae_n=\alpha_n e_n.
\]

We can define, for every $\xi\in\bbR$, the fractional powers~$A^\xi$, of $A$, by
\[
 y={\textstyle\sum\limits_{n=1}^{+\infty}}y_ne_n,\quad A^\xi y=A^\xi {\textstyle\sum\limits_{n=1}^{+\infty}}y_ne_n
 \coloneqq{\textstyle\sum\limits_{n=1}^{+\infty}}\alpha_n^\xi y_n e_n,
\]
and the corresponding domains~$\rmD(A^{|\xi|})\coloneqq\{y\in H\mid A^{|\xi|} y\in H\}$, and
$\rmD(A^{-|\xi|})\coloneqq \rmD(A^{|\xi|})'$.
We have that~$\rmD(A^{\xi})\xhookrightarrow{\rm d,\,c}\rmD(A^{\zeta_1})$, for all $\xi>\xi_1$,
and we can see that~$\rmD(A^{0})=H$, $\rmD(A^{1})=\rmD(A)$, $\rmD(A^{\frac{1}{2}})=V$.

For the time-dependent operator and external forcing we assume the following:

\begin{assumption}\label{A:A1}
For almost every~$t>0$ we have~$A_{\rm rc}(t)\in\clL(V, H)$,
and we have a uniform bound as $\norm{A_{\rm rc}}{L^\infty(\bbR_0,\clL(V, H))}\eqqcolon C_{\rm rc}<+\infty.$
\end{assumption}

\begin{assumption}\label{A:realy}
There exists a closed subspace~$\fkG$ of~$\rmD(A)$ such that the orthogonal component~$y_{\ttr\fkG}\coloneqq P_\fkG^{\fkG^{\perp\rmD(A)}} y_\ttr$ of the reference state~$y_\ttr$, solving~\eqref{sys-y}, is persistently uniformly bounded as follows. There are constants~$C_{y_\ttr}\ge0$ and~$\tau_{y_\ttr}>0$ such that
\begin{align}
\sup_{s\ge0}\norm{y_{\ttr\fkG}(s)}{V}\le C_{y_\ttr}\quad\mbox{and}\quad\sup_{s\ge0}\norm{y_{\ttr\fkG}}{L^2((s,s+\tau_{y_\ttr}),\rmD(A))}<C_{y_\ttr}.\notag
\end{align}
\end{assumption}

\begin{assumption}\label{A:N}
 We have~$\clN(t,\Bigcdot)\in\clC_{\rm b,\iota}(\rmD(A),H)$ 
and~$\clN(t,y)=\clN(t,P_\fkG^{\fkG^{\perp\rmD(A)}} y)$ for all~$y\in \rmD(A)$, where~$\fkG\subseteq\rmD(A)$ is as in Assumption~\ref{A:realy}. Furthermore,
there exist constants $C_\clN\ge 0$, $n\in\bbN_0$, 
$\zeta_{1j}\ge0$, $\zeta_{2j}\ge0$,
 $\delta_{1j}\ge 0$, ~$\delta_{2j}\ge 0$, 
 with~$j\in\{1,2,\dots,n\}$, such that
 for  all~$t>0$ and
 all~$(y_1,y_2)\in \fkG\times\fkG$, we have
\begin{align}
&\norm{\clN(t,y_1)-\clN(t,y_2)}{H}\le C_\clN\textstyle\sum\limits_{j=1}^{n}
  \left( \norm{y_1}{V}^{\zeta_{1j}}\norm{y_1}{\rmD(A)}^{\zeta_{2j}}+\norm{y_2}{V}^{\zeta_{1j}}
  \norm{y_2}{\rmD(A)}^{\zeta_{2j}}\right)
   \norm{d}{V}^{\delta_{1j}}\norm{d}{\rmD(A)}^{\delta_{2j}},\notag
\end{align}
with~$d\coloneqq y_1-y_2$, $\zeta_{2j}+\delta_{2j}<1$ and~$\delta_{1j}+\delta_{2j}\ge1$.
\end{assumption}

Together with the set~$W_S$ of sensors we will need an auxiliary set of functions~$\widetilde W_S$.
 \begin{assumption}\label{A:sens} 
The sequence of pairs~$(W_S,\widetilde W_S)_{S\in\bbN}$ satisfies
\begin{align}
&W_S\coloneqq\{\fkw^{S,j}\mid 1\le j\le S_\sigma\}\subseteq  \rmD(A^{-1}),&\quad\clW_S\coloneqq\linspan W_S,&\quad\dim \clW_S=S_\sigma;\notag\\
&\widetilde W_S\coloneqq\{\Psi^{S,j}\mid 1\le j\le S_\sigma\}\subseteq \rmD(A),&\quad\widetilde \clW_S\coloneqq\linspan \widetilde W_S,&\quad\dim \widetilde \clW_S=S_\sigma;\notag
\end{align}
with~$S_\sigma\coloneqq\sigma(S)$, where~$\sigma\colon\bbN_+\to\bbN_+$ is a strictly increasing function.
Furthermore, we have the direct sum~$\rmD(A)=\widetilde \clW_S\oplus\fkZ_S$, where~$\fkZ_S\coloneqq\rmD(A){\textstyle\bigcap}\ker\clZ_S$.
\end{assumption}

\begin{assumption}\label{A:Poincare}
The sequence~$(\beta_{S})_{S\in\bbN_0}$ of the Poincar\'e-like constants
\begin{align}\label{Poinc_const}
\beta_{S}\coloneqq\inf_{\Theta\in \fkZ_S\setminus\{0\}}
\tfrac{\norm{\Theta}{\rmD(A)}^2}{\norm{\Theta}{V}^2},
\end{align}
is divergent, $\lim\limits_{S\to+\infty}\beta_{S}=+\infty$. Where~$\fkZ_S=\rmD(A){\textstyle\bigcap}\ker\clZ_S$ as in Assumption~\ref{A:sens}. 
\end{assumption}

 \begin{assumption}\label{A:Inj} 
The operator~$\overline\fkI_{S}\in\clL(\bbR^{S_\sigma},\rmD(A^{-1}))$  satisfies
\[
(A^{-1}\overline\fkI_{S}\clZ_S z,Az)_H \ge \tfrac12\norm{\clZ_Sz}{\bbR^{S_\sigma}}^2, 
\quad \mbox{for all} \quad z\in D(A).
\]
\end{assumption}

%%%%%%%%%%%%%%%%%%%%%%%%%%%%%%
%%%%%%%%%%%%%%%%%%%%%%%%%%%%%%
\section{Exponential stability of the error dynamics}\label{S:exp_error}
Recall the auxiliary functions~$\Psi^{S,j}\in\widetilde\clW_S$ in Assumption~\ref{A:sens} and its linear span giving us the direct sum~$\rmD(A)=\widetilde \clW_S\oplus\fkZ_S$, where~$\fkZ_S\coloneqq\rmD(A){\textstyle\bigcap}\ker\clZ_S$. Let us introduce
\begin{equation}\label{normOP}
C^P_S=\sup_{h\in\rmD(A)\setminus\ker\clZ_S}\tfrac{  \norm{P_{\widetilde \clW_S}^{\fkZ_S}h}{\rmD(A)} }{ \norm{\clZ_Sh}{\bbR^{S_\sigma}}  }.
\end{equation} 
We shall show that~$C^P_S<+\infty$, for an arbitrary fixed~$S$.

The main result of this manuscript, in abstract form, is the following. 
 \begin{theorem}\label{T:main}
 Let Assumptions~\ref{A:A0sp}--\ref{A:Inj} hold true. Let us be
 given~ $R>0$, $\varrho>1$, and~$\mu>0$. Then the exists~$S_*\in\bbN_+$ such that
 the error solving~\eqref{sys-z}
satisfies
\begin{align}\label{goal-dif-error} 
 \norm{z(t)}{V}\le \varrho\ex^{-\mu (t-s)}\norm{z(s)}{V},\quad\mbox{for all}\quad t\ge s\ge0.
\end{align}
for all~$S\ge S_*$ and all~$\lambda\ge\lambda_*(S)$, where
the constants~$S_*$ and~$\lambda_*(S)$ can be taken as
\begin{align}
 S_*&=\ovlineC{a}\in\bbN_+\quad\mbox{and}\quad \lambda_*(S)=\ovlineC{a;C^P_S}\in\bbR_+,
 \label{formlam*}\\
 \mbox{with}\quad a&\coloneqq(R,\mu,\varrho,\tfrac1{\varrho^\frac12-1},
 \tfrac1{\tau_{y_\ttr}},\tau_{y_\ttr},
 \tfrac{2\dnorm{\zeta_{1}}{}}{1-\dnorm{\delta_{2}+\zeta_{2}}{}},
 \tfrac{2+\dnorm{\frac{2\zeta_{2}}{1-\delta_{2}}}{}}{2-\dnorm{\frac{2\zeta_{2}}{1-\delta_{2}}}{}}
 ,\tfrac{2\dnorm{\delta_{1}+\delta_{2}+\zeta_{1}+\zeta_{2}}{}-2}{1-\dnorm{\delta_{2}+\zeta_{2}}{}},C_{\rm rc},C_{y_\ttr}),\notag
\end{align}
where~$(C_{\rm rc},C_{y_\ttr},\tau_{y_\ttr},\delta_{1j},\delta_{2j},\zeta_{1j},\zeta_{2j})$ are as in Assumptions~\ref{A:A1}--\ref{A:N} and~$ C^P_S$  as in~\eqref{normOP}.
\end{theorem}

\begin{remark}
 Recall the notation~$\|b\|=\max\limits_{1\le j\le n} b_j$, for~$b\in\overline\bbR_+^n$. For example, in Theorem~\ref{T:main}, we have
$\dnorm{\frac{2\zeta_{2}}{1-\delta_{2}}}{}=\max\limits_{1\le j\le n}\frac{2\zeta_{2j}}{1-\delta_{2j}}$.
\end{remark}

 \subsection{Auxiliary results}
We gather results we shall use in the proof of Theorem~\ref{T:main}. We start with a key inequality holding for large enough~$S$ and  large enough~$\lambda=\lambda(S)$.
\begin{lemma}\label{L:bddCPS}
Let~$\rmD(A)=\widetilde \clW_S\oplus\fkZ_S$, where~$\fkZ_S\coloneqq\rmD(A){\textstyle\bigcap}\ker\clZ_S$. Then the
constant in~\eqref{normOP} is bounded,~$C^P_S<+\infty$.
\end{lemma}
\begin{proof}
 Consider the mapping~$\fkE_S\colon\bbR^{S_\sigma\times 1}\to\bbR^{S_\sigma\times 1}$ as follows
\[
v\mapsto \fkE_S v\coloneqq\clZ_S\Psi v=\begin{bmatrix}
  \Psi^{S,j}(x^{S,i})
\end{bmatrix}v,\quad\mbox{with}\quad \Psi v\coloneqq{\textstyle\sum\limits_{j=1}^{4S}}v_{j,1}  \Psi^{S,j}, 
\]
where the~$\Psi^{S,j}$ form a basis for~$\widetilde \clW_S$, as in Assumption~\ref{A:sens}.
We show next that the matrix
$\fkE_S=\begin{bmatrix}
 \Psi^{S,j}(x^{S,i})
\end{bmatrix}
$
with entry~$\fkE_{S,(i,j)}=\Psi^{S,j}(x^{S,i})$ in the~$i$th row and~$j$th column is invertible. 
Indeed, if~$\fkE_Sv=0$ then~$\Psi v\in\widetilde \clW_S\bigcap\fkZ_S=\{0\}$, hence ~$v=0$ because the family~$\widetilde W_S$ is linearly independent. Now the invertibility of~$\fkE_S$ implies that
the oblique
projection~$P_{\widetilde \clW_S}^{\fkZ_S}$ in~$\rmD(A)$ onto~$\widetilde \clW_S$ along~$\fkZ_S$ is given by
\begin{equation}\label{obli_proj}
P_{\widetilde \clW_S}^{\fkZ_S}h=\Psi \fkE_S^{-1}\clZ_S h,
\end{equation}
because~$\clZ_S \Psi \fkE_S^{-1}\clZ_S h=\clZ_S h$ and~$\clZ_S(h- \Psi \fkE_S^{-1}\clZ_S h)=0$, that is,
\[
h= \Psi \fkE_S^{-1}\clZ_S h+(h- \Psi \fkE_S^{-1}\clZ_S h),
\]
with $ \Psi \fkE_S^{-1}\clZ_S h\in\widetilde \clW_S$ and~$h- \Psi \fkE_S^{-1}\clZ_S h\in\fkZ_S$. In particular, from the surjectivity of~$P_{\widetilde \clW_S}^{\fkZ_S}$ and~\eqref{obli_proj} it follows that~$\clZ_S$ is also surjective, which leads us to
\begin{align}
C^P_S&=\sup_{h\in\rmD(A)\setminus\ker\clZ_S}\tfrac{  \norm{P_{\widetilde \clW_S}^{\fkZ_S}h}{\rmD(A)} }{ \norm{\clZ_Sh}{\bbR^{S_\sigma}}  }
=\sup_{h\in\rmD(A)\setminus\ker\clZ_S}\tfrac{  \norm{\Psi \fkE_S^{-1}\clZ_Sh}{\rmD(A)} }{ \norm{\clZ_Sh}{\bbR^{S_\sigma}}  }
=\sup_{v\in\bbR^{S_\sigma\times 1}\setminus\{0\}}\tfrac{  \norm{\Psi \fkE_S^{-1}v}{\rmD(A)} }{ \norm{v}{\bbR^{S_\sigma}}  }.\notag
\end{align} 
Therefore, $C^P_S\le\norm{\Psi \fkE_S^{-1}}{\clL(\bbR^{S_\sigma\times 1},\rmD(A))}<+\infty$, which ends the proof.
\end{proof}

 \begin{lemma}\label{L:MlamPoinc} 
For every constant $\xi >0$ we can find $S_*=\ovlineC{\xi}\in\bbN_+$ such that
\[
\norm{z}{\rmD(A)}^2 + \lambda\norm{\clZ_Sz}{\bbR^{S_\delta}}^2 \ge \xi\norm{z}{V}^2, 
\quad \mbox{for all} \quad z\in \rmD(A),
\]
for all integer~$S\ge S_*$ and for all~$\lambda\ge\lambda_*(S)=\ovlineC{\xi,C^P_S}$, with
with~$C^P_S<+\infty$ as in~\eqref{normOP}.
\end{lemma}
\begin{proof}
Recall the sets of sensors~$W_S$ and auxiliary functions~$\widetilde W_S$, and the subspace~$\fkZ_S=\rmD(A){\textstyle\bigcap}\ker\clZ_S$  in Assumption~\ref{A:sens}, satisfying~$\rmD(A)=\widetilde \clW_S\oplus\fkZ_S$. We can write
\[
z = \Theta + \vartheta, \quad\mbox{with}\quad 
\vartheta\in \widetilde \clW_S \quad\mbox{and}\quad\Theta\in \fkZ_S,
\]
and we find that 
\begin{align}
    \norm{z}{\rmD(A)}^2 + \lambda\norm{\clZ_Sz}{\bbR^{S_\sigma}}^2 &=\norm{\Theta + \vartheta}{D(A)}^2 + \lambda\norm{\clZ_Sz}{\bbR^{S_\sigma}}^2\notag\\
    &=\norm{\Theta}{\rmD(A)}^2+2(\Theta,\vartheta)_{\rmD(A)}+\norm{\vartheta}{\rmD(A)}^2+\lambda\norm{\clZ_Sz}{\bbR^{S_\sigma}}\notag\\
    &\ge \tfrac{1}{2}\norm{\Theta}{\rmD(A)}^2 - \norm{\vartheta}{\rmD(A)}^2 + \lambda\norm{\clZ_Sz}{\bbR^{S_\sigma}}^2.\notag
\end{align}
Then, with~$C^P_S<+\infty$ as in~\eqref{normOP} (cf. Lem.~\ref{L:bddCPS}) and~$\beta_{S}$ as in Assumption~\ref{A:Poincare}, we obtain
\begin{equation}\label{pre-choicelambeta}
\norm{z}{\rmD(A)}^2+\lambda\norm{\clZ_Sz}{\bbR^{S_\sigma}}^2 \ge \tfrac{1}{2}\beta_{S}\norm{\Theta}{V}^2 + (\lambda (C^P_S)^{-2} - 1)\norm{\vartheta}{\rmD(A)}^2.
\end{equation}
Hence, for given $\xi>0$, by choosing
\begin{equation}\notag
S_*\coloneqq\inf_{\widehat S}\{\widehat S\in\bbN_+\mid \min_{S\ge \widehat S}\beta_{S} \ge 4\xi\}
\end{equation}
and, subsequently, 
$S$ and $\lambda$ so that
\begin{equation}\notag
S \ge  S_* \qquad\mbox{and}\qquad \lambda\ge (C^P_S)^2(2\xi+1)
\end{equation}
we have that~$\lambda (C^P_S)^{-2} -1 \ge 2\xi>0$ and, since~$\norm{\vartheta}{V}^2\le\norm{\vartheta}{\rmD(A)}^2$, we arrive at
\begin{equation}\label{lambeta-fin}
\norm{z}{\rmD(A)}^2 +\lambda\norm{\delta(z)}{\bbR^{S_\sigma}}^2 \ge 2\xi\left(\norm{\Theta}{V}^2 + \norm{\vartheta}{V}^2\right)\ge \xi\norm{\Theta+\vartheta}{V}^2=\xi\norm{z}{V}^2,
\end{equation}
which ends the proof.
\end{proof}

 For convenience of reader and to simplify the exposition, we recall next auxiliary results from the literature that we shall use to prove Theorem~\ref{T:main}. 
 The proofs can be found in~\cite{Rod21-jnls}.
We start with auxiliary results for the nonlinear terms~$\clN(t, y)$ and~$\fkN_{y_\ttr}(t, z)=\clN(t, z+y_\ttr)-\clN(t, y_\ttr)$. Let~$\fkG\subset\rmD(A)$ be as in Assumption~\ref{A:N}.
 \begin{lemma}[{\cite[Prop.~3.5]{Rod20-eect}; cf.~\cite[Lem.~3.4]{Rod21-jnls}}]\label{L:NN}
 Let Assumptions~\ref{A:A0sp}, \ref{A:A0cdc}, and~\ref{A:N} hold true, and let~$P\in\clL(H)$. Then, 
there exists a constant $\overline C_{\clN1}=\ovlineC{n,\frac{1}{1-\|\delta_{2}\|},C_\clN,\norm{P}{\clL(H)}}>0$
such that:
  for all~$\widehat\gamma_0>0$, all~$t>0$, 
  and all~$(y_1,y_2)\in \fkG\times \fkG$, 
  we have
\begin{align}
 &2\Bigl( P\left(\clN(t,y_1)-\clN(t,y_2)\right),A(y_1-y_2)\Bigr)_{H}\label{NNyAy}\\
 &\le \widehat\gamma_0 \norm{y_1-y_2}{\rmD(A)}^{2}
  +\!\left(\!1+\widehat\gamma_0^{-\frac{1+\|\delta_2\|}{1-\|\delta_2\|} }\right)
 \!\overline C_{\clN1}\sum\limits_{j=1}^n\norm{y_1-y_2}{V}^{\frac{2\delta_{1j}}{1-\delta_{2j}}}
 \sum\limits_{k=1}^2\norm{y_k}{V}^\frac{2\zeta_{1j}}{1-\delta_{2j}}
 \norm{y_k}{\rmD(A)}^\frac{2\zeta_{2j}}{1-\delta_{2j}}.\notag
\end{align}
 \end{lemma}

\begin{lemma}[{\cite[Prop.~3.7]{Rod21-jnls}}]\label{L:fkN}
Let Assumptions~\ref{A:A0sp}, \ref{A:A0cdc}, \ref{A:N}, and~\ref{A:realy} hold true. Then, 
there exist constants $\widetilde C_{\fkN1}>0$, and~$\widetilde C_{\fkN2}>0$
such that: with~$y_{\ttr\fkG}\coloneqq P_\fkG^{\fkG^{\perp\rmD(A)}}y_\ttr$,
  for all~$\widehat\gamma_0>0$, all~$t>0$, 
  all~$(z_1,z_2)\in \fkG\times \fkG$, we have
\begin{align}
 &2\Bigl( \fkN_{y_\ttr}(t,z_1)-\fkN_{y_\ttr}(t,z_2),A(z_1-z_2)\Bigr)_{H}
 \le \widehat\gamma_0 \norm{z_1-z_2}{\rmD(A)}^{2}\label{fkNyAy}\\
 &\hspace*{1em}
   +\!\left(\!1+\widehat\gamma_0^{-\frac{1+\|\delta_2\|}{1-\|\delta_2\|} }\right)
 \!\widetilde C_{\fkN1}\sum\limits_{j=1}^n\norm{z_1-z_2}{V}^{\frac{2\delta_{1j}}{1-\delta_{2j}}}
 \sum\limits_{k=1}^2\norm{y_{\ttr\fkG}+z_k}{V}^\frac{2\zeta_{1j}}{1-\delta_{2j}}
 \norm{y_{\ttr\fkG}+z_k}{\rmD(A)}^\frac{2\zeta_{2j}}{1-\delta_{2j}}.\notag\\
&2\Bigl( \fkN_{y_\ttr}(t,z_1),Az_1\Bigr)_{H}
 \le \widehat\gamma_0 \norm{z_1}{\rmD(A)}^{2}\label{fkNyAy0}\\
 &\hspace*{.0em}  +\widetilde C_{\fkN2}\left(1+\widehat\gamma_0^{-\chi_5 }\right)\!
\left(1+\widehat\gamma_0^{-\frac{(\chi_5+1)\chi_2\chi_4}2  }\right)\!
   \left(1+\norm{y_{\ttr\fkG}}{V}^{\chi_1} \right)\!
  \left(1+\norm{y_{\ttr\fkG}}{\rmD(A)}^{\chi_2}\right)\!
  \left(1+\norm{z_1}{V}^{\chi_3}\right)\!
  \norm{z_1}{V}^2,\notag
  \end{align}
 with~$\widetilde C_{\fkN2}
  =\ovlineC{n,\widetilde C_{\clN1},\dnorm{\zeta_{1}}{},
  \dnorm{\zeta_{2}}{},\frac{1}{1-\dnorm{\delta_{2}}{}},\frac{1}{1-\dnorm{\zeta_{2}+\delta_{2}}{}}}$ and
 \begin{align}
 &\chi_1\coloneqq \tfrac{2\dnorm{\zeta_{1}}{}}{1-\dnorm{\delta_{2}+\zeta_{2}}{}}\ge0,
  \quad &\chi_2&\coloneqq\dnorm{\tfrac{2\zeta_{2}}{1-\delta_{2}}}{}\in[0,2),&&\label{chi12}\\
  &\chi_3\coloneqq
  \tfrac{2\dnorm{\delta_{1}+\delta_{2}+\zeta_{1}+\zeta_{2}}{}-2}{1-\dnorm{\delta_{2}+\zeta_{2}}{}}\ge0,
  \quad&\chi_4&\coloneqq\tfrac{1}{1-\dnorm{\delta_2+\zeta_2}{}  }>1,
  \quad&\chi_5&\coloneqq\tfrac{1+\|\delta_2\|}{1-\|\delta_2\|}>1.\label{chi345}
 \end{align}
  \end{lemma}

\begin{remark}
With~$z_i\in\fkG$, $i\in\{1,2\}$, since~$\fkG$ is a closed subspace of~$\rmD(A)$ and $\clN(t, y_{\ttr})=\clN(t, y_{\ttr\fkG})$, we also find that~$\fkN_{y_\ttr}(t, z_i)=\clN(t, z_i+{y_\ttr})-\clN(t, {y_\ttr})=\fkN_{y_{\ttr\fkG}}(t, z_i)$. Thus, Assumptions~\ref{L:NN} and~\ref{L:fkN} follow as corollaries of those in~\cite[Prop.~3.7]{Rod21-jnls} shown for the case~$\fkG=\rmD(A)$.
\end{remark}
Finally, we recall auxiliary results used to analyze the stability of the error dynamics.
\begin{lemma}[{\cite[Prop.~3.10]{Rod21-jnls}}]\label{L:maxpoly1r}
 Let~$\eta_1>0$, $\eta_2>0$  and~$\fks\in(0,1)$. Then
 \[
 \max_{\tau\ge0}\{-\eta_1\tau+\eta_2\tau^\fks \}
 =(1-\fks)\fks^\frac{s}{1-s} \eta_2^\frac{1}{1-\fks}\eta_1^\frac{\fks}{\fks-1}.
 \]
\end{lemma}

 \begin{lemma}[{\cite[Prop.~3.11]{Rod21-jnls}}]\label{L:ode-stab0}
 Let~$T>0$, $C_h>0$,  $\fkr>1$, and~$h\in L^\fkr_{\rm loc}(\bbR_0,\bbR)$ satisfy
 \[\label{ode-h}
  \sup_{s\ge0}\norm{h}{L^\fkr((s,s+T),\bbR)}= C_h\le+\infty.
 \]
Let also, $\mu>0$, and~$\varrho>1$. Then,
for every scalar~$\overline\mu>0$ satisfying
\[\label{ode-condstab0}
 \overline\mu  \ge \max\left\{2\tfrac{\fkr-1}{\fkr}
 \left(\tfrac{C_h^\fkr}{\fkr\log(\varrho)}\right)^\frac{1}{\fkr-1}, 2\mu\right\}+T^\frac{-1}{\fkr}C_h,
\]
we have that the scalar~{\sc ode} system
\[\label{ode-nonl-p0}
 \dot v=-(\overline\mu-\norm{h}{\bbR})v,\quad v(0)=v_0,
\]
is exponentially stable with rate $-\mu$ and transient bound~$\varrho$. For every~$v_0\in\bbR$, 
\[\label{ode-v-stab}
 \norm{v(t)}{}=\varrho\ex^{-\mu(t-s)}\norm{v(s)}{},\quad t\ge s\ge 0,\quad v(0)=v_0.
\]
\end{lemma}

 \begin{lemma}[{\cite[Prop.~3.12]{Rod21-jnls}}]\label{L:ode-stab}
 Let~$T>0$, $C_h>0$,   $\fkr>1$, and~$h\in L^\fkr_{\rm loc}(\bbR_0,\bbR)$ satisfying~\eqref{ode-h}.
Let also $R>0$, $p>0$, $\mu>0$, $\varrho>1$, and~$c>1$. Then the  scalar~{\sc ode}
\begin{equation}\label{ode-nonl}
 \dot\varpi=-(\overline\mu-\norm{h}{\bbR}(1+\norm{\varpi}{\bbR}^p))\varpi,\quad \varpi(0)=\varpi_0,
\end{equation}
is exponentially stable with transient bound~$\varrho$ and rate~$-\mu_0<-\mu$ as
\begin{equation}\notag
 \mu_0\coloneqq\max\left\{\mu,\tfrac{\log(2)}{pT},
 \left(\tfrac{\varrho^{2p+1}R^{p}C_h}{\varrho^\frac12-1}\right)^{\frac{\fkr}{\fkr-1}}
 \left(\tfrac{\fkr-1}{\fkr}\right)2^\frac{1}{\fkr-1}
 ,2^\frac{\fkr+1}{\fkr-1} \left(
  \varrho^{2p+\frac12} C_h\tfrac{p+1}{p}R^pc\right)^\frac{\fkr}{\fkr-1} p^\frac{1}{\fkr-1}\right\},
\end{equation}
if 
\begin{align}\notag
&\norm{\varpi_0}{}\le R\quad\mbox{and}\quad\overline\mu 
\ge \overline\mu_*\coloneqq \max\left\{2\tfrac{\fkr-1}{\fkr}
\left(\tfrac{2C_h^\fkr}{\fkr\log(\varrho)}\right)^\frac{1}{\fkr-1}, 4\mu_0\right\}+T^\frac{-1}{\fkr}C_h.
\end{align}
That is, the solution satisfies
\[\label{ode-stab}
 \norm{\varpi(t)}{\bbR}\le\varrho\ex^{-\mu_0 (t-s)}\norm{\varpi(s)}{\bbR},\quad\mbox{for all}\quad
 t\ge s\ge 0,\quad\mbox{if}\quad\norm{\varpi_0}{}<R.
\]
\end{lemma}

%%%%%%%%%%%%%%%%%%%%%%%%%%%%%%%%%%
%%%%%%%%%%%%%%%%%%%%%%%%%%%%%%%%%%
\subsection{Proof of the main Theorem~\ref{T:main}}\label{sS:proofT:main}
Observe that, from~\eqref{sys-z}, we obtain
\begin{align} 
 \tfrac{\ed}{\ed t}\norm{z}{V}^2
 &=2\left(-  Az
  -A_{\rm rc}z-\fkN_{y_\ttr}(z)
  -\lambda A^{-1}\overline\fkI_{S}\clZ_Sz,Az\right)_H.\label{dtz1}
 \end{align}
 Using Assumptions~\ref{A:A0sp}--\ref{A:A1}
 and the Young inequality, we find for
 all~$\gamma_1>0$,
  \begin{align} 
 &2\left(-  Az
  -A_{\rm rc}z,Az\right)_H
 \le-(2-\gamma_1)\norm{z}{\rmD(A)}^2+\gamma_1^{-1}C_{\rm rc}^2\norm{z}{V}^2
\label{thetas-est1}
  \end{align}
 and, due to Assumption~\ref{A:Inj},
 \begin{align} 
 2\left(-\lambda A^{-1}\overline\fkI_{S}\clZ_Sz,Az\right)_H
 &\le -\lambda\norm{\clZ_Sz}{\bbR^{M_\sigma}}^2.
 \label{thetas-est2}
  \end{align}
 
For the nonlinear term, proceeding as in~\cite[sect.~3.2]{Rod21-jnls}, 
using~\eqref{fkNyAy0} and the Young inequality we find that for all~$\gamma_2\in\bbR_0$,  with~$y_{\ttr\fkG}\coloneqq P_\fkG^{\fkG^{\perp\rmD(A)}}y_\ttr$,
\begin{subequations}\label{thetas-est3}
 \begin{align} 
 &2\Bigl( \fkN_{y_\ttr}(t,z),Az\Bigr)_{H}\le \gamma_2 \norm{z}{\rmD(A)}^{2}+\widehat C
   \Psi(y_{\ttr\fkG})
  \left(1+\norm{z}{V}^{\chi_3}\right)
  \norm{z}{V}^2,\\
  \mbox{with}\quad&\widehat C=\ovlineC{n,\widetilde C_{\clN1},\dnorm{\zeta_{1}}{},
  \dnorm{\zeta_{2}}{},\frac{1}{1-\dnorm{\delta_{2}}{}},
  \frac{1}{1-\dnorm{\zeta_{2}+\delta_{2}}{}},\frac{1}{\gamma_3}},\\
 \mbox{and}\quad&\Psi(y_{\ttr\fkG})\coloneqq\left(1+\norm{y_{\ttr\fkG}}{V}^{\chi_1} \right)
  \left(1+\norm{y_{\ttr\fkG}}{\rmD(A)}^{\chi_2}\right).
   \end{align}
\end{subequations}

Combining~\eqref{dtz1}, \eqref{thetas-est1}, \eqref{thetas-est2}, and~\eqref{thetas-est3}, 
 it follows that
 \begin{subequations}\label{dtz2}
 \begin{align} 
 \tfrac{\ed}{\ed t}\norm{z}{V}^2
 &\le -\left(2-\gamma_1- \gamma_2\right)\norm{z}{\rmD(A)}^2
 +2C_{\rm rc}^2\norm{z}{V}^2
 -\lambda\norm{\clZ_Sz}{\bbR^{M_\sigma}}^2
 \notag\\
 &\quad+\widehat C   \Psi(y_{\ttr\fkG})  \left(1+\norm{z}{V}^{\chi_3}\right)  \norm{z}{V}^2,
\end{align}
\end{subequations}

Next, choosing~$(\gamma_1,\gamma_2)=(\tfrac{1}{2},\tfrac{1}{2})$, and using Lemma ~\ref{L:MlamPoinc} with
~$\xi\coloneqq\overline\mu + 2C_{\rm rc}^2$, we obtain
\begin{equation}
-\left(2-\gamma_1- \gamma_2\right)\norm{z}{\rmD(A)}^2
 +2C_{\rm rc}^2\norm{z}{V}^2
 -\lambda\norm{\clZ_Sz}{\bbR^{M_\sigma}}^2\le-\overline\mu\norm{z}{V}^2,\notag
\end{equation}
for all~$S\ge S_*=\ovlineC{\overline\mu,C_{\rm rc}}\in\bbN_+$ and for all~$\lambda\ge0$ such that~$\lambda\ge\lambda_*(S)=\ovlineC{\overline\mu,C_{\rm rc}, C^P_S}$ with~$ C^P_S$  as in ~\eqref{normOP}. Hence, we can write
 \begin{subequations}\label{dtz5}
 \begin{align}
 &\tfrac{\ed}{\ed t}\norm{z}{V}^2
 \le -\Bigl(\overline\mu
 -\norm{h(y_{\ttr\fkG})}{}\left(1+\norm{z}{V}^{\chi_3}\right)\Bigr)\norm{z}{V}^2,\qquad \norm{z(0)}{V}^2=\norm{z_0}{V}^2,\\
 \intertext{with}
&\norm{h(y_{\ttr\fkG})}{}=h(y_{\ttr\fkG})\coloneqq\widehat C   \Psi(y_{\ttr\fkG})\in L^{\fkr}_{\rm loc}(\bbR_0,\bbR),
\qquad \fkr\coloneqq\tfrac{2}{\chi_2}>1,\\
&\norm{h(y_{\ttr\fkG})}{L^{\fkr}((s,s+\tau_{y_\ttr}),\bbR)}= \widehat C\norm{\Psi(y_{\ttr\fkG})}{L^{\fkr}((s,s+\tau_{y_\ttr}),\bbR)}\notag\\
&\hspace*{4em}\le \widehat C\norm{1+\norm{y_{\ttr\fkG}}{V}^{\chi_1} }{L^{\infty}((s,s+\tau_{y_\ttr}),\bbR)}
  \norm{1+\norm{y_{\ttr\fkG}}{\rmD(A)}^{\chi_2}}{L^{\fkr}((s,s+\tau_{y_\ttr}),\bbR)}\notag\\
  &\hspace*{4em}\le \widehat C(1+C_{y_\ttr}^{\chi_1})
  \left(\tau_{y_\ttr}^\frac{1}{\fkr}+\norm{y_{\ttr\fkG}}{L^{2}((s,s+\tau_{y_\ttr}),\rmD(A))}^\frac{2}{\fkr}\right)\eqqcolon C_h,
 \end{align}
 \end{subequations}
 and see that~$\varpi\coloneqq\norm{z}{V}^2$ satisfies system~\eqref{ode-nonl},
 with~$h=h(y_{\ttr\fkG})$ and~$p=\chi_3\ge0$.
 
 Let us be given arbitrary~$\varrho>1$ and~$\mu>0$.
 In the case~$p>0$, we use
 Lemma~\ref{L:ode-stab} to conclude that the norm satisfies
 \begin{equation}\label{final.expp}
 \norm{z(t)}{V}^2\le\varrho\ex^{-\mu (t-s)}\norm{z(s)}{V}^2,\quad\mbox{for}\quad
 t\ge s\ge 0,\quad\mbox{and}\quad\norm{z(0)}{V}^2<R, \qquad p>0,
  \end{equation}
provided we take~$\overline\mu$ large enough.

In the case~$p=0$, we use
 Lemma~\ref{L:ode-stab0} to conclude that the norm satisfies
  \begin{equation}\label{final.expp0}
 \norm{z(t)}{V}^2\le\varrho\ex^{-\mu (t-s)}\norm{z(s)}{V}^2,\quad\mbox{for}\quad
 t\ge s\ge 0,\quad\mbox{and}\quad z(0)\in V,  \qquad p=0,
  \end{equation}
provided we take~$\overline\mu$ large enough.

In particular~\eqref{final.expp} actually holds for all~$p\ge0$: we have that
 \begin{equation}\label{final.exp}
 \norm{z(t)}{V}^2\le\varrho\ex^{-\mu (t-s)}\norm{z(s)}{V}^2,\quad\mbox{for all}\quad
 t\ge s\ge 0,\quad\mbox{and all}\quad\norm{z(0)}{V}^2<R,
 \end{equation}
provided we take~$\overline\mu$ large enough.
That is, provided we take a large enough~$S$
and a large enough~$\lambda$.
Finally, note that from Lemma~\ref{L:ode-stab},  it is enough
to take
 \begin{equation}\notag
\overline\mu=
\ovlineC{b},\qquad\mbox{with}\quad b\coloneqq(\mu,\tfrac1{\tau_{y_\ttr}},\varrho,\tfrac1{\varrho^\frac12-1},\tfrac{\fkr+1}{\fkr-1},R,C_h,\chi_).
 \end{equation}
For that, using~\eqref{dtz5}, it is enough to choose, firstly~$S\ge S^*$ with~$S_*$ in the
form
\[S_*=\ovlineC{a},\qquad\mbox{with}\quad a=(R,\mu,\varrho,\tfrac1{\varrho^\frac12-1},\tfrac1{\tau_{y_\ttr}},\tau_{y_\ttr},\chi_1,\tfrac{2+\chi_2}{2-\chi_2},\chi_3,
C_{\rm rc},C_{y_\ttr})\]
and subsequently~$\lambda\ge\lambda^*(S)$ with~$\lambda^*(S)$ in the
form~$\lambda^*(S)=\ovlineC{a; C^P_S}$.

We can finish the proof by recalling~\eqref{chi12} and~\eqref{chi345}.
\qed

\subsection{On the existence and uniqueness of solutions for the error}\label{sS:unique-error}
The existence of a solution can be proven as a weak limit of solutions of appropriate finite-dimensional Galerkin approximations. In fact, we can follow the arguments in~\cite[sect.~3.4]{Rod21-jnls}; with the exception of the output injection operator, which in~\cite{Rod21-jnls} is based on sensors in~$H$, while in this manuscript the sensors are taken in a larger space~$V'\supseteq H$. So, let us consider a Galerkin approximation as
\begin{subequations}\label{sys-error-Gal}
\begin{align} 
 &\dot{z}^N +Az^N+P_{\clE^\rmf_N} A_{\rm rc}(t)z^N +P_{\clE^\rmf_N}\fkN_{y_\ttr}(t,z^N)
 =-\lambda P_{\clE^\rmf_N}A^{-1}\overline\fkI_{S}\clZ_{S} z^N,\quad t\ge 0,\label{sys-error-dyn-Gal}\\
 &z^N(0)=P_{\clE^\rmf_N}z_0\in V,\label{sys-error-ic-Gal}
\end{align}
\end{subequations}
 where~$P_{\clE^\rmf_N}\in\clL(H)$ is the orthogonal projection in~$H$ onto the
 space~$\clE^\rmf_N\coloneqq\linspan\{e_n\mid 1\le n\le N\}$ spanned by the first
eigenfunctions of~$A$.

Proceeding as in~\cite[sect.~3.4]{Rod21-jnls} we can conclude that for large enough~$S$ and~$\lambda$ we will have that there exists a weak limit
\begin{equation}\label{st-reg-error}
z^\infty\in L^2((0,s),\rmD(A))\quad\mbox{with}\quad\dot z^\infty\in L^2((0,s),H),
\end{equation}
 for an arbitrary fixed~$s>0$, so that
\begin{equation}
 z^N\xrightharpoonup[L^2((0,s),\rmD(A))]{} z^\infty\qquad\mbox{and}\qquad \dot z^N
 \xrightharpoonup[L^2((0,s),H)]{} \dot z^\infty,
\end{equation}
and also that
\begin{align}
&Az^N+ A_{\rm rc}(t)z^N 
 \xrightharpoonup[L^2((0,s),H)]{} Az^\infty+ A_{\rm rc}(t)z^\infty ,
\\
& P_{E_N}\fkN_{y_\ttr}(t,z^N)\xrightarrow[L^2((0,s),H)]{} \fkN_{y_\ttr}(t,z^\infty).
\end{align}

Now, we simply observe that~$A^{-1}\overline\fkI_{S}\clZ_{S}\in\clL(V,H)$, to obtain
\begin{equation}
-\lambda A^{-1}\overline\fkI_{S}\clZ_{S} z^N
\xrightharpoonup[L^2((0,s),H)]{}  -\lambda A^{-1}\overline\fkI_{S}\clZ_{S} z^\infty.
\end{equation}

Then, we can follow the arguments in~\cite[sect.~3.4]{Rod21-jnls} to conclude that~$z^\infty$ is the unique solution for the error dynamics satisfying~\eqref{st-reg-error} for all~$s>0$.

%%%%%%%%%%%%%%%%%%%%%%%%%%%%%
%%%%%%%%%%%%%%%%%%%%%%%%%%%%%
\subsection{On the existence and uniqueness of solutions for systems~\eqref{sys-y} and~\eqref{sys-haty}}
\label{sS:unique-yhay}
The solution~$y_\ttr$ for system~\eqref{sys-y} is assumed to exist in Assumption~\ref{A:realy}. Proceeding as in~\cite[Sect.~4.3]{Rod20-eect}, due to the regularity in Assumption~\ref{A:realy},
we can show that such solution~$y_\ttr$ is unique.
From Section~\ref{sS:unique-error} the solution~$z$, given by Theorem~\ref{T:main} for the
error dynamics, is also unique. Consequently,
the solution~$y_\tte=y_\ttr+z$ for~\eqref{sys-haty} exists and is unique as well.

%%%%%%%%%%%%%%%%%%%%%%%%%%%%%
%%%%%%%%%%%%%%%%%%%%%%%%%%%%%
 \section{The Kuramoto--Sivashinsky model for flame propagation}\label{S:KS-satisfAssumpt}
We apply our abstract result to the Kuramoto--Sivashinsky system~\eqref{sys-KS-flame}. Given a Luenberger observer as~\eqref{sys-KS-flame-obs}, we obtain the error dynamics for the error~$z=y_\tte-y_\ttr$ as
\begin{subequations}\label{sys-KS-flame-error1}
  \begin{align}
 &\tfrac{\p}{\p t}z+\nu_2 \Delta^2z+\nu_1 \Delta z+\nu_0\tfrac{1}{2}\norm{\nabla y_\tte}{\bbR^d}^2-\nu_0\tfrac{1}{2}\norm{\nabla y_\ttr}{\bbR^d}^2=\fkI_{S}^{[\lambda,\Lambda]}\clZ_{S}z,\qquad\clG y_\tte\rest{\p\Omega}=0, \\ & z(0,x)= z_0\coloneqq y_\tte(0,x)-y_\ttr(0,x).
\end{align}
\end{subequations}
Hence, we introduce the operators
\begin{subequations}\label{KS-abstract}
\begin{align}
Az&\coloneqq \nu_2 \Delta^2z-2\nu_2\Delta z+\nu_2z;\qquad A_{\rm rc}z\coloneqq (\nu_1 +2\nu_2) \Delta z-\nu_2z; \\
\clN(y)&\coloneqq \tfrac{1}{2}\nu_0\norm{\nabla  y}{\bbR^d}^2;\qquad
 \fkN_{y_\ttr}(z)\coloneqq\nu_0(\nabla y_\ttr,\nabla z)_{\bbR^d}+\tfrac{1}{2}\nu_0\norm{\nabla  z}{\bbR^d}^2.
\end{align}
\end{subequations}
A simple observation gives us
\begin{subequations}\label{KS-abst-conc}
 \begin{align}
Az+A_{\rm rc}z&=\nu_2 \Delta^2z+\nu_1\Delta z,\\
\fkN_{y_\ttr}(z)&=\nu_0\tfrac{1}{2}\norm{\nabla y_\tte}{\bbR^d}^2-\nu_0\tfrac{1}{2}\norm{\nabla y_\ttr}{\bbR^d}^2,
\end{align}
\end{subequations}
hence
the above operators allow us to write system~\eqref{sys-KS-flame-error1} as system~\eqref{sys-z}. Thus it is enough to check our assumptions  to be able to apply the abstract result in Theorem~\ref{T:main} and conclude the Main Result stated in the Introduction.

For simplicity, we restrict ourselves to periodic boundary conditions, which are often considered for the Kuramoto--Sivashinsky model. For other boundary conditions we refer the reader to the discussion in~\cite[sect.~2.1]{RodSeifu22-arx}. For periodic boundary conditions, our spatial domain is the $d$-dimensional torus~$\bbT_L^d\coloneqq\bigtimes\limits_{i=1}^d\frac{L_i}{2\pi}\bbT^1\sim\bigtimes\limits_{i=1}^d[0,L_i)$, where~$\bbT^1=\bbT^1_{2\pi}\sim\{(r_1,r_2)\in\bbR^2\mid r_1^2+r_2^2=1\}\sim[0,2\pi)$ is the one-dimensional torus.

%%%%%%%%%%%%%%%%%%%%%%%%%%%%%%%%%%%%%
\subsection{Satisfiability of Assumptions~\ref{A:A0sp} and~\ref{A:A0cdc}}
The satisfiability of Assumptions~\ref{A:A0sp} and~\ref{A:A0cdc} has been shown in~\cite[Lems.~2.4 and~2.5]{RodSeifu22-arx}. We recall here the abstract setting.
It is convenient to use the equivalence of the norms of~$\rmD(A)$ and~$V$ to the norms of standard Sobolev spaces~$W^{k,2}(\bbT_L^d)$.
As usual we take the Lebesgue space~$H\coloneqq L^2(\bbT_L^d)$ as pivot space~$H=H'$. Let us
denote~$\clW\coloneqq W^{1,2}(\bbT_L^d)$, and consider the shifted Laplacian operator~$\clA\coloneqq-\Delta+\Id$,
\begin{align}
-\Delta+\Id\in\clL(\clW,\clW'),\qquad \langle(-\Delta+\Id) h,w\rangle_{\clW',\clW}\coloneqq(\nabla h,\nabla w)_{(H)^d}+(h,w)_H.\notag
\end{align}
We observe that
\begin{align}
h\in W^{1,2}(\bbT_L^d)\quad&\Longleftrightarrow\quad (-\Delta+\Id)^\frac12 h\in L^2(\bbT_L^d)=H;\notag\\
\mbox{hence}\quad h\in W^{m,2}(\bbT_L^d)\quad&\Longleftrightarrow\quad ((-\Delta+\Id)^\frac12)^{(m)} h \in L^2(\bbT_L^d);\qquad m\in\bbN_+.\notag
\end{align}
In particular, with~$V= W^{2,2}(\bbT_L^d)$ and~$A\in\clL(V,V')$ as 
\begin{equation}\label{AVV'-KS}
\langle Av,w\rangle_{V',V}\coloneqq \nu_2 (\Delta v,\Delta w)_H+2\nu_2(\nabla v,\nabla w)_H +\nu_2(v,w)_H,
\end{equation}
We observe that in this way, the operator~$A$
defines a scalar product
\[
(v,w)_V\coloneqq\langle Av,w\rangle_{V',V},
\]
the norm of which,
$
\norm{v}{V}\coloneqq ((v,v)_V)^\frac12,
$
 is equivalent to the usual norm of the Sobolev  space~$W^{2,2}(\bbT_L^d)$.
Therefore,  Assumptions~\ref{A:A0sp}--\ref{A:A0cdc} are satisfied.

Furthermore, we can write
\[
A=\nu_2(\Delta^2 -2\Delta+\Id)=\nu_2(-\Delta+\Id)^2=\nu_2\clA^2,
\]
and observe that
\begin{equation}\label{Vnorm}
(v,w)_V=\nu_2 ((-\Delta+\Id) v,(-\Delta+\Id) w)_H=\nu_2(\clA v,\clA w)_H,
\end{equation}
and also that~$A=\nu_2\clA^2=\nu_2((-\Delta+\Id)^\frac12)^4$, from which we can find
\begin{equation}\notag
\rmD(A)=W^{4,2}(\bbT_L^d),
\end{equation}
because for~$h\in\rmD(A)$ we have that~$Ah\in H$. Thus~$\rmD(A)\subseteq W^{4,2}(\bbT_L^d)$. The reverse inclusion is clear. Now, with~$(v,w)_{\rmD(A)}\coloneqq(Av,Aw)_{H}$, the equivalence of the norms~$\norm{\Bigcdot}{\rmD(A)}$ and~$\norm{\Bigcdot}{W^{4,2}(\bbT_L^d)}$ also follows straightforwardly.

%%%%%%%%%%%%%%%%%%%%%%%%%%%%%%%%%%%%%
\subsection{On the satisfiability of Assumption~\ref{A:realy}}\label{sS:assumpt_aux}
The average of solutions of~\eqref{sys-KS-flame} will be strictly decreasing unless the gradient vanishes. Furthermore its average will diverge to~$-\infty$ if the (norm of the) gradient remains away from zero. Thus we can expect that the solutions are not necessarily bounded in~$V$. In the literature we can find results addressing settings where the gradient remains bounded; here noticing that the evolution of~$\nabla y_\ttr$ satisfies the dynamics of the Kuramoto--Sivashinsky model for fluid flow, we can use the available estimates for this model; for results on (or related to) the boundedness of~$\nabla y_\ttr$ we refer the reader to~\cite{NicolaenkoSchTem85,Otto09,GoluskinFantuzzi19,GiacomelliOtto05} for the  case of one-dimensional spatial domains~$\Omega\in\bbR^1$, and to~\cite{FengMazzucato21} for higher-dimensional spatial domains, $\Omega\in\bbR^d$ with~$d\in\{2,3\}$. Due to such results we can (and due to the variety of necessary conditions in such results we shall) assume an analogue of Assumption~\ref{A:realy}, namely, Assumption~\ref{A:realy-KS} below.

Firstly, note that the constant function~$\indf_{\bbT_L^d}$ is an eigenfunction of the operator~$A$. By setting the orthogonal complement~$(\bbR\indf_{\bbT_L^d})^\perp$ in~$H$ and defining~$\fkG\coloneqq\rmD(A)\bigcap (\bbR\indf_{\bbT_L^d})^\perp$ we have that~$\fkG$ is a closed subspace of~$\rmD(A)$. Furthermore,~$P_{\fkG}^{\fkG^{\perp\rmD(A)}}y=P_{(\bbR\indf_{\bbT_L^d})^\perp}^{\bbR\indf_{\bbT_L^d}}y$ for all~$y\in\rmD(A)$. Let us denote again~$y_{\ttr\fkG}= P_{\fkG}^{\fkG^{\perp\rmD(A)}}y_\ttr$.
\begin{assumption}\label{A:realy-KS}
Let~$\fkG=\rmD(A)\bigcap (\bbR\indf_{\bbT_T^d})^\perp$, Then, the targeted real state~$y_\ttr$, solving~\eqref{sys-KS-flame}, is persistently uniformly bounded as follows. There are constants~$C_{y_\ttr}\ge0$ and~$\tau_{y_\ttr}>0$ such that
\begin{align}
\sup_{s\ge0}\norm{y_{\ttr\fkG}(s)}{V}\le C_{y_\ttr}\quad\mbox{and}\quad\sup_{s\ge0}\norm{y_{\ttr\fkG}}{L^2((s,s+\tau_{y_\ttr}),\rmD(A))}<C_{y_\ttr}.\notag
\end{align}
\end{assumption}
Note that~$y_{\ttr\fkG}= y_{\ttr}-y_{\ttr\rm av}$ where~$y_{\ttr\rm av}=P_{\bbR\indf_{\bbT_L^d}}^{(\bbR\indf_{\bbT_L^d})^\perp}y_\ttr=\tfrac{\int_{\bbT_L^d} y_\ttr\,\rmd x}{\int_{\bbT_L^d} 1\,\rmd x}$ is the orthogonal projection in~$H$ onto~$\bbR\indf_{\bbT_L^d}$. Then~$\norm{\nabla y_\ttr(s)}{W^{1,2}({\bbT_L^d})^d)}$ is a norm equivalent to~$\norm{y_\ttr(s)}{V}$ in~$V\bigcap(\bbR\indf_{\bbT_L^d})^\perp$, and~$\norm{\nabla y_\ttr(s)}{W^{3,2}({\bbT_L^d})^d)}$ is a norm equivalent to~$\norm{y_\ttr(s)}{\rmD(A)}$ in~$\fkG$.

\begin{remark}\label{R:satisf-pers-bdd}.
We do not know whether Assumption~\ref{A:realy-KS} is satisfied for all~$\nu=(\nu_2,\nu_1,\nu_0)$, and all regular enough external force~$f$ and initial state~$y_{\ttr0}=y_\ttr(0,\Bigcdot)$ defined in our periodic domain~$\Omega=\bbT_L^d$. However, there are related results in the literature concerning data~$(\nu,L,f,y_{\ttr0})$ given in a particular set. See, for example, for~$d=1$, in~\cite[Thm.~2.1]{NicolaenkoSchTem85} a result on the boundedness of
the norm~$\sup_{s\ge0}\norm{\tfrac{\p}{\p x}y_{\ttr}(s)}{L^2(\Omega)}$ and~\cite[Thm.~3.5]{NicolaenkoSchTem85} for a result concerning a bound analogous to the one required in Assumption~\ref{A:realy-KS}. 
\end{remark}

%%%%%%%%%%%%%%%%%%%%%%%%%%%%%%%%%%%%%
\subsection{Satisfiability of Assumptions~\ref{A:A1} and~\ref{A:N}}\label{sS:assumpt-chkA1N}
From
\begin{equation}\notag
\norm{A_{\rm rc}z}{H}\le(\nu_1 +2\nu_2)(\norm{\Delta z}{H}+\norm{z}{H})\le (\nu_1 +2\nu_2)2^\frac12(\norm{\Delta z}{H}^2+\norm{z}{H}^2)^\frac12\le (\nu_1 +2\nu_2)(\tfrac{2}{\nu_2})^\frac12\norm{z}{V} 
\end{equation}
we have that Assumption~\ref{A:A1} holds with~$C_{\rm rc}\coloneqq(\nu_1 +2\nu_2)(\tfrac{2}{\nu_2})^\frac12$.

Next, recalling~\eqref{KS-abstract}, we can see that~$\clN(t,y)=\clN(t,y_{\ttr\fkG})$, because $\nabla y=\nabla y_{\ttr\fkG}$, with~$\fkG$ and~$y_{\ttr\fkG}= P_{\fkG}^{\fkG^{\perp\rmD(A)}}y$ as in section~\ref{sS:assumpt_aux}. Further, with~$d=y_2-y_1$, we find that 
\begin{align}
&2\nu_0^{-1}(\clN(t,y_1)-\clN(t,y_2))= (\nabla y_1,\nabla d)_{\bbR^d}+ (\nabla y_2,\nabla d)_{\bbR^d},
\notag
\end{align}
from which we obtain
\begin{align}
&2\nu_0^{-1}\norm{\clN(t,y_1)-\clN(t,y_2)}{H}\le\norm{\nabla y_1}{L^3(\bbT_L^d)}\norm{\nabla d}{L^6(\bbT_L^d)}
+\norm{\nabla y_2}{L^3(\bbT_L^d)}\norm{\nabla d}{L^6(\bbT_L^d)}
\notag
\end{align}
and, since~$d\in\{1,2,3\}$, we can use the Sobolev embeddings~$W^{1,2}\xhookrightarrow{}L^6\xhookrightarrow{}L^3$ to obtain
\begin{align}
2\nu_0^{-1}\norm{\clN(t,y_1)-\clN(t,y_2)}{H}&\le C_1(\norm{y_1}{W^{2,2}(\bbT_L^d)}+\norm{y_2}{W^{2,2}(\bbT_L^d)})\norm{d}{W^{2,2}(\bbT_L^d)}
\notag\\
&\le C_2(\norm{y_1}{V}+\norm{y_2}{V})\norm{d}{V}.
\notag
\end{align}
We conclude that Assumption~\ref{A:N} holds with~$n=1$ and~$(\zeta_{11},\zeta_{21},\delta_{11},\delta_{21})=(1,0,1,0)$.

%%%%%%%%%%%%%%%%%%%%%%%%%%%%%%%%%%%%%
\subsection{Satisfiability of Assumptions~\ref{A:sens} and~\ref{A:Poincare}}
To check Assumption~\ref{A:sens} we need to specify the set of sensors and auxiliary functions.

\subsubsection{The sensors and auxiliary functions}
We consider the case of point measurements giving us an output as~\eqref{sys-Intro-o}, taking an output operator~$\clZ_S\colon \rmD(A)\to\bbR^{S_\sigma\times1}$ as
\begin{subequations}\label{clZ=delta}
\begin{align}
 &\clZ_S=\deltafun_S\coloneqq\begin{bmatrix}\deltafun_{x^{S,1}}&\deltafun_{x^{S,2}}&\dots&\deltafun_{x^{S,S_\sigma}}\end{bmatrix}^\top,
\intertext{involving delta distributions centered at  spatial points~$x^{S,j}\in\bbT_L^d$,}
&\deltafun_{x^{S,j}}(z)\coloneqq z(x^{S,j}),\qquad 1\le j\le S_\sigma.
\end{align}
\end{subequations}

\begin{remark}
Note that, actually,~$\clZ_S$ is defined for~$z$ in the larger space~$V\supset\rmD(A)$, because~$V=W^{2,2}(\bbT_L^d)\xhookrightarrow{}\clC(\bbT_L^d,\bbR)$, since~$d\in\{1,2,3\}$ (e.g., see~\cite[sect.~4.5, Thm.~4.57]{DemengelDem12}). Observe that, though in theory, it is enough to take~$\clZ_S\in\clL(\rmD(A),\bbR^{S_\sigma\times1})\sim (\rmD(A)')^{S_\sigma}$ because strong solutions~$z$ will satisfy~$z(t)\in\rmD(A)$ for {\em almost every} $t>0$, taking such a~$\clZ_S$ may allow for the error output~$\clZ_Sz(t)\in \bbR^{S_\sigma\times1}$ to be not well defined at some time instants. For example, if the norm~$\norm{z(t)}{\rmD(A)}$ blows-up as~$t$ approaches a certain time instant~$\underline t>t$, then also the norm~$\norm{\clZ_Sz(t)}{\bbR^{S_\sigma\times1}}$ could blow-up. Therefore, for applications it may be convenient to take~$\clZ_S$ in the smaller space~$\clZ_S\in\clL(V,\bbR^{S_\sigma\times1})\sim (V')^{S_\sigma}$, because we will have~$z(t)\in V$ for {\em every} $t>0$ (for the same strong solutions) and consequently we will have the error output~$\clZ_Sz(t)\in \bbR^{S_\sigma\times1}$ defined for {\em every} $t>0$ 
\end{remark}
Next, we describe the placement of the sensors~$\deltafun_{x^{S,j}}$, that is, the placement of the points~$x^{S,j}$ in~$\Omega = \bbT_L^d\sim\bigtimes\limits_{n=1}^d[0,L_n)$, $L_n>0$.
Namely, for each $S\in\bbN_+$ and~$1\le j\le S_\sigma$, we construct points $x^{S,j}$  as follows,  motivated by the number of monomials of degree at most~$3$, defined in the rectangle~$\Omega$.
With~$x=(x_1,\dots,x_d)$ standing for a generic element in~$\Omega$, let us denote the set of those monomials as
\begin{align}\label{monomials}
\fkM_d\coloneqq\{\fkm^\kappa(x)\coloneqq x_1^{\kappa_1}\dots x_d^{\kappa_d}\mid\kappa\in\{0,1,2,3\}^d\mbox{ and }\kappa_1+\dots+\kappa_d\le3\},
\end{align}
 which form a basis for the space of polynomials~$\fkP_d\coloneqq\linspan\fkM_d$ of degree at most~$3$, and their number is given by~$\#\fkM_d=\frac{(d+3)(d+2)(d+1)}{6}$ (cf.~Remark~\ref{R:balls-and-cells}). That is,~$\#\fkM_1=4$, $\#\fkM_2=10$, and~$\#\fkM_3=20$.
\begin{enumerate}[label=\roman*),leftmargin=*,align=left]
\item 
For $S=1$, we choose $\#\fkM_d$ points such that not all monomials vanish at those points. Let us denote the set of those points as
\begin{subequations}\label{sensors}
\begin{align}\label{sensorsS=1}
    &\fkX_1\coloneqq\{x^{1,j}\mid 1\le j\le \#\fkM_d\}\subset\Omega;
    \intertext{satisfying:}
    &\mbox{If }g\in\fkP_d,\mbox{ then}\quad g(\fkX_1)=\{0\}\;\Longleftrightarrow\; g=0.\label{sensorsS=1-g}
\end{align}

\item For $S>1$, we partition $\Omega$ into $S^d$ rescaled copies $\Omega^{S,k},1\le k\le S^d$, of itself as
\begin{align}
&\Omega^{S,k}=v^k+\tfrac{1}{S}\Omega,\quad\mbox{with}\\
&v^k\in\left\{(v_1,\ldots,v_d)\in\bbR^d\mid v_n=(i-1)\tfrac{L_n}{S},\;
1\le i\le S,\; 1\le n\le d\right\}.
\end{align}

 In each copy we select $\#\fkM_d$ points as
\begin{align}\label{sensorsS>1}
    &x^{S,\#\fkM_d(k-1)+s}=v^k+\tfrac{1}{S}x^{1,s}\in \Omega^{S,k},\quad 1\le s\le \#\fkM_d,
\end{align}
and take the set of delta sensors, and its linear span, as
\begin{align}\label{WS-constr}
    &W_S:=\{\deltafun_{x^{S,j}}|1\le j\le S_\sigma\};\qquad \clW_S:=\linspan W_S;\qquad S_\sigma=\sigma(S)\coloneqq \#\fkM_d S^d.
\end{align}
\end{subequations}
Note that,
\begin{align}
\Omega^{S,k}\setminus\{x^{S,\#\fkM_d(k-1)+s}|1\le s\le \#\fkM_d\} = v^k+\tfrac{1}{S}\Omega\setminus\{x^{1,s}|1\le s\le \#\fkM_d\}. 
\end{align}

\end{enumerate}

\begin{subequations}\label{aux}
Next, as set of auxiliary functions, and its linear span, we choose an arbitrary set
\begin{align}
    &\widetilde W_S\coloneqq\{\Psi^{S,j}|1\le j\le S_\sigma\}\subset\rmD(A),\qquad \widetilde\clW_S\coloneqq\linspan\widetilde W_S\\
    \intertext{of functions satisfying}
    &\Psi^{S,j}(x^{S,j})=1 \quad\mbox{and}\quad \Psi^{S,j}(x^{S,i})=0,\quad\mbox{for}\quad 1\le i\ne j\le S_\sigma.
\end{align}
\end{subequations}

\begin{remark}\label{R:choicefkX1}
In the case~$d=1$, an arbitrary set~$\fkX_1$ consisting of~$4$ distinct points in~$\Omega = \bbT_L^1\sim[0,L_1)$, $L_1>0$ will satisfy~\eqref{sensorsS=1-g}, simply because a polynomial in one variable with degree at most~$3$ can have at most~$3$ zeros. For~$d\in\{2,3\}$ the choice of~$\fkX_1$ is less clear. In applications we can check~\eqref{sensorsS=1-g} numerically, because it is equivalent to the fact that the matrix
\begin{align}
M_\fkM=\begin{bmatrix} \fkm^{\kappa^j}(x^{1,i})\end{bmatrix}\in S_\sigma\times S_\sigma=\#\fkM_d\times \#\fkM_d
\end{align}
has full rank. Here the entry~$\fkm^{\kappa^j}(x^{1,i})$ in the~$i$th row and~$j$th column is given by the evaluation of the monomial~$\fkm^{\kappa^j}\in\fkM_d$ at~$x^{1,i}\in\fkX_1$ and
the vector index~$\kappa^j$ runs over the set~$\Xi_3\coloneqq\{\kappa\in\{0,1,2,3\}^d\mid \kappa_1+\dots+\kappa_d\le3\}$.
Indeed, if we write~$g\in\fkP_d$ as~$g=\overline g_{(1,1)}\fkm^{\kappa^1}+\overline g_{(2,1)}\fkm^{\kappa^2}+\dots+\overline g_{(S_\sigma,1)}\fkm^{\kappa^{S_\sigma}}$ for a scalar vector
$\overline g\in\bbR^{S_\sigma\times1}$, then~$g(\fkX_1)=\{0\}$ if, and only if, $M_\fkM\overline g=0$.
\end{remark}

\begin{remark}\label{R:balls-and-cells}
 For~$d\in\{1,2,3\}$, we can find the value~$\#\fkM_d=\frac{(d+3)(d+2)(d+1)}{6}$ simply by writing down the monomials in~$\fkM_d$ spanning the space of polynomials in the~$d$ variabes~$x_1,\dots,x_d$ with degree at most~$3$. More generally, the number of monomials in~$m$ variabes~$x_1,\dots,x_{m}$ and with degree at most~$p$ is given by~$\#\fkM_m^{[\le p]}\coloneqq\frac{(m+p)!}{m!p!}$ where as usual~$(n+1)!\coloneqq (n+1)(n!)$ is the factorial of the positive integer~$n+1$, and~$0!\coloneqq 1$. Though this result seems to be well known, we could not find a direct statement of it in the published literature. Thus, we present the proof in the Appendix, section~\ref{Apx:balls-and-cells}. Note that, in this manuscript, it is enough for us to consider polynomials up to degree~$3$, but in different applications it may be necessary to consider polynomials up a higher degree (cf.~Rem.~\ref{R:SobolQuoc}), and it may useful to know the dimension of such subspace. 
\end{remark}

\subsubsection{Checking the assumptions}
Firstly, note that~$\sigma\colon S\mapsto S_\sigma=\#\fkM_d S^d$ is strictly increasing. Next, observe that we can decompose each $z\in\rmD(A)$ as $z=z_\Psi + z_\fkZ$ with
\begin{equation}\notag
    z_\Psi = \sum_{j=1}^{S_\sigma} z(x^{S,j})\Psi^{S,j}\in\widetilde \clW_S \qquad\mbox{and}\qquad z_\fkZ = z - z_\Psi,
\end{equation}
and we see that~$z_\fkZ(x^{S,j})=0$ for all $1\le j\le S_\sigma$; due to~\eqref{aux}. Hence,~$z_\fkZ\in\fkZ_S\coloneqq\rmD(A)\bigcap \ker\clZ_S$ and~$\rmD(A)=\widetilde \clW_S+\fkZ_S$. Assumption~\ref{A:sens} follows, due to~$\widetilde \clW_S\bigcap\fkZ_S=\{0\}$. 
Indeed, for~$h\in\widetilde \clW_S\bigcap\fkZ_S$ we find~$h=h_\Psi=\sum_{j=1}^{S_\sigma} h(x^{S,j})\Psi^{S,j}$ and~$h(x^{S,j})=0$ for all~$1\le j\le S_\sigma$, which implies~$h=h_\Psi=0$.

Next, we can show that Assumption~\ref{A:Poincare} holds by an argument as in~\cite[sect.~5]{Rod21-sicon} and~\cite[sect.~4.2]{Rod21-jnls}. The key point is that we consider rescaled copies of the rectangular domain~$\Omega$ where each copy has~$\#\fkM_d$ sensors as in the case~$S=1$. 
Note that for~$S=1$, we have that the constant in Assumption~\ref{A:Poincare} is nonzero,~$\beta_1\ge\alpha_1>0$, where~$\alpha_1$ is the first eigenvalue of~$A$, because~$\norm{h}{\rmD(A)}^2=\norm{Ah}{H}^2\ge\alpha_1\norm{A^\frac12h}{H}^2=\norm{h}{V}^2$.
Then, we can show the divergence stated in Assumption~\ref{A:Poincare}, as~$S$ increases, by following the arguments in~\cite[sect.~4.2]{Rod21-jnls}. To follow those arguments, since we already have that the norms of~$\rmD(A)$ and~$V$ are equivalent to standard Sobolev norms,  it is enough to show (for~$S=1$) that  those Sobolev norms are  equivalent to the norm defined by the sum of the seminorm involving only the largest order derivatives and a seminorm which is a norm on the space of polynomials~$\fkP_d=\linspan\fkM_d$ of degree less or equal than~$3$.
Observe that the Sobolev norm~$h\mapsto\norm{h}{W^{4,2}(\bbT_L^d)}$ is equivalent to the norm given by
\[
h\mapsto\left(\norm{\nabla_x^4h}{L^{2}(\bbT_L^d)}^{2}+\eta(h)\right)^\frac12,\quad\mbox{with}\quad\eta(h)\coloneqq{\textstyle\sum\limits_{j=1}^{\#\fkM_d}\norm{h(x^{S,j})}{\bbR}^2}.
\]
This equivalence follows from the fact that~$g\mapsto\eta(g)^\frac12$ is a norm in the finite-dimensional space~$\fkP_d$. Indeed, it is clear that~$\eta$ defines a seminorm, further if~$g\in\fkP_d$ and~$\eta(g)=0$, then~$g=0$, due to the choice of~$\fkX_1$; see~\eqref{sensorsS=1-g}. 
 Therefore, we can conclude that Assumption~\ref{A:Poincare} holds true for this choice of~$\fkX_1$.

 \begin{remark}\label{R:SobolQuoc}
  In~\cite[sect.~4.2]{Rod21-jnls}, arguments are used to prove that a quotient as
\[
\inf\limits_{\Theta\in \clX_S\setminus\{0\}}\tfrac{\norm{\nabla_x^2\Theta}{L^{2}}^2}{\norm{\nabla_x^1\Theta}{L^{2}}^2},
\]
diverges to~$+\infty$ as~$S$ increases, for an appropriate space~$\clX_S\subset W^{2,2}$. Thus, nonzero polynomials of degree up to~$1$ were avoided in~$\clX_S$ (otherwise there would exist some~$\Theta$ for which the quotient would vanish, and thus the divergence to~$+\infty$ would not hold).
In our case, we can follow analogue arguments,  to show that the quotient
\[
\inf\limits_{\Theta\in \fkZ_S\setminus\{0\}}\tfrac{\norm{\nabla_x^4\Theta}{L^{2}}^{2}}{\norm{\nabla_x^2\Theta}{L^{2}}^2},
\]
will diverge as~$S$ increases. Now, we need to avoid nonzero polynomials of degree up to~$3$ in~$\fkZ_S$. Of course, in other applications it may happen that we may need to show the divergence of quocients of analogue different seminorms (e.g., for higher order parabolic-like equations). Thus we may need to avoid higher order polynomials (cf.~Rem.~\ref{R:balls-and-cells}).
\end{remark}

\subsubsection{Remark on the choice of the sensor locations}
We have seen that it is enough to choose the reference set~$\fkX_1$ of points as~\eqref{sensorsS=1}. So in the case~$d=1$ we can take four arbitrary distinct points.

However, we can see that the choice of~$\lambda$ in Lemma~\ref{L:MlamPoinc} depends on~$S$, namely, on the constant~$C^P_S$
as in~\eqref{normOP}. In particular, it depends on an auxiliary space~$\widetilde W_S$ such that
the direct sum~$\rmD(A)=\widetilde W_S\oplus (\rmD(A)\bigcap\ker\clZ_s)$ holds true. Thus, it is not clear a priori how to give a quantitative estimate on how~$\lambda$ depends on~$S$. In particular, if we need large~$S$ and if~$\lambda(S)$ increases fast then the required values of~$\lambda$ may be too large for practical applications/simulations. In other words,  in applications to concrete problems, the location of the sensors (for each given~$S$) may play an important role on the practicability of the proposed observer.

Next, we give an example where we can give a quantitative estimate on the constant~$C^P_S$, thus an estimate for a lower bound for the required~$\lambda$. 
\begin{theorem}
Let~$d=1$, and let us set the reference set of points
\begin{equation}\label{fkXuni1d}
\fkX_1=\fkX_1^{\rm uni}\coloneqq\{\tfrac18,\tfrac38,\tfrac58,\tfrac78\},
\end{equation}
in the Torus~$\Omega=\bbT_1^1=[0,1)$. Let us construct~$W_S$ as in~\eqref{WS-constr} and choose the auxiliary set of functions~$\widetilde W_S$, in Assumption~\ref{A:sens}, as the set of the first eigenfunctions~$e_j$ of the Laplacian~$\Delta$ under periodic boundary conditions, $\widetilde W_S=E_S\coloneqq\{e_j\mid 1\le j\le 4S\}$. Then, with~$\nu_2>0$ and~$A=\nu_2(-\Delta+\Id)^2$
we have that~$C^P_S$
 as in~\eqref{normOP} is given by
\[
C^P_S=\nu_2(16S^2\pi^2+1)^2(\tfrac1{2S})^\frac12.
\]
\end{theorem}
\begin{proof}
Let us fix an arbitrary~$S\in\bbN_+$. With~$\fkX_1=\fkX_1^{\rm uni}$ as in~\eqref{fkXuni1d}, the set of sensors constructed as in~\eqref{sensorsS>1} is given by
\[
W_S=\{x^{S,i}\mid 1\le i\le 4S\},\qquad x^{S,i}\coloneqq\tfrac{2i-1}{8S}.
\]

Let~$\widetilde W_S\coloneqq\clE_S^\rmf\coloneqq \linspan E_S^\rmf$ be the linear span of the first eigenfunctions,
$ E_S^\rmf\coloneqq\{e_j\mid 1\le j\le 4S\}$. Recall that
\begin{equation}\label{eigsLap1D}
-\Delta e_j=-\tfrac{\p^2}{\p x^2}e_j=\overline\alpha_j e_j,\quad e_{j}(x)\coloneqq \begin{cases}
\cos((j-1)\pi x),&\mbox{ if $j$ is odd},\\
\sin(j\pi  x),&\mbox{ if $j$ is even}.
\end{cases}
 \end{equation}
 
In Assumption~\ref{A:sens} we require that 
$\rmD(A)=\clE_S^\rmf\oplus\fkZ_S$, with~$\fkZ_S=\rmD(A)\bigcap\ker\clZ_S$. Hence, for any given~$h\in\rmD(A)$ we must have
$\clZ_S(h)=\clZ_S(P_{\clE_S^\rmf}^{\fkZ_S}h)$, in particular, we need to show that the output operator~$\clZ_S$ defines a bijection from~$\clE_S^\rmf$ onto~$\bbR^{S_\sigma}$, with~$S_\sigma=4S$. As in the proof of Lemma~\ref{L:bddCPS} we consider the mapping
\[
v\mapsto \overline\fkE_S v\coloneqq\clZ_S\overline\Psi v=\begin{bmatrix}
 e_j(x^{S,i})
\end{bmatrix}v,\quad\mbox{with}\quad \overline\Psi v\coloneqq{\textstyle\sum\limits_{j=1}^{4S}}v_j \overline e_j, 
\]
where we take normalized eigenfunctions in~$\rmD(A)$, that is,~$\overline  e_i\coloneqq\norm{e_i}{\rmD(A)}^{-1} e_i$.
We see that we need to show that the matrix
$\overline\fkE_S=\begin{bmatrix}
 \overline e_j(x^{S,i})
\end{bmatrix}
$
with entry~${\overline\fkE_S}_{(i,j)}=\overline e_j(x^{S,i})$ in the~$i$th row and~$j$th column is invertible. 
It is enough to show that the analogue matrix~$\fkE_S\coloneqq\begin{bmatrix}
 e_j(x^{S,i})
\end{bmatrix}$
where we do not normalize the eigenfunctions is invertible. The latter matrix~$\fkE_S$
coincides with the transpose of the matrix in~\cite[Eq.2.19]{RodSeifu22-arx} which has been proven to be invertible within~\cite[Proof of Lemma~2.9]{RodSeifu22-arx}. We can conclude that the direct sum~$\rmD(A)=\clE_S\oplus\fkZ_S$ holds true. Furthermore, following the arguments in the proof of Lemma~\ref{L:bddCPS} we arrive at the analogue of~\eqref{obli_proj} as
\[
P_{\clE_S}^{\fkZ_S}h=\overline \Psi\,\overline \fkE_S^{-1}\clZ_S h,
\]
 for the oblique
projection~$P_{\clE_S}^{\fkZ_S}$ in~$\rmD(A)$ onto~$\clE_S$ along~$\fkZ_S$.
Then, we find that
\begin{align}\notag
\norm{P_{\clE_S}^{\fkZ_S}h}{\rmD(A)}^2=\norm{\overline \Psi\,\overline \fkE_S^{-1}\clZ_S h}{\rmD(A)}^2
=\norm{\overline \fkE_S^{-1}\clZ_S h}{\bbR^{4S}}^2=(\clZ_S h)^\top(\overline \fkE_S^{-1})^\top\overline \fkE_S^{-1}\clZ_S h
 \end{align}
which implies that the constant in~\eqref{normOP} satisfies
\begin{align}
(C^P_S)^2
&=\sup_{h\in\rmD(A)\setminus\ker\clZ_S}\tfrac{  \norm{P_{\widetilde \clW_S}^{\fkZ_S}h}{\rmD(A)}^2 }{ \norm{\clZ_Sh}{\bbR^{S_\sigma\times1}}^2  }
=\sup_{h\in\rmD(A)\setminus\ker\clZ_S}\tfrac{  \norm{(\clZ_S h)^\top(\overline \fkE_S^{-1})^\top\overline \fkE_S^{-1}\clZ_S h}{\bbR^{S_\sigma\times1}}^2 }{ \norm{\clZ_Sh}{\bbR^{S_\sigma\times1}}^2  }\notag\\
&=\sup_{v\in\bbR^{S_\sigma\times1}\setminus\{0\}}\tfrac{  \norm{v^\top(\overline \fkE_S^{-1})^\top\overline \fkE_S^{-1}v}{\bbR^{S_\sigma\times1}} }{ \norm{v}{\bbR^{S_\sigma\times1}}  }
= {\rm eig}(\overline \Pi_S,4S),\notag
 \end{align}
where~${\rm eig(\overline \Pi_S,4S)}$ is the largest eigenvalue of~$\overline \Pi_S\coloneqq(\overline \fkE_S^{-1})^\top\overline \fkE_S^{-1}$.

Following the arguments in \cite[proof of Lem.~2.10]{RodSeifu22-arx} we know that~$\Pi_S^{-1}\coloneqq(\fkE_S) \fkE_S^{\top}$ is a diagonal matrix with entries~$\Pi_{S,(1,1)}^{-1}=4S$ and~$\Pi_{S,(i,i)}^{-1}=\frac{4S}2$ for~$2\le i\le 4S$.
It follows, from the orthogonality of the eigenfunctions, that
\[
\overline \Pi_S=((D_S \fkE_S)^{-1})^\top(D_S\fkE_S)^{-1}=(D_S^{-1})^\top \Pi_{S} D_S^{-1}=D_S^{-1} \Pi_{S} D_S^{-1}
\]
where~$D_S$ is the diagonal matrix with entries~$D_{S,(i,i)}={\rm eig}(A^{-1},i)=\norm{e_i}{\rmD(A)}^{-1}$ as the eigenvalues of~$A^{-1}$.
Hence~$\overline \Pi_S$ is again diagonal with
entries
\begin{align}
&\overline\Pi_{S,(1,1)}={\rm eig}(A,1)^2{(4S)}^{-1}=\nu_2^2\tfrac1{4S};\notag\\
&\overline\Pi_{S,(i,i)}={\rm eig}(A,i)^2{(2S)}^{-1}=\nu_2^2(4(\tfrac{i}2)^2\pi^2+1)^4\tfrac1{2S}&&\quad\mbox{for even~$i$},\quad 2\le i\le 4S;\notag\\
&\overline\Pi_{S,(i,i)}={\rm eig}(A,i)^2{(2S)}^{-1}=\nu_2^2(4(\tfrac{i-1}2)^2\pi^2+1)^4\tfrac1{2S},&&\quad\mbox{for odd~$i$},\quad 3\le i\le 4S.\notag
\end{align}
Hence we can conclude that the largest eigenvalue satisfies
\[
{\rm eig}(\overline \Pi_S,4S)=\nu_2^2(4(\tfrac{4S}2)^2\pi^2+1)^4\tfrac1{2S}=\nu_2^2(16S^2\pi^2+1)^4\tfrac1{2S},
\]
which gives us~$
(C^P_S)^2=\nu_2^2(16S^2\pi^2+1)^4\tfrac1{2S}$.
\end{proof}

%%%%%%%%%%%%%%%%%%%%%
\subsection{Satisfiability of Assumption~\ref{A:Inj}}\label{sS:chkAInj}
To check Assumption~\ref{A:Inj} we need to specify the reference output injection operator~$\overline\fkI_{S}$. 
Recalling our proposed output injection operator in~\eqref{sys-haty-o-Inj},
by a comparison with the abstract formulation in~\eqref{sys-z}, we look for~$\overline\fkI_{S}$ such that~$\fkI_{S}^{[\lambda,\Lambda]}z=-\lambda A^{-1}\overline\fkI_{S}\clZ_S z$, 
\begin{equation}\label{KSlam}
\fkI_{S}^{[\lambda,\Lambda]}\clZ_S z=-\lambda A^{-1}\clZ_S^*\Lambda\clZ_S z=-\lambda A^{-1}\overline\fkI_{S}\clZ_S z. 
\end{equation}
By taking~$\overline\fkI_{S}\coloneqq\clZ_S^*\Lambda$, it follows
 \begin{align}\notag
(A^{-1}\overline\fkI_{S}\clZ_S z,Az)_H &=(\overline\fkI_{S}\clZ_S z,z)_{\rmD(A^{-1}),\rmD(A)}=(\Lambda\clZ_S z,\clZ_S z)_{\bbR^{S_\sigma}}.
\end{align}
Observe that from~${\rm eig}(\Lambda+\Lambda^\top,1)=1$ as in~\eqref{sys-haty-o-Inj} and
 \begin{align}\label{posAAtop}
(\Lambda\clZ_S z,\clZ_S z)_{\bbR^{S_\sigma}}=(\clZ_Sz)^\top\Lambda\clZ_Sz=\tfrac12(\clZ_Sz)^\top(\Lambda+\Lambda^\top)\clZ_Sz,
\end{align}
it follows that~$(\Lambda\clZ_S z,\clZ_S z)_{\bbR^{S_\sigma}}\ge\tfrac12(\clZ_S z,\clZ_S z)_{\bbR^{S_\sigma}}$. Thus, Assumption~\ref{A:Inj} holds true.

%%%%%%%%%%%%%%%%%%%%%%%%%%%%%
%%%%%%%%%%%%%%%%%%%%%%%%%%%%%
 \section{Numerical Simulations for the K--S flame propagation model}\label{S:simul-flame}
 
We consider the one-dimensional, $d=1$, Kuramoto--Sivashinsky model for flame propagation. We show the results of simulations illustrating the detectability result
stated in Main Result in the Introduction; see main Theorem~\ref{T:main}. We take a vanishing external forcing~$f=0$, and consider periodic boundary conditions with period~$1$. Hence, the spatial domain is the Torus~$\Omega=\bbT_1^1=[0,1)$.
We solve both the nominal system model~\eqref{sys-KS-flame}, with output as in~\eqref{sys-Intro-o}
\begin{subequations}\label{sys-y-num}
  \begin{align}
 &\tfrac{\p}{\p t}y_{\ttr}+\nu_2 \tfrac{\p^4}{\p x^4} y_{\ttr}+\nu_1 \tfrac{\p^2}{\p x^2}  y_{\ttr}+\nu_0\tfrac{1}{2}\norm{\tfrac{\p}{\p x} y_{\ttr}}{\bbR}^2=0,\\
 &\clZ_{S} y_{\ttr}\coloneqq\begin{bmatrix}
           y_{\ttr} (x^{S,1}) &  y_{\ttr}(x^{S,2}) & \dots & y_{\ttr}(x^{S,S_\sigma})
         \end{bmatrix}^\top\in\bbR^{S_\sigma\times 1},
\end{align}
\end{subequations}
and the Luenberger observer~\eqref{sys-KS-flame-obs}, with output injection operator~\eqref{sys-haty-o-Inj},
\begin{subequations}\label{sys-haty-num}
  \begin{align}
 &\tfrac{\p}{\p t}y_\tte+\nu_2 \tfrac{\p^4}{\p x^4} y_\tte+\nu_1 \tfrac{\p^2}{\p x^2} y_\tte+\nu_0\tfrac{1}{2}\norm{\tfrac{\p}{\p x} y_\tte}{\bbR}^2=-\lambda A^{-1}\clZ_S^*\Lambda(\clZ_S y_\tte-\clZ_S y_{\ttr})\\
&\mbox{with}\quad A=\nu_2(-\tfrac{\p^2}{\p x^2}+\Id)^2,
\end{align}
\end{subequations}
where~$\Lambda\in\bbR^{S_\sigma\times S\sigma}$ is a suitable matrix such that (cf.~\eqref{sys-haty-o-Inj})
$\Lambda+\Lambda^\top$ is positive definite with~${\rm eig}(\Lambda+\Lambda^\top,1)=1$.

%%%%%%%%%%%%%%%%%%%%%%%%%%%%%%%
\subsection{Spatial Discretization}
We use similar numerical setting as in \cite{RodSeifu22-arx}, where spectral elements were used to compute the solutions of Galerkin approximations based on ``the'' first ~$N$ (periodic) eigenfunctions of the Laplacian operator. Note that the eigenfunctions of the Laplacian coincide with those of~$A=\nu_2(-\tfrac{\p^2}{\p x^2}+\Id)^2$. 
We look for approximations~$y_\ttr^N(t,\Bigcdot)\in\clE^\rmf_N$  and~$y_\tte^N(t,\Bigcdot)\in\clE^\rmf_N$ of the states~$y_\ttr(t,\Bigcdot)$  and~$y_\tte(t,\Bigcdot)$ of the systems~\eqref{sys-y-num} and~\eqref{sys-haty-num}, respectively, in the linear span~$\clE^\rmf_N\coloneqq\linspan\{e_n\mid 1\le n\le N\}$, of the first~$N\ge S_\sigma$ eigenfunctions of the Laplacian as in~\eqref{eigsLap1D}.
Namely, for~$t\ge0$ and~$x\in[0,1)=\bbT_1^1$, as
\begin{equation}\label{yN-Galsum}
y_\ttr^N(t,x)=\sum\limits_{n=1}^N y_{\ttr;n}^N(t)e_n(x)\qquad\mbox{and}\quad y_{\tte}^N(t,x)=\sum\limits_{n=1}^N y_{\tte;n}^N(t)e_n(x).
\end{equation}
Denoting  the orthogonal projection in~$H=L^2(\bbT_1^1)$ onto~$\clE^\rmf_N$ by~$P_{\clE^\rmf_N}\in\clL(H,\clE^\rmf_N)$, we compute~$y_\ttr^N$ and~$y_\tte^N$ by solving Galerkin approximations as follows (cf.\cite[sect.~4.1]{RodSeifu22-arx}, \cite[sect.~3.4]{Rod21-jnls}, \cite[sect.~4.3]{Rod20-eect}, \cite[Ch.~3, sect.~3.2]{Temam01}),
\begin{align}
 &\dot{y_\ttr}^N+ \nu_2\tfrac{\p^4}{\p x^4} y_\ttr^N+ \nu_1\tfrac{\p^2}{\p x^2} y_\ttr^N+\nu_0 P_{\clE^\rmf_N}\clN_1(y_\ttr^N)=0,
 \notag\\
 &\dot y_\tte^N+ \nu_2\tfrac{\p^4}{\p x^4} y_\tte^N+ \nu_1\tfrac{\p^2}{\p x^2} y_\tte^N+\nu_0 P_{\clE^\rmf_N}\clN_1(y_\tte^N)=-\lambda P_{\clE^\rmf_N} A^{-1}\clZ_S^*\Lambda(\clZ_S y_\tte^N-\clZ_S y_{\ttr}^N),\notag
 \intertext{with the nonlinearity}
 &\clN_1(h^N)\coloneqq\tfrac12\norm{\p_x h^N}{\bbR}^2,
 \quad\mbox{and initial states}\quad y_\ttr^N(0,\Bigcdot)=P_{\clE^\rmf_N}y_{\ttr0},\quad y_\tte^N(0,\Bigcdot)=P_{\clE^\rmf_N}y_{\tte0}.
 \notag
\end{align}
Note that both~$\tfrac{\p^4}{\p x^4}$ and~$\tfrac{\p^2}{\p x^2}$ map~$\clE^\rmf_N$ into itself.
Essentially, we solve a system of~$N$ ordinary differential equations (one equation for each spectral coordinate index~$n$) as 
\begin{subequations}\label{galerkin-coord}
\begin{align}
& \dot y_{\tte;n}^N=- (\nu_2\overline\alpha_n^2-\nu_1\overline\alpha_n) y_{\tte;n}^N-\nu_0\left(P_{\clE^\rmf_N}\clN_1(y_\tte^N)\right)_n\notag\\
 &\hspace{2.5em}-\lambda\left(P_{\clE^\rmf_N}A^{-1}\clZ_S^*\Lambda \clZ_S(y_\tte^N-y_\ttr^N)\right)_n,\\
 &y_{\tte;n}^N(0,\Bigcdot)=y_{\tte0;n}.
\end{align}
\end{subequations}
for the observer state estimate, and analogously for the coordinates~$y_{\ttr;n}^N$ of the nominal state with~$\lambda=0$.
In conclusion, the computed approximated solutions for~$y_\ttr$ and for~$y_\tte$  are linear combinations  of the first eigenfunctions in the spatial interval~$[0,1)$ with coordinates given as in~\eqref{galerkin-coord}.

We compute the orthogonal projections~$P_{\clE^\rmf_N}h$  above by firstly evaluating $h$ in the nodes of a (regular) mesh/partition, for a space-step~$0<x^{\rm step}=\tfrac1{N_{\rm fem}}<1$,
\begin{equation}\label{disc-spInt}
[0,1)_{\rm disc}= \{(n-1)\tfrac1{N_{\rm fem}}\mid 1\le n\le N_{\rm fem}\},\qquad N_{\rm fem}\in\bbN_+,\quad N_{\rm fem}\ge2.
\end{equation}
 of the spatial interval~$[0,1)$ and use the associated finite-element mass matrix (corresponding to the considered periodic boundary conditions) as an auxiliary tool, to compute the coordinates of~$P_{\clE^\rmf_N}h$ following~\cite[sect.~8.1]{RodSturm20}. 
The finite-element basis vectors were taken as the classical hat-functions (piecewise-linear elements).

While solving the equations, we take the output~$w^N=\clZ_S y_\ttr^N$ containing the measurements~$y_\ttr^N(x^{S,i})$, which is then injected into the dynamics of the observer as
\begin{align}\label{GalInj}
-\lambda  P_{\clE^\rmf_N}A^{-1}\clZ_S^*\Lambda (\clZ_Sy_\tte^N-\clZ_Sy_\ttr^N).
\end{align}
For an arbitrary eigenfunction~$e_m\in\clE^\rmf_N$ we find
\begin{align}\notag
&(P_{\clE^\rmf_N}A^{-1}\clZ_S^*\Lambda (\clZ_Sy_\tte^N-\clZ_Sy_\ttr^N),e_m)_H=(\clZ_S^*\Lambda (\clZ_Sy_\tte^N-\clZ_Sy_\ttr^N),A^{-1}e_m)_H\notag\\
&\hspace{3em}=(\Lambda (\clZ_Sy_\tte^N-\clZ_Sy_\ttr^N),\clZ_S A^{-1}e_m)_{\bbR^{S_\sigma}}=(\Lambda(\clZ_Sy_\tte^N-\clZ_Sy_\ttr^N), \nu_2^{-1}\alpha_m^{-1}\clZ_Se_m)_{\bbR^{S_\sigma}}\notag\\
&\hspace{3em}=\nu_2^{-1}\alpha_m^{-1}(\clZ_Se_m)^\top\Lambda(\clZ_Sy_\tte^N-\clZ_Sy_\ttr^N),\label{discInj1}
\end{align}
where~$\alpha_m\coloneqq (\overline\alpha_m +1)^2$ is the~$m$th eigenvalue of~$A_0\coloneqq(-\tfrac{\p^2}{\p x^2}+\Id)^2$. Therefore, if we set the matrix
\begin{equation}\label{matrix_Ex}
\bfE_\fkX^\circ= \begin{bmatrix}\bfE^\circ_{\fkX,(r,c)}\end{bmatrix}\in\bbR^{N\times S_\sigma}\quad\mbox{with}\quad
\bfE^\circ_{\fkX,(r,c)} \coloneqq \alpha_c^{-1}e_c(x^{S,r}),
\end{equation}
that is, with entry~$\bfE^\circ_{\fkX,(r,c)}$ in the~$r$th row and~$c$th column, we observe that 
\begin{align}
\clZ_Sy^N=\bfE_\fkX^\circ y^N\in\bbR^{S_\sigma\times1},\quad\mbox{for every}\quad y^N\in\clE^\rmf_N,\label{disc-output}
\end{align}
giving us the spectral discretization of the output/measurement operator. 

We observe that, with~$I\in\bbR^{N\times 1}$,
\begin{subequations}\label{matrix_bfI-spe}
\begin{align}
&-\lambda P_{\clE^\rmf_N}A^{-1}\clZ_S^*\Lambda (\clZ_Sy_\tte^N-\clZ_S y_\ttr^N)= {\textstyle\sum\limits_{j=1}^{N}} I_{(j,1)}e_j\notag\\
\Longleftrightarrow\quad& I= -\lambda\nu_2^{-1} \clD_\alpha^{-1}(\bfE_\fkX^\circ)^\top\Lambda (\clZ_Sy_\tte^N-\clZ_S y_\ttr^N).
\intertext{where~$\clD_\alpha\in\bbR^{N\times N}$ is diagonal with entries}
&\hspace{-2em}\clD_{\alpha,(i,i)}=\alpha_i=(\overline\alpha_i+1)^2,\quad 1\le i\le N.
 \end{align}
 \end{subequations}
  Indeed,  with~$\omega\coloneqq\clZ_Sy_\tte^N-\clZ_S y_\ttr^N$, 
\begin{align}
&({\textstyle\sum\limits_{j=1}^{N}}(-\lambda \clD_\alpha^{-1}(\bfE_\fkX^\circ)^\perp\Lambda \omega)_{(j,1)}e_j, e_m)_H=(-\lambda \clD_\alpha^{-1}(\bfE_\fkX^\circ)^\perp\Lambda \omega)_{(m,1)}\notag\\
&\hspace{3em}=-\lambda \alpha_m^{-1}\begin{bmatrix} (\bfE_\fkX^\circ)^\perp_{(m,1)}& (\bfE_\fkX^\circ)^\perp_{(m,2)}&\dots& (\bfE_\fkX^\circ)^\perp_{(m,S_\sigma)}\end{bmatrix}\Lambda \omega\notag\\
&\hspace{3em}=-\lambda \alpha_m^{-1}\begin{bmatrix} e_m(x^{S,1})& e_m(x^{S,2})&\dots& e_m(x^{S,S_\sigma})\end{bmatrix}\Lambda \omega\notag\\
&\hspace{3em}=-\lambda \alpha_m^{-1}(\clZ_Se_m)^\top\Lambda(\clZ_Sy_\tte^N-\clZ_Sy_\ttr^N)\notag
\end{align}
which agrees with~\eqref{discInj1}. Thus, \eqref{matrix_bfI-spe} gives us a simple expression to compute the (Galerkin coordinates of the) output injection term.

\subsection{Temporal discretization}
 The temporal discretization is based on an implicit-explict (IMEX) method combining the  Crank--Nicolson scheme and the Adams--Bashford extrapolation. The (implicit) Crank--Nicolson scheme is used for the linear component~$-(\nu_2\overline\alpha_n^2-\nu_1\overline\alpha_n)y^N_n$, $y^N_n\in\{y_{\ttr;n}^N,y_{\tte;n}^N\}$, and an (explicit) Adams-Bashford extrapolation is used for the nonlinear and injection component~$-\nu_0\left(P_{\clE_N}\clN_1(y^N)\right)_n-\lambda\left(P_{\clE_N}\fkI y^N\right)_n$, $y^N\in\{y_{\ttr}^N,y_{\tte}^N\}$.
The temporal step~$t^{\rm step}>0$ was taken uniform,
\begin{equation}\label{disc-tInt}
[0,+\infty)_{\rm disc}= \{nt^{\rm step}\mid n\in\bbN\}.
\end{equation}

\subsection{Performance of the observer}
 We choose the parameters in~\eqref{sys-y-num} and~\eqref{sys-haty-num} as
 \begin{subequations}\label{param-flame-nominal}
\begin{equation}
\nu_2=10^{-6},\quad \nu_1=10^{-2},\quad \nu_0=10^{-2},
\end{equation}
and the initial states for targeted~$y$ and estimate~$\widetilde y$ trajectories as
\begin{equation}
y_{\ttr0}= 1+\sin(4\pi x),
\qquad y_{\tte0}=\cos(2\pi x)(1+\sin(2\pi x)).
\end{equation}
 \end{subequations}
 \begin{subequations}\label{param-flame-inj}
We take~$4$ reference sensors (corresponding to the case~$S=1$) located at the points
\begin{equation}
\fkX_1=\{x^{1,1},x^{1,2},x^{1,3},x^{1,4}\}=\{0,\tfrac14,\tfrac12,\tfrac34\},
\end{equation}
thus uniformly distributed in the Torus~$\bbT_1^1=[0,1)$.
We shall test the performance of the observer for several values of the gain parameter~$\lambda$ in the output injection operator, where the parameter~$S$ in the same operator is taken as
\begin{equation}
 S=9,\qquad\mbox{hence}\quad S_\sigma=36
 \end{equation}
sensors; see~\eqref{sensors}.
By construction as in~\eqref{sensorsS>1} (for~$S>1$) we have the sensor located at the points in the set
\begin{align}
    &\hspace{-.3em}\fkX_S\coloneqq\{x^{S,j}\mid 1\le j\le S_\sigma\},\qquad\mbox{with }\\
   &\hspace{-.3em}x^{S,4(k-1)+s}=(k-1)\tfrac{1}{S}+\tfrac{1}{S}x^{1,s}\in (0,\tfrac{1}{S}(k-1))\subset\bbT^1_1,\quad1\le k\le S,\; 1\le s\le 4.
\end{align}
Hence, our sensors are the delta distributions in the set
\begin{align}
W_S:=\{\deltafun_{x^{S,j}}|1\le j\le S_\sigma\}.
\end{align}
\end{subequations}
 \begin{subequations}\label{param-flame-Lam}
We choose the matrix~$\Lambda\in\bbR^{4S\times 4S}$ in the injection operator~\eqref{GalInj} as follows, motivated by numerical experiments,
\begin{align}
&\Lambda=\tfrac1{{\rm eig}(\overline\Lambda+\overline\Lambda^\top,1)}\overline\Lambda\quad
\mbox{with}\quad\overline\Lambda\coloneqq (\underline\bfE_\fkX^\circ)^{-\top}\underline\clD_\alpha (\underline\bfE_\fkX^\circ)^{\top},\\ \mbox{where}\quad&\underline\bfE_\fkX^\circ=\bfE_\fkX^\circ(1:S_\sigma,1:S_\sigma)
\quad\mbox{and}\quad\underline\clD_\alpha=\clD_\alpha(1:S_\sigma,1:S_\sigma).
\end{align}
That is, $\underline\bfE_\fkX^\circ\in\bbR^{S_\sigma\times S_\sigma}$  and~$\underline\clD_\alpha\in\bbR^{S_\sigma\times S_\sigma}$ are, respectively, the blocks of the first~$S_\sigma\le N$ rows and columns of~$\bfE_\fkX^\circ\in\bbR^{S_\sigma\times N}$ and~$\clD_\alpha\in\bbR^{N\times N}$.
\end{subequations}
Note that, the matrix~$\bfE_\fkX^\circ$ in~\eqref{matrix_Ex} does not depend on the discretization, hence~$\Lambda$ is independent of the discretization. Numerically, we have observed that the symmetric matrix~$\Lambda+\Lambda^\top$ is positive definite with eigenvalues ranging from~${\rm eig}(\Lambda+\Lambda^\top,1)=1$ to~${\rm eig}(\Lambda+\Lambda^\top,S_\sigma)\approx 2.49\times10^8$.

 \begin{subequations}\label{param-flame-discret}
 We solve  the $N$-dimensional Galerkin approximation, with
 \begin{align}
 N=200.
 \end{align}
 The temporal and spatial time steps, in~\eqref{disc-spInt} and~\eqref{disc-tInt} were taken as
  \begin{align}
 t^{\rm step}=10^{-3},\quad  x^{\rm step}=10^{-4}.
 \end{align}
 \end{subequations}

\begin{remark} Concerning the definition of~$\fkX_S$, note that~$4(k_1-1)+s_1=4(k_2-1)+s_2$ if, and only if, $4(k_1-k_2)=s_2-s_1$. Hence, 
if~$-3\le s_2- s_1\le 3$ we have that~$\tfrac{s_2- s_1}4$ is an integer number if, and only if,~$s_2=s_1$. Therefore, we can conclude that
\begin{align}
\Bigl(4(k_1-1)+s_1=4(k_2-1)+s_2\mbox{ and }  1\le s_1,s_2\le 4\Bigr)
\quad\Longleftrightarrow\quad &(k_1,s_1)=(k_2,s_2).\notag
\end{align}
Thus,~$(k,s)\mapsto 4(k-1)+s$ is a bijection from~$\{1,\dots,S\}\times\{1,\dots,4\}$ onto~$\{1,\dots,4S
\}$.
\end{remark}

\begin{remark}
The parameters above, in~\eqref{param-flame-nominal}, were taken in~\cite{RodSeifu22-arx} in the context of stabilizability (i.e., state stabilization by means of feedback controls).  Here we take the same parameters in the context of detectability (i.e., state estimation by means of an output injection operator). We stress that these are conceptually different contexts: for example, for stabilizability we are given a control operator~$B_M\coloneqq\bbR^{M_\varsigma}\to\clU_M$ onto the linear span of a given set of~$M_\varsigma$ actuators and we look for an operator~$K\colon V\to\bbR^{M_\varsigma}$ mapping the state~$y$ into the tuning parameters of  the control (e.g., ~$K=-B_M^*\Pi$ as in the classical linear quadratic regulator problem, where~$\Pi$ satisfies a Riccati equation), while for detectability we are given an output operator~$\clZ_S$ and look for an injection operator~$I\colon\bbR^{S_\sigma}\to H$ mapping the output~$\clZ_Sy$ onto a suitable superspace $H\supseteq V$ of the state space~$V$ (e.g.,~$I=-\lambda A^{-1}\clZ_S^*\Lambda$ as in~\eqref{KSlam}).
\end{remark}

%%%%%%%%%%%%%%%%%%%%%%%%%%%%%%
\subsubsection{Lack of asymptotic stability of the free dynamics.}
In Fig.~\ref{Fig:free-flame} we see that without the injection operator (i.e., with~$\lambda=0$) the estimate~$y_\tte=y^N_\tte$ given by the observer (likely) does not converge to the real targeted state~$y_\ttr=y_\ttr^N$. This shows, in particular, that the error free dynamics is not asymptotically stable and thus a nontrivial injection operator is necessary for the observer to give us a state estimate converging exponentially to~$y$.

In Fig.~\ref{Fig:free-av-flame} we plot the average-free components of the states, that is, recalling~\eqref{yN-Galsum},
\[
y^N_{\ttr\rm avf}(t,x)\coloneqq  y^N_{\ttr}(t,x)-y^N_{\ttr;1}(t)e_1(x)\quad\mbox{and}\quad y^N_{\tte\rm avf}(t,x)\coloneqq  y^N_{\tte}(t,x)-y^N_{\tte;1}(t)e_1(x),
\]
with~$e_1(x)=1$.  We see that also this component of the estimate  is not converging to the corresponding component of the targeted reference state. Note that for~$j\ge2$ the eigenfunctions are zero-averaged, $\int_\Omega e_j(x)\,\rmd x=(e_j,e_1)_{L^2(\Omega)}=0$, thus~$y^N_{\ttr\rm avf}(t,x)$ and~$y^N_{\tte\rm avf}(t,x)$ are zero-averaged.
\begin{figure}[ht]
\centering
\subfigure[targeted state.]
{\includegraphics[width=0.45\textwidth,height=0.35\textwidth]{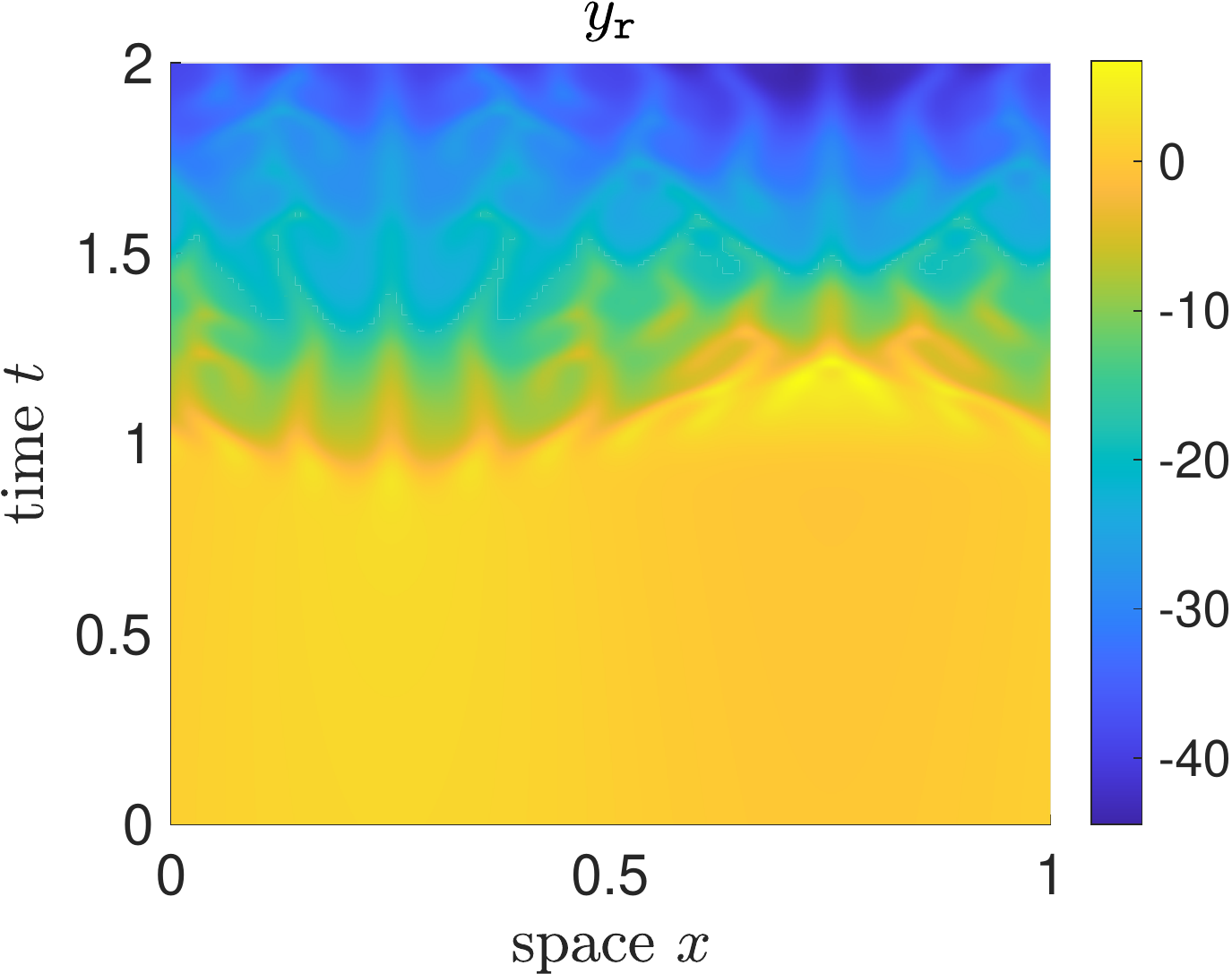}}
\qquad
\subfigure[free estimate.]
{\includegraphics[width=0.45\textwidth,height=0.35\textwidth]{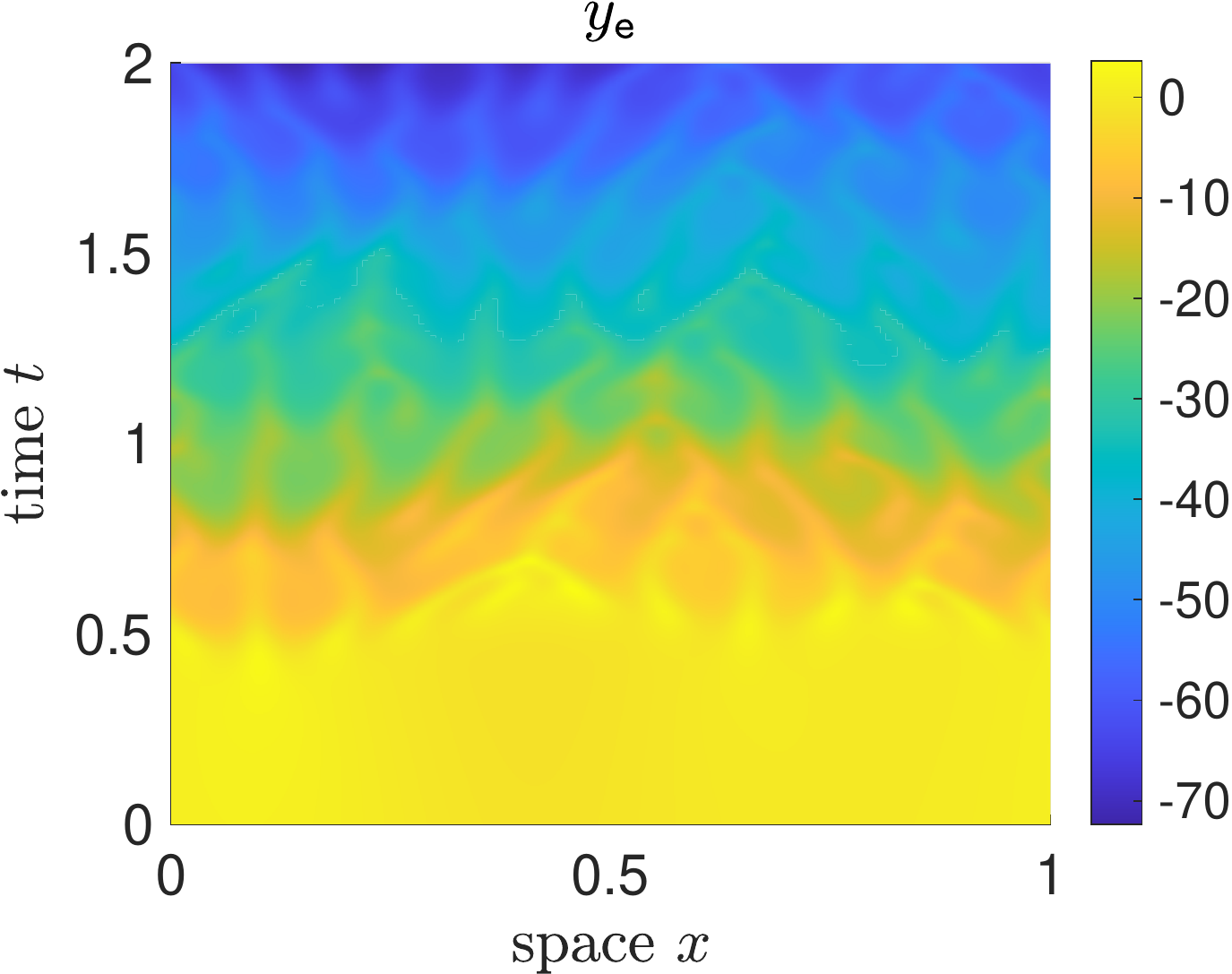}}
\caption{Case~$\lambda=0$ (free dynamics). Model~\eqref{sys-y-num}--\eqref{sys-haty-num}. \newline\newline}\label{Fig:free-flame}
\subfigure[zero-average component of targeted  state.]
{\includegraphics[width=0.45\textwidth,height=0.35\textwidth]{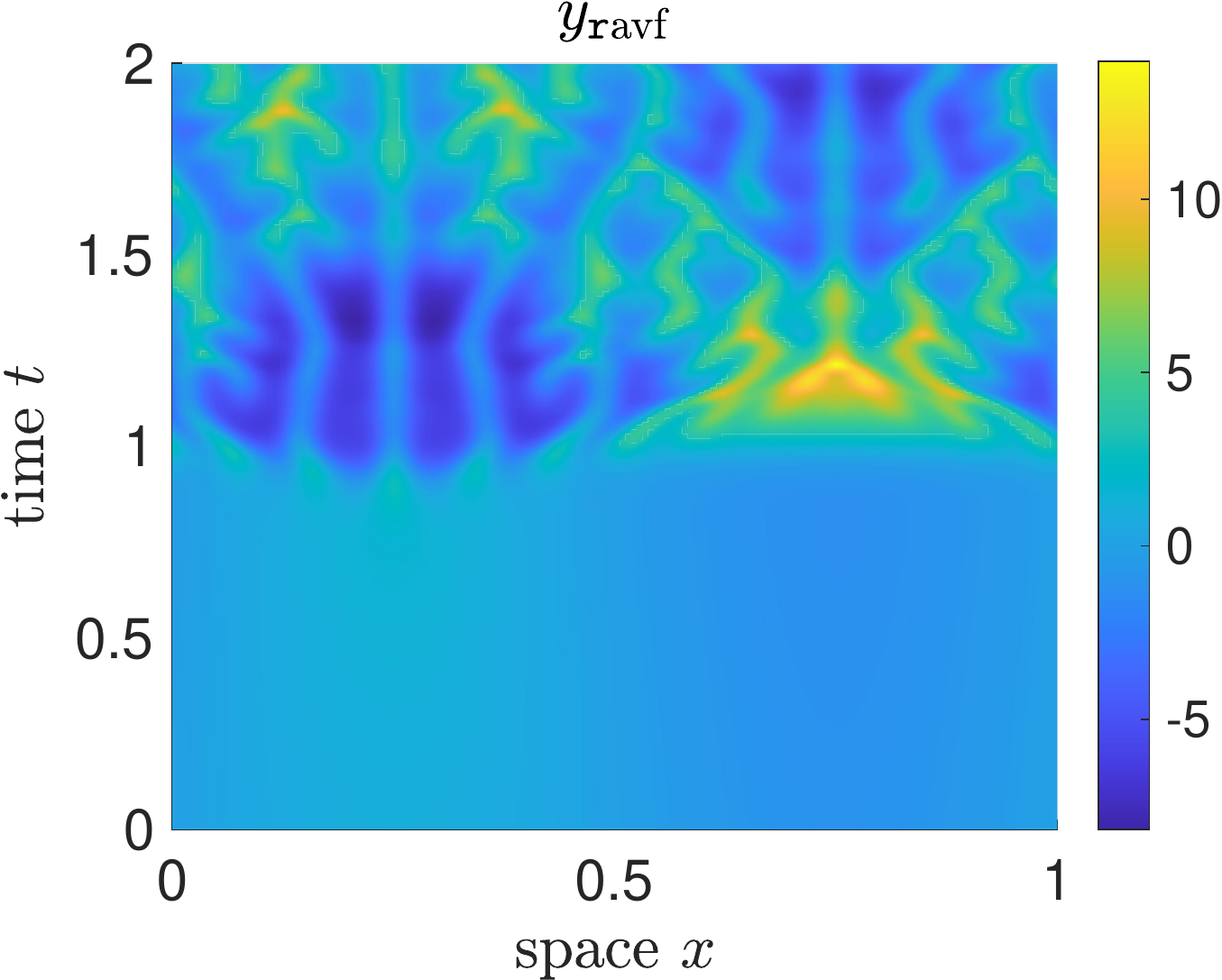}}
\qquad
\subfigure[zero-average component of free estimate.]
{\includegraphics[width=0.45\textwidth,height=0.35\textwidth]{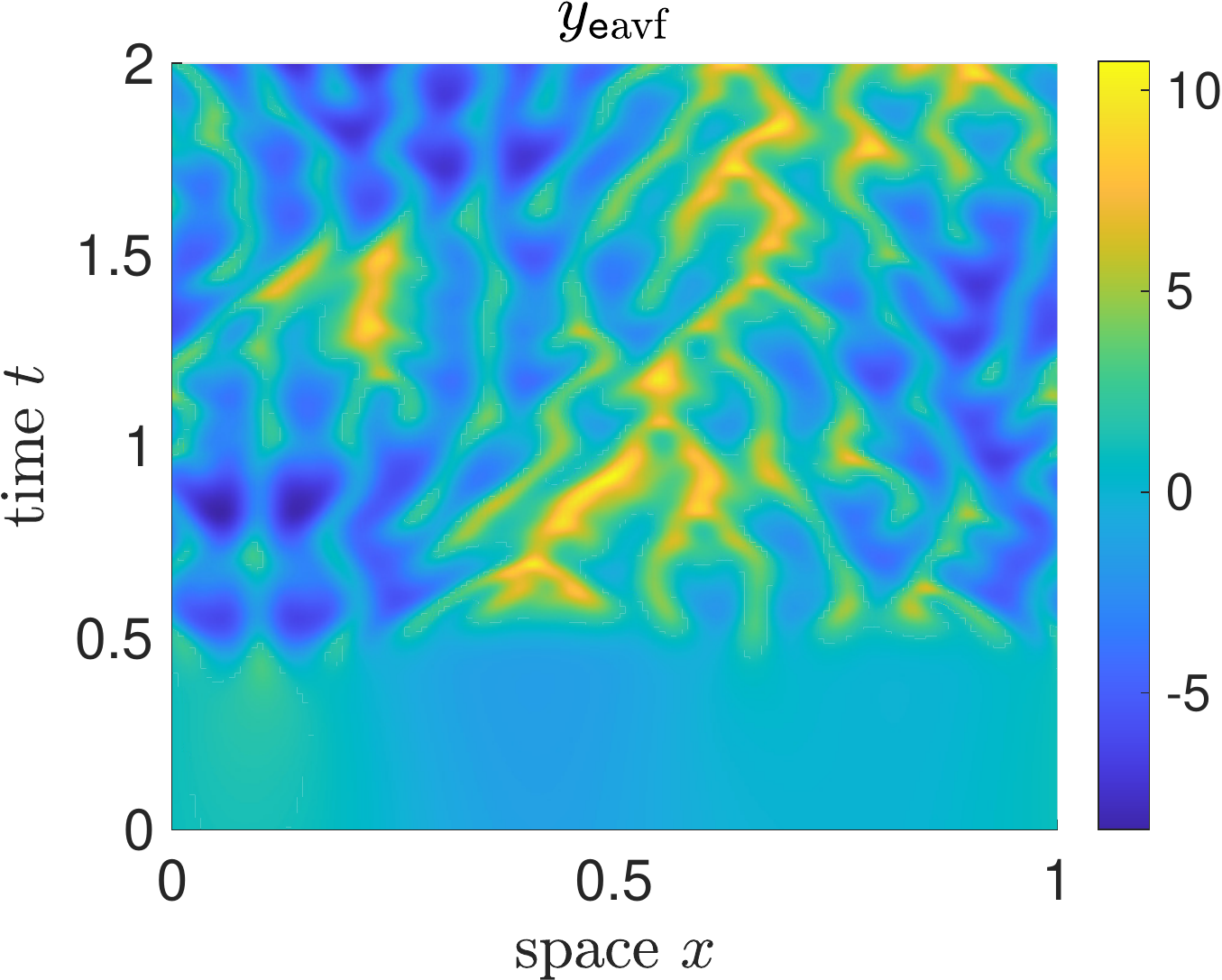}}
\caption{Case~$\lambda=0$. Zero-average component. Model~\eqref{sys-y-num}--\eqref{sys-haty-num}.\newline}\label{Fig:free-av-flame}
\end{figure}

%%%%%%%%%%%%%%%%%%%%%%%%%%%%%%
\subsubsection{With output injection as~\eqref{matrix_bfI-spe}}
The lack of asymptotic stability is confirmed in Fig.~\ref{Fig:lam_small-flame} where we plot, for the case~$\lambda=0$, the evolution of the norm of the error for a larger time interval.
Here we recall that~$\norm{\Bigcdot}{H}^2=\norm{\Bigcdot}{L^2(\bbT^1_1)}^2$ and $\norm{\Bigcdot}{V}^2=\nu_2\norm{(-\Delta+\Id)\Bigcdot}{L^2(\bbT^1_1)}^2$ as in~\eqref{Vnorm}.
In the same Fig.~\ref{Fig:lam_small-flame}, we also observe that we are not able to achieve the desired exponential stability of the error dynamics for small values of~$\lambda$. This shows that we need to take large enough~$\lambda$.
In Fig.~\ref{Fig:lam_large-flame} we confirm that by taking large (enough) values of~$\lambda$ we are able to reach the exponential stability of the error dynamics, which agrees with the theoretical result.

 Finally, in Fig.~\ref{Fig:lam_large-flame-snap} we present time-snapshots of the state error estimate together with the corresponding output error (from each of the $40$~sensors).
\begin{figure}[ht]
\centering
\subfigure%[]
{\includegraphics[width=0.45\textwidth,height=0.35\textwidth]{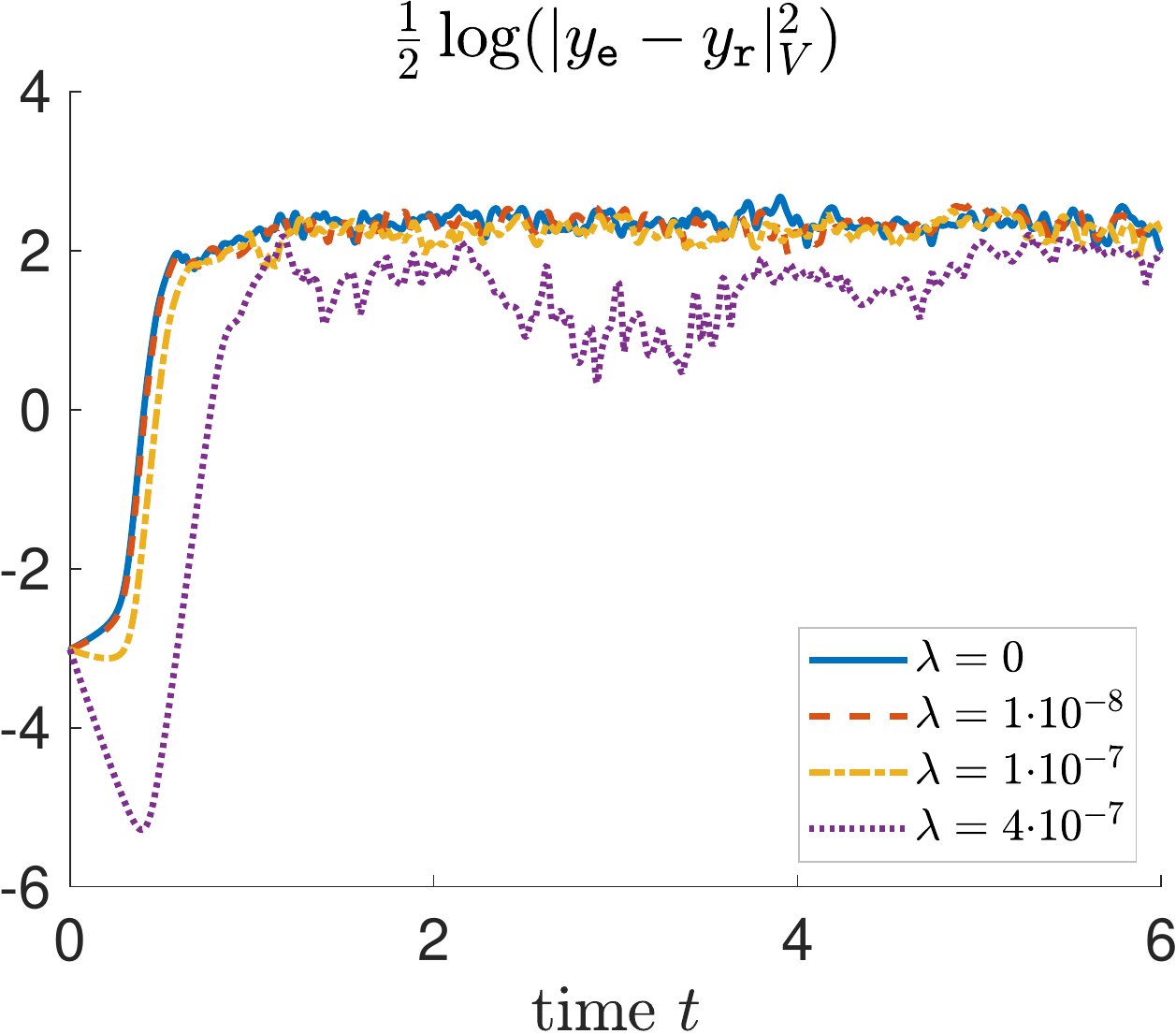}}
\qquad
\subfigure%[]
{\includegraphics[width=0.45\textwidth,height=0.35\textwidth]{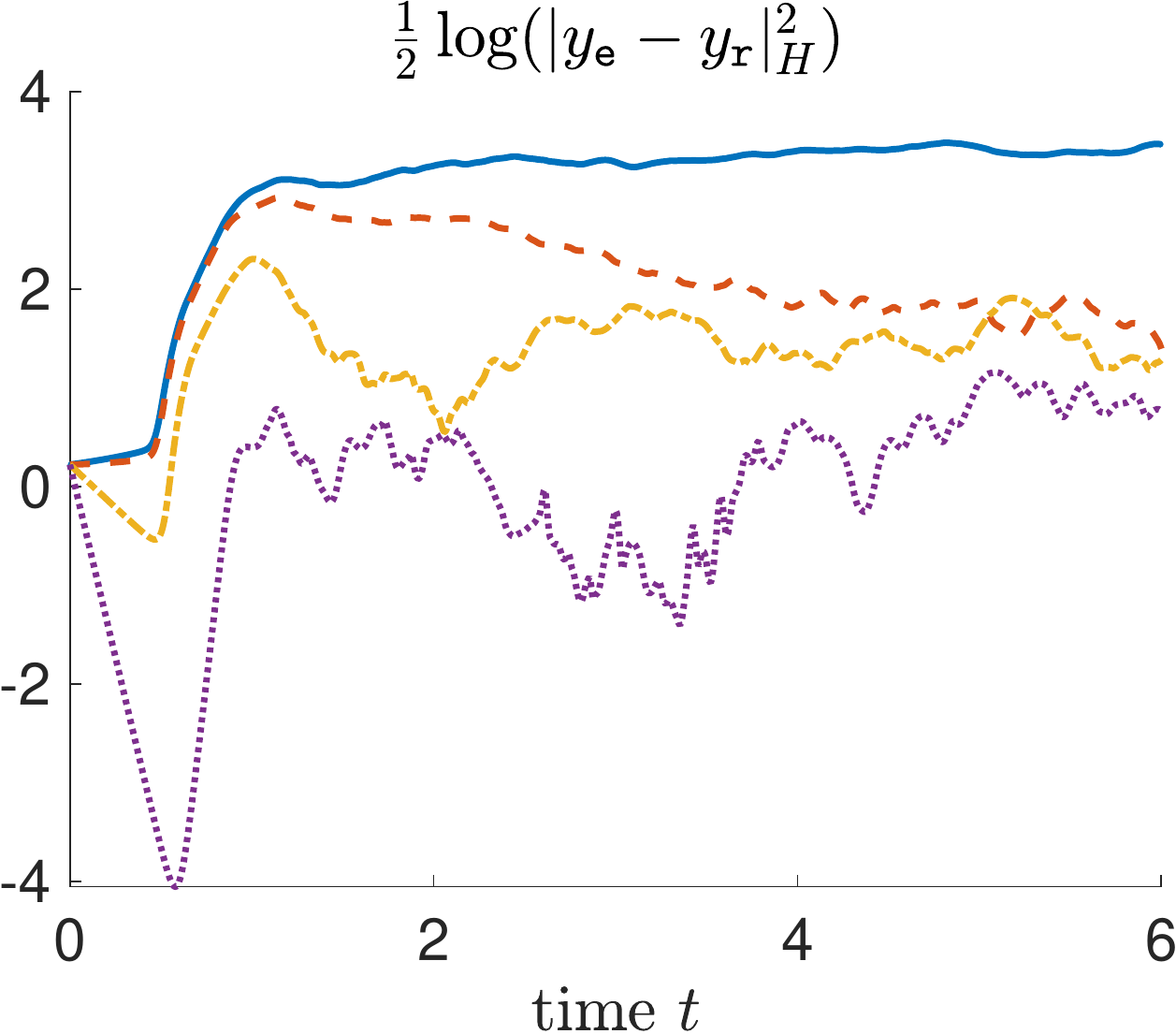}}
\caption{Estimate error for small~$\lambda\ge0$ in~\eqref{matrix_bfI-spe}. Model~\eqref{sys-y-num}--\eqref{sys-haty-num}.\newline\newline}\label{Fig:lam_small-flame}
\subfigure%[]
{\includegraphics[width=0.45\textwidth,height=0.35\textwidth]{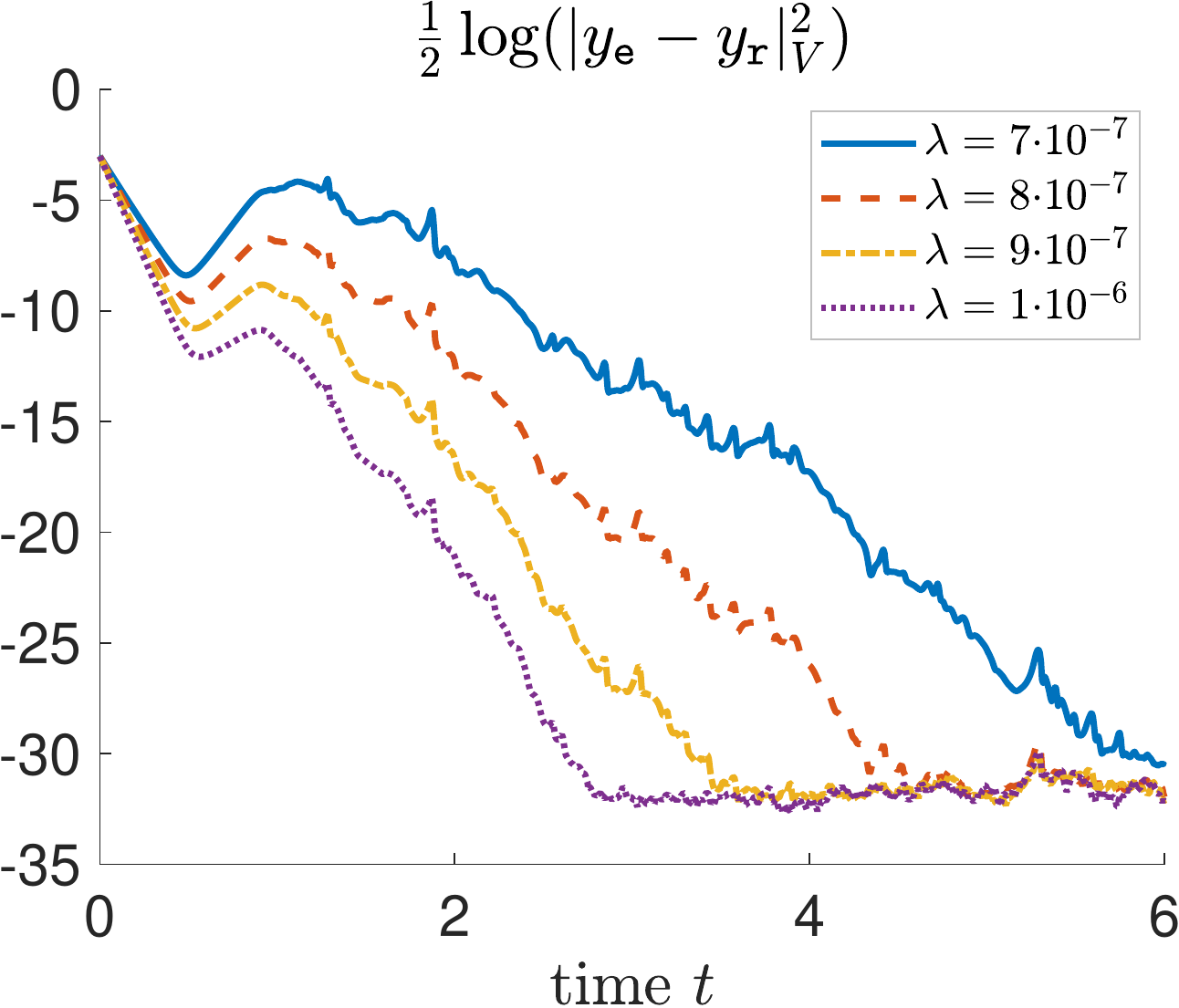}}
\qquad
\subfigure%[]
{\includegraphics[width=0.45\textwidth,height=0.35\textwidth]{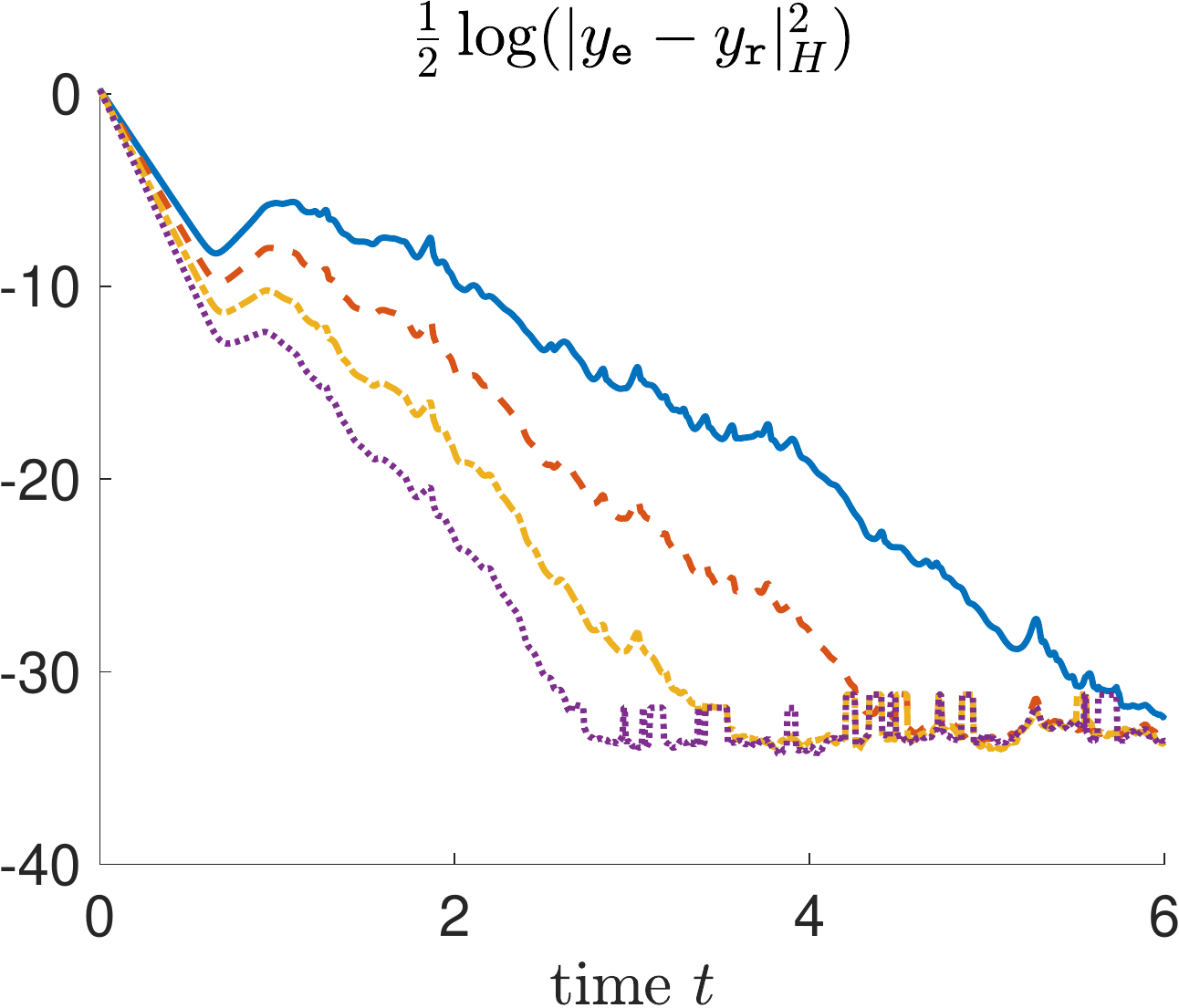}}
\caption{Estimate error for larger~$\lambda\ge0$ in~\eqref{matrix_bfI-spe}. Model~\eqref{sys-y-num}--\eqref{sys-haty-num}.\newline}\label{Fig:lam_large-flame}
\end{figure}

\begin{figure}[ht]
\centering
\subfigure%[]
{\includegraphics[width=0.35\textwidth,height=0.30\textwidth]{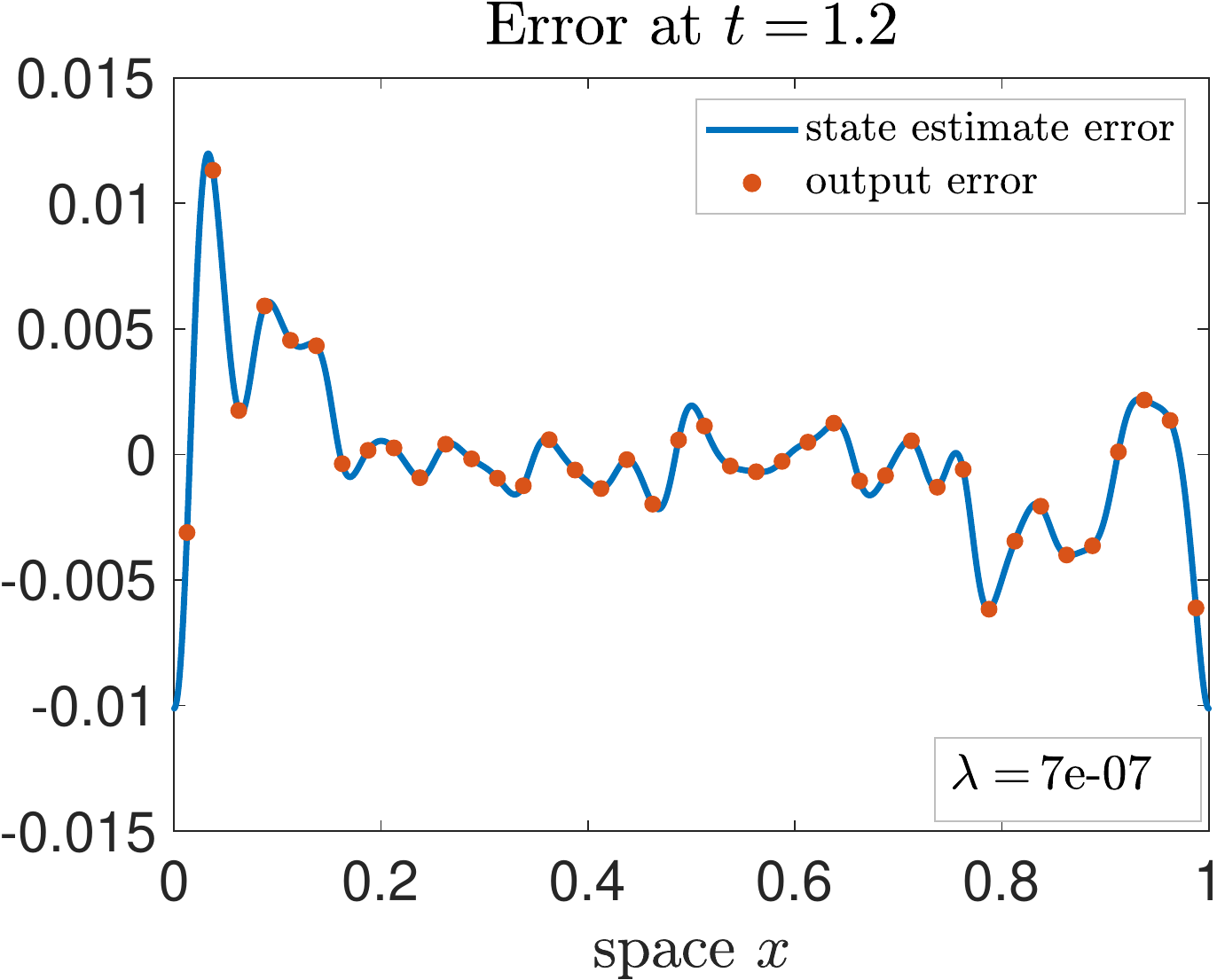}}
\subfigure%[]
{\includegraphics[width=0.30\textwidth,height=0.30\textwidth]{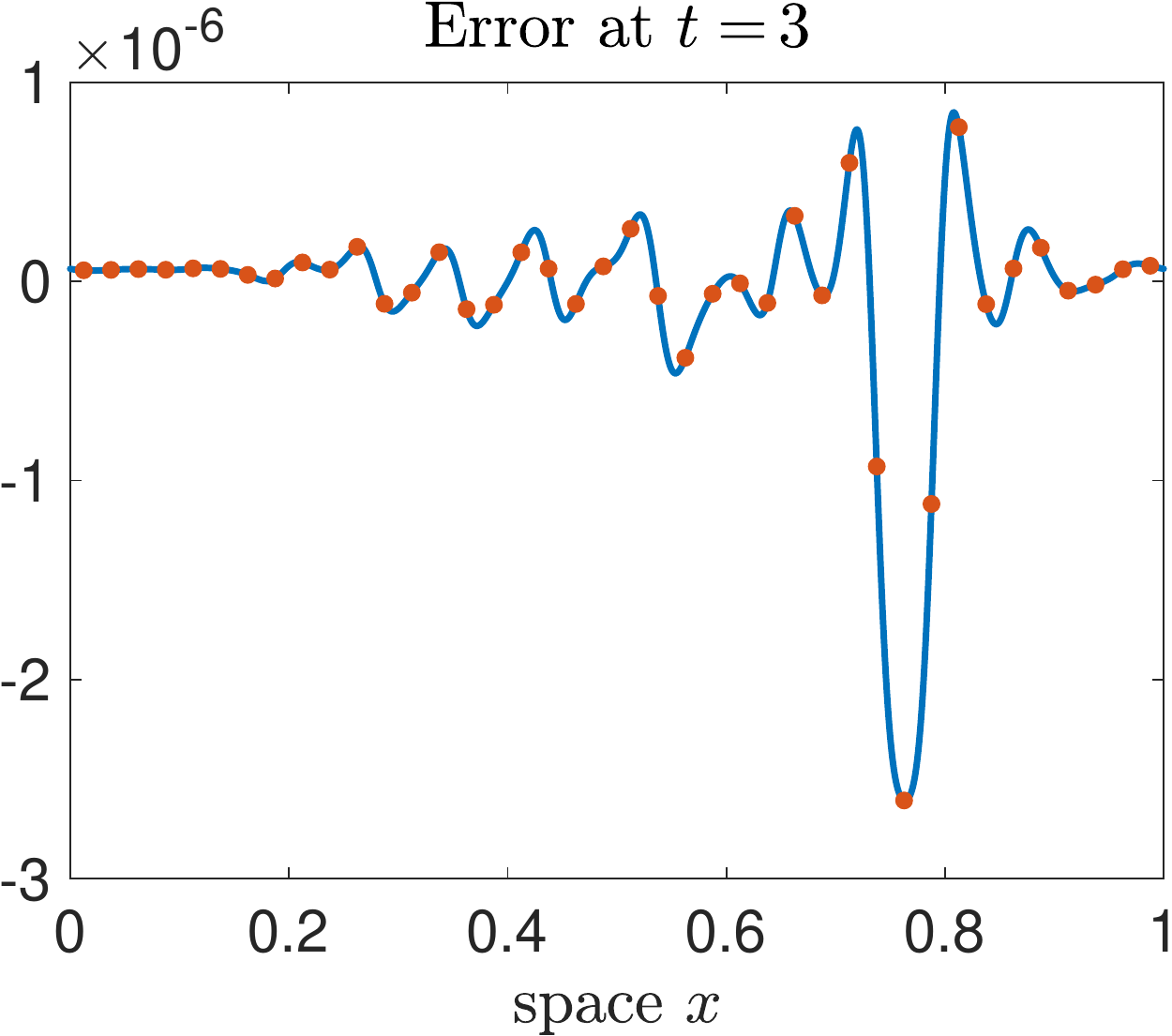}}
\subfigure%[]
{\includegraphics[width=0.30\textwidth,height=0.30\textwidth]{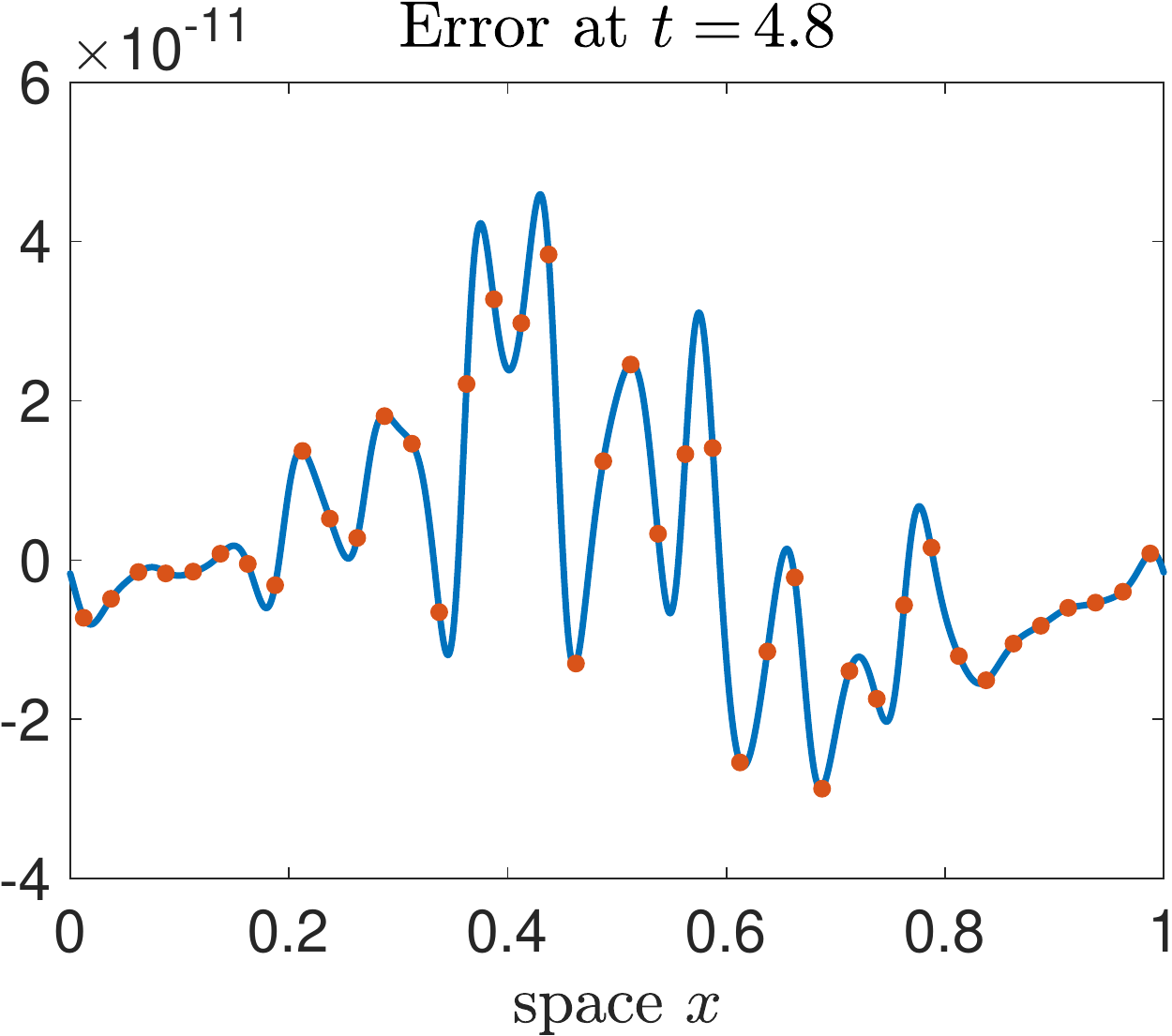}}
\caption{Snapshots of estimate and output error. Model~\eqref{sys-y-num}--\eqref{sys-haty-num}.}
\label{Fig:lam_large-flame-snap}
\end{figure}

\begin{remark}
The small  magnitude of the larger values of~$\lambda\approx 10^{-7}$ in Fig.~\ref{Fig:lam_large-flame} can be understood from the small magnitude of the parameter~$\nu_2=10^{-6}$. Recall that~$A=\nu_2A_0$ with~$A_0=(-\tfrac{\p^2}{\p x^2}+\Id)^2$ and the output injection operator is given by~\eqref{sys-haty-o-Inj}, thus we find that~$\fkI_{S}^{[\lambda,\Lambda]}=-\frac{\lambda}{\nu_2} A_0^{-2}\clZ_S^*\Lambda$ where~$\frac{\lambda}{\nu_2}\approx 10^{-1}$.
\end{remark}

\begin{remark}\label{R:choiceLambda}
The eigenvalues of the operator~$\fkA\coloneqq-\nu_2\tfrac{\p^4}{\p x^4}-\nu_1\tfrac{\p^2}{\p x^2}$ are given by
\begin{align}
&\fka_k=-\nu_2 16\pi^4(\tfrac{k-1}{2})^4+\nu_14\pi^2(\tfrac{k-1}{2})^2,&&\mbox{for~$k$ odd};\notag\\
&\fka_k=-\nu_2 16\pi^4(\tfrac{k}{2})^4+\nu_14\pi^2(\tfrac{k}{2})^2,&&\mbox{for~$k$ even}.\notag
\end{align}
Thus, from
\[
(-\nu_2 \xi^2+\nu_1)\xi^2\ge0\quad\Longleftrightarrow\quad\xi^2\le\tfrac{\nu_1}{\nu_2}
\]
and direct computations we can find that
\begin{align}
&\fka_k\ge0 \quad\Longleftrightarrow\quad k\le\tfrac1{\pi}(\tfrac{\nu_1}{\nu_2})^\frac12+1,&&\mbox{for~$k$ odd};\notag\\
&\fka_k\ge0 \quad\Longleftrightarrow\quad k\le\tfrac1{\pi}(\tfrac{\nu_1}{\nu_2})^\frac12,&&\mbox{for~$k$ even};\notag
\end{align}
and with~$(\nu_2,\nu_1)$ as in~\eqref{param-flame-nominal},
\begin{align}
&\fka_k\ge0 \quad\Longleftrightarrow\quad k\le\tfrac1{\pi}100+1\approx 32.8310,&&\mbox{for~$k$ odd};\notag\\
&\fka_k\ge0 \quad\Longleftrightarrow\quad k\le\tfrac1{\pi}100\approx 31.8310,&&\mbox{for~$k$ even}.\notag
\end{align}

Now, if we look at the 31 first eigenvalues of the operator~$A_0$ appearing in the injection operator we see that they range from~$1$ to~$(4\pi^2(\frac{(31-1)}2)^2+1)^2=((30\pi)^2+1)^2> 10^7$.
This range is too wide and it is (likely) at this point that the choice of the matrix~$\Lambda$ in the output injection operator can play a crucial role. For example, with~$\Lambda=\Id$ we would have that the magnitude of the forcing induced by the output injection operator on each of the first ~$31$ unstable spectral modes (i.e., on the first 31 components of the Galerkin space) would be quite different from each other. If we would need to apply a forcing of magnitude~$\fkm$ in the $31$st mode, this would correspond to a forcing of magnitude~$\fkm\times10^{7}$  in the~$1$st mode, which may be unpractical for applications/numerical simulations (e.g., at least, we may need to take a very small time step to capture/approximate the dynamics induced by such magnitudes, which means that we would likely be not able to compute the estimate in real time). 
A suitable choice of~$\Lambda$ may help to obtain an injection operator  inducing a forcing with closer magnitudes on each of the first~$S_\sigma$ spectral modes; note that with~$\Lambda$ as in~\eqref{param-flame-Lam}, 
\begin{align}\notag
& I(1:S_\sigma,1)= -\lambda\nu_2^{-1} (\underline\bfE_\fkX^\circ)^\perp (\clZ_Sy_\tte^N-\clZ_S y_\ttr^N).
 \end{align}
for the projection of~$I$ onto the first~$S_\sigma$ coordinates. Thus, the construction as in~\eqref{param-flame-Lam} will give us an injection forcing with close magnitudes on the first~$S_\sigma$ spectral modes.
\end{remark}

%%%%%%%%%%%%%%%%%%%%%%%%%%%%%
%%%%%%%%%%%%%%%%%%%%%%%%%%%%%
 \section{Numerical Simulations for the K--S fluid flow model}\label{S:simul-fluid}
 We consider the Kuramoto--Sivashinsky model for fluid flow with output as in~\eqref{sys-Intro-o}, that is, we consider the nominal system
\begin{subequations}\label{sys-y-num-fluid}
  \begin{align}
 &\tfrac{\p}{\p t}y_{\ttr}+\nu_2 \tfrac{\p^4}{\p x^4} y_{\ttr}+\nu_1 \tfrac{\p^2}{\p x^2}  y_{\ttr}+\nu_0y_{\ttr}\tfrac{\p}{\p x} y_{\ttr}=0,\\
 &\clZ_{S} y_{\ttr}\coloneqq\begin{bmatrix}
           y_{\ttr} (x^{S,1}) &  y_{\ttr}(x^{S,2}) & \dots & y_{\ttr}(x^{S,S_\sigma})
         \end{bmatrix}^\top\in\bbR^{S_\sigma\times 1},
\end{align}
\end{subequations}
and the Luenberger observer~\eqref{sys-KS-flame-obs}, with output injection operator~\eqref{sys-haty-o-Inj},
\begin{subequations}\label{sys-haty-num-fluid}
  \begin{align}
 &\tfrac{\p}{\p t}y_\tte+\nu_2 \tfrac{\p^4}{\p x^4} y_\tte+\nu_1 \tfrac{\p^2}{\p x^2} y_\tte+\nu_0y_{\tte}\tfrac{\p}{\p x} y_{\tte}=-\lambda A^{-1}\clZ_S^*\Lambda(\clZ_S y_\tte-\clZ_S y_{\ttr})\\
&\mbox{again with}\quad A=\nu_2(-\tfrac{\p^2}{\p x^2}+\Id)^2,
\end{align}
\end{subequations}
where, again, we take~$\Lambda\in\bbR^{S_\sigma\times S\sigma}$ as in~\eqref{param-flame-Lam}.

We consider all data as in~\eqref{param-flame-nominal}~\eqref{param-flame-inj}~\eqref{param-flame-Lam} with the exception of the coefficient of the nonlinearity which we set as
\[
\nu_0=1,
\]
as in~\cite{RodSeifu22-arx} (motivated by~\cite[sect.~4]{KassamTrefethen05}, after rescaling the spatial and temporal variables, cf.~\cite[Rem.~4.3]{RodSeifu22-arx}; see also~\cite[sect.~5.1]{Krogstad05}).

%%%%%%%%%%%%%%%%%%%%%%%%%%%%%%%%%%%%%
\subsection{Satisfiability of Assumption~\ref{A:realy} and~\ref{A:N}}
First of all we note that the only difference compared to the flame propagation model is the nonlinearity. Thus in order to be able to apply our abstract result to the fluid flow model we need to check Assumption~\ref{A:N}. Further, we need to check Assumption~\ref{A:realy} involving the subspace~$\fkG$ used in Assumption~\ref{A:N}.  That is, we need to revisit the arguments in Sections~\ref{sS:assumpt_aux} and~\ref{sS:assumpt-chkA1N}.

We can see that the average of the solutions of~\eqref{sys-y-num-fluid} is preserved, $\frac{\rmd}{\rmd t}\int_0^1 y_{\ttr}(t,x)\,\rmd x=0$. Since the nonlinearity vanishes again for constant functions, we assume the analogue of Assumption~\ref{A:realy-KS}.
\begin{assumption}\label{A:realy-KS-fluid}
Let~$\fkG=\rmD(A)\bigcap (\bbR\indf_{\bbT_1^1})^\perp$. There are constants~$C_{y_\ttr}\ge0$ and~$\tau_{y_\ttr}>0$ such that the targeted real state~$y_\ttr$, solving~\eqref{sys-y-num-fluid}, satisfies
\begin{align}
\sup_{s\ge0}\norm{y_{\ttr\fkG}(s)}{V}\le C_{y_\ttr}\quad\mbox{and}\quad\sup_{s\ge0}\norm{y_{\ttr\fkG}}{L^2((s,s+\tau_{y_\ttr}),\rmD(A))}<C_{y_\ttr}.\notag
\end{align}
\end{assumption}

Note that if~$y$ solves the flame propagation model, then~$\bfy=\tfrac{\p}{\p x}y$ solves the fluid flow model. Then, results on the satisfiability of the bound required in Assumption~\ref{A:realy-KS-fluid} can be derived from the analogous bounds for the fluid flow model. Recall also Remark~\ref{R:satisf-pers-bdd}.

Next, with~$d=y_2-y_1$, for~$\clN(t,y)\coloneqq \nu_0y\tfrac{\p}{\p x}y$ we find that
\begin{align}
&\nu_0^{-1}(\clN(t,y_1)-\clN(t,y_2))= y_1\tfrac{\p}{\p x}d+d\tfrac{\p}{\p x}y_2,
\notag
\end{align}
from which we obtain
\begin{align}
&\nu_0^{-1}\norm{\clN(t,y_1)-\clN(t,y_2)}{H}\le\norm{y_1}{L^\infty(\bbT_1^1)}\norm{\tfrac{\p}{\p x} d}{L^2(\bbT_1^1)}
+\norm{d}{L^\infty(\bbT_1^1)}\norm{\tfrac{\p}{\p x} y_2}{L^2(\bbT_1^1)}
\notag
\end{align}
and, since~$d=1$, we can use~$V=W^{2,2}(\bbT_1^1)\xhookrightarrow{}W^{1,2}(\bbT_1^1)\xhookrightarrow{}L^\infty(\bbT_1^1)$ to obtain
\begin{align}
2\nu_0^{-1}\norm{\clN(t,y_1)-\clN(t,y_2)}{H}&\le C_1(\norm{y_1}{W^{1,2}(\bbT_1^1)}+\norm{y_2}{W^{1,2}(\bbT_1^1)})\norm{d}{W^{1,2}(\bbT_1^1)}
\notag\\
&\le C_2(\norm{y_1}{V}+\norm{y_2}{V})\norm{d}{V}.
\notag
\end{align}
Thus, Assumption~\ref{A:N} holds with~$n=1$ and~$(\zeta_{11},\zeta_{21},\delta_{11},\delta_{21})=(1,0,1,0)$.

%%%%%%%%%%%%%%%%%%%%%%%%%%%%%%%%%%%%%
\subsection{Performance of the observer}
In Fig.~\ref{Fig:free-fluid} we see that the free dynamics (with~$\lambda=0$) is likely not asymptotically unstable. This fact is supported by the corresponding norm evolution shown in Fig.~\ref{Fig:lam_small-fluid}. In the same figure we also see that small values of~$\lambda$ do not give us an exponential observer able to give us an estimate converging to the targeted state. We can obtain such an observer by increasing~$\lambda$ as shown in Fig.~\ref{Fig:lam_large-fluid}. Time-snapshots  of the state error estimate are given in Fig.~\ref{Fig:lam_large-flame-snap}. Thus, also in this example the simulations agree with the theoretical result.
\begin{figure}[ht]
\centering
\subfigure[targeted state.]
{\includegraphics[width=0.45\textwidth,height=0.35\textwidth]{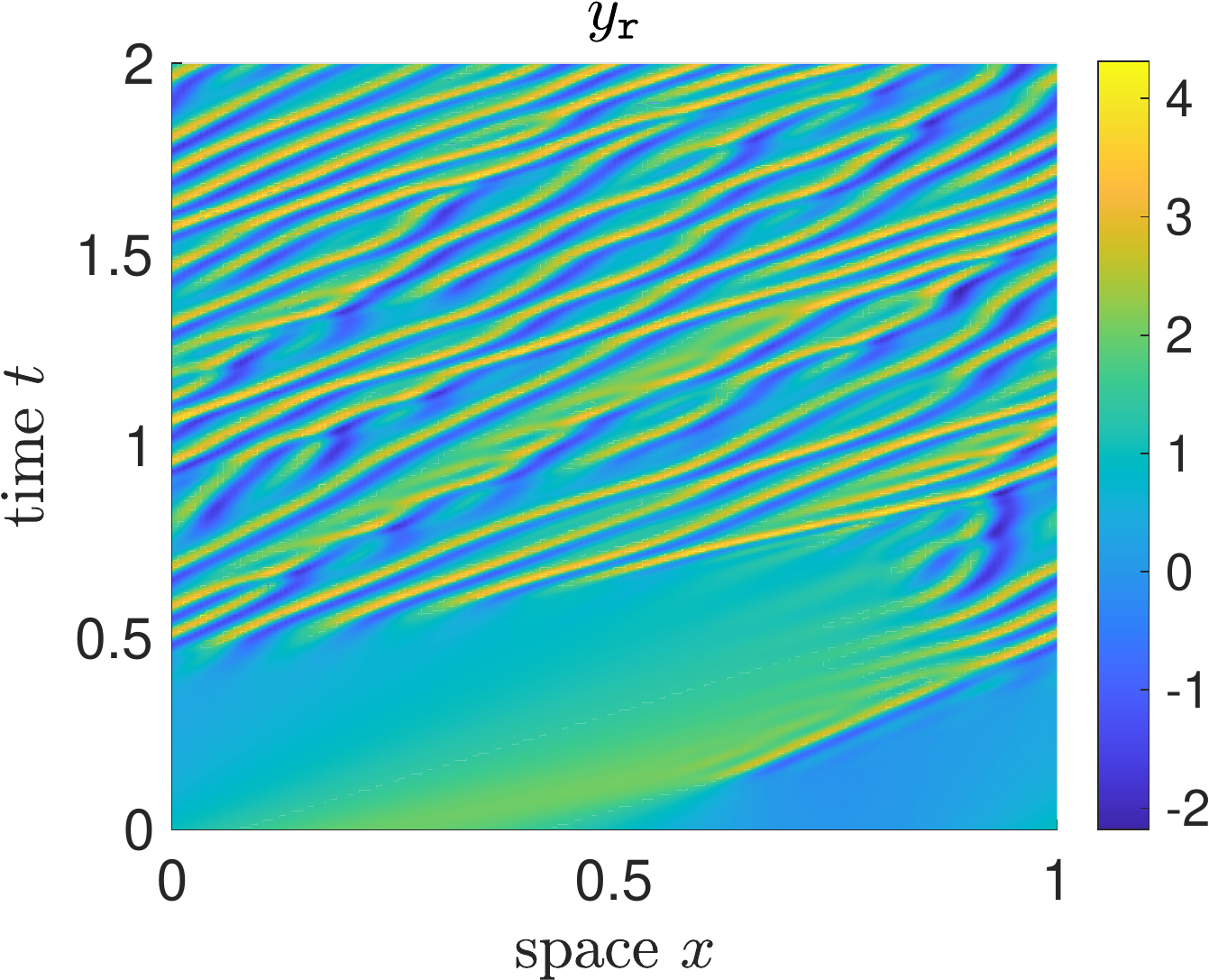}}
\qquad
\subfigure[free estimate.]
{\includegraphics[width=0.45\textwidth,height=0.35\textwidth]{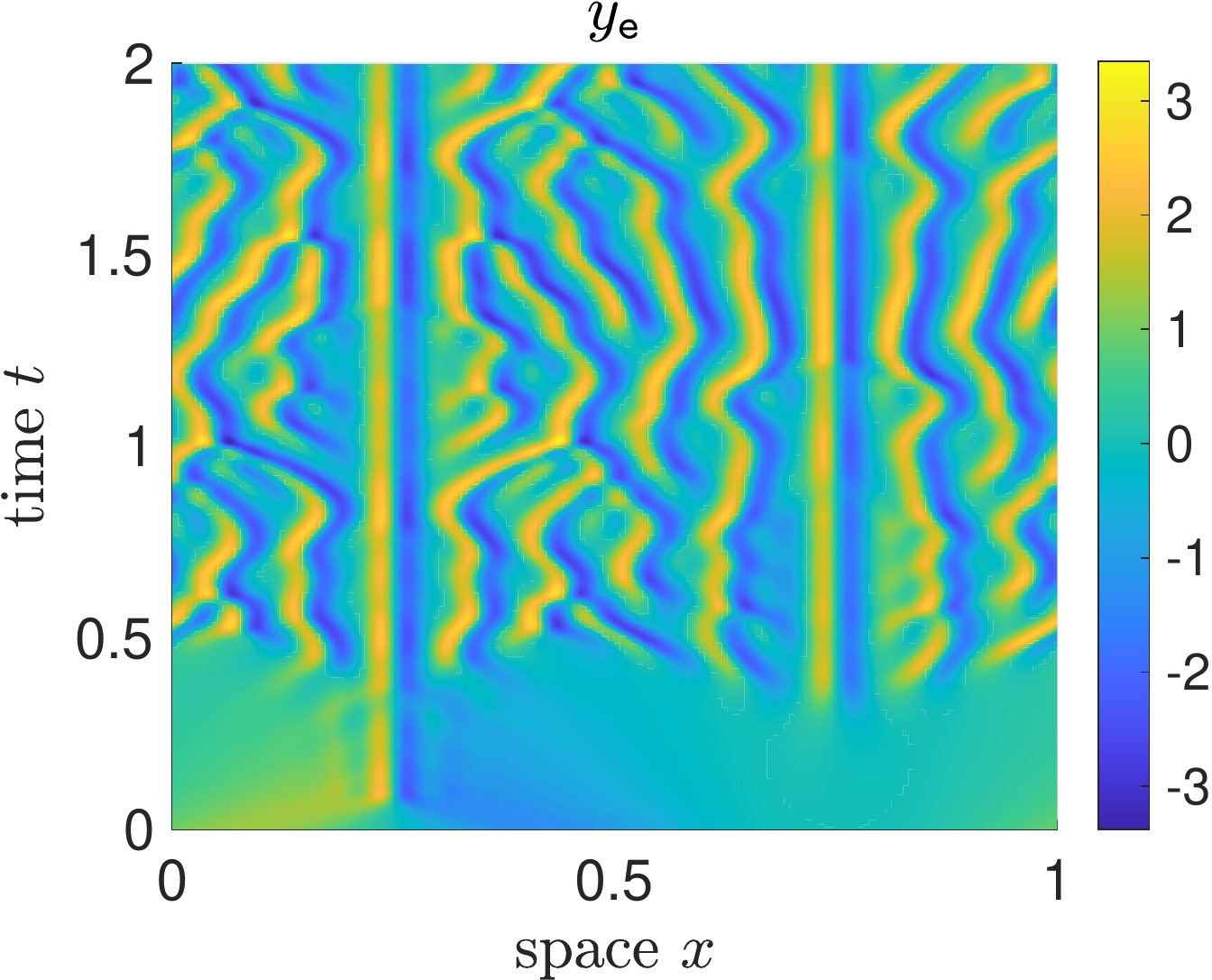}}
\caption{Case~$\lambda=0$ (free dynamics). Model~\eqref{sys-y-num-fluid}--\eqref{sys-haty-num-fluid}.\newline\newline}\label{Fig:free-fluid}
\end{figure}
\begin{figure}[ht]
\centering
\subfigure%[]
{\includegraphics[width=0.45\textwidth,height=0.35\textwidth]{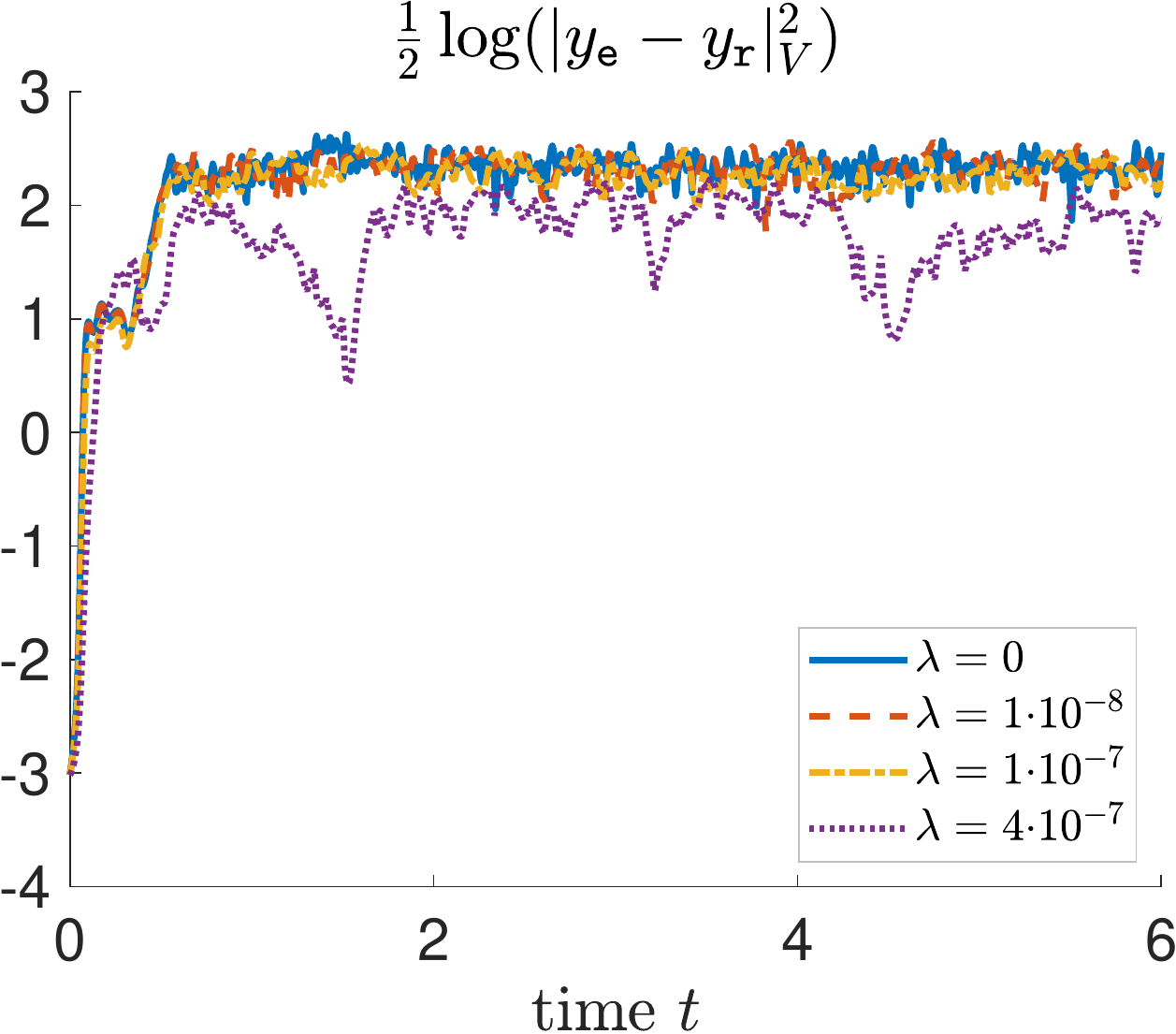}}
\qquad
\subfigure%[]
{\includegraphics[width=0.45\textwidth,height=0.35\textwidth]{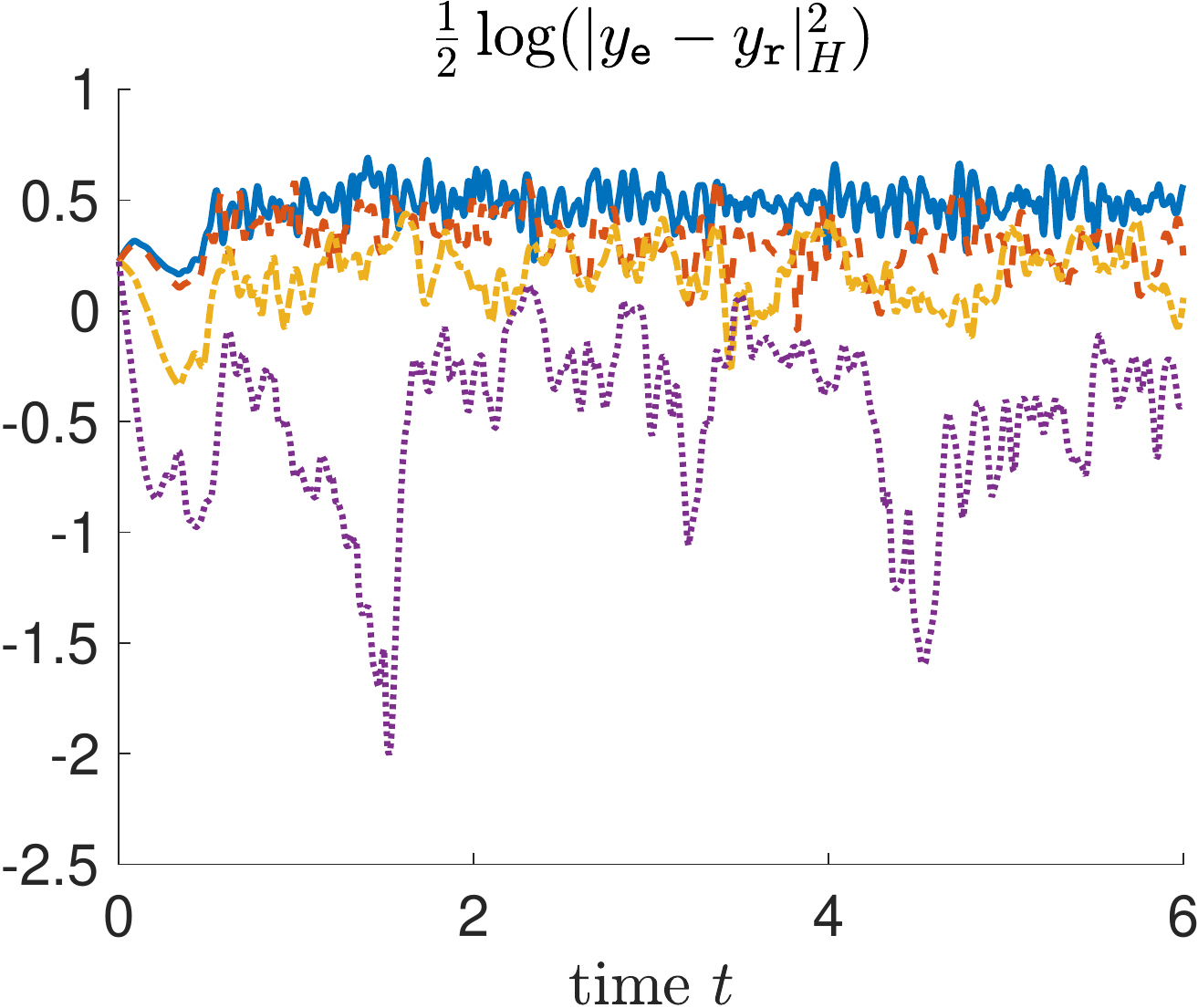}}
\caption{Estimate error  for small~$\lambda\ge0$ in~\eqref{matrix_bfI-spe}. Model~\eqref{sys-y-num-fluid}--\eqref{sys-haty-num-fluid}.\newline\newline}\label{Fig:lam_small-fluid}
\subfigure%[]
{\includegraphics[width=0.45\textwidth,height=0.35\textwidth]{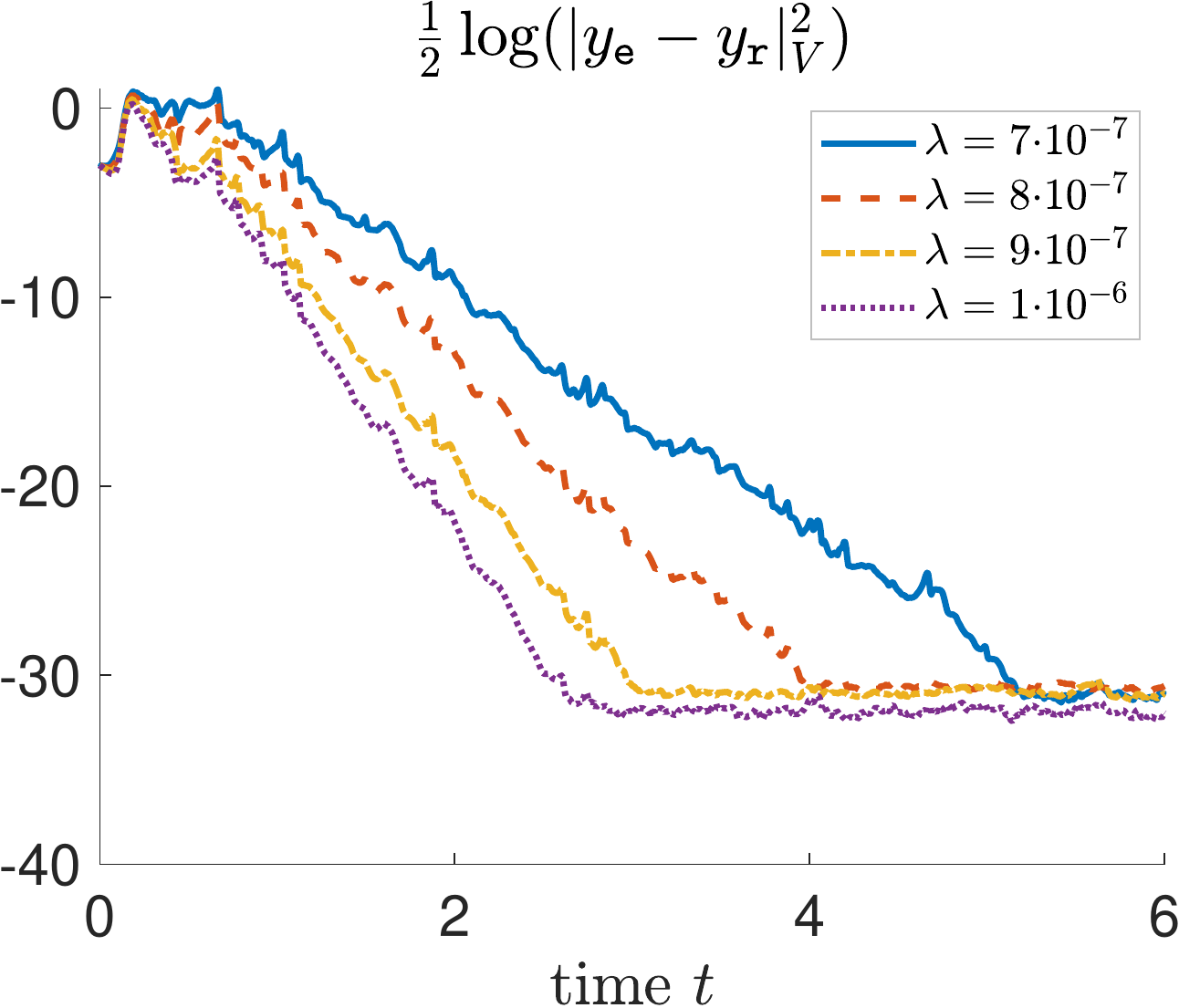}}
\qquad
\subfigure%[]
{\includegraphics[width=0.45\textwidth,height=0.35\textwidth]{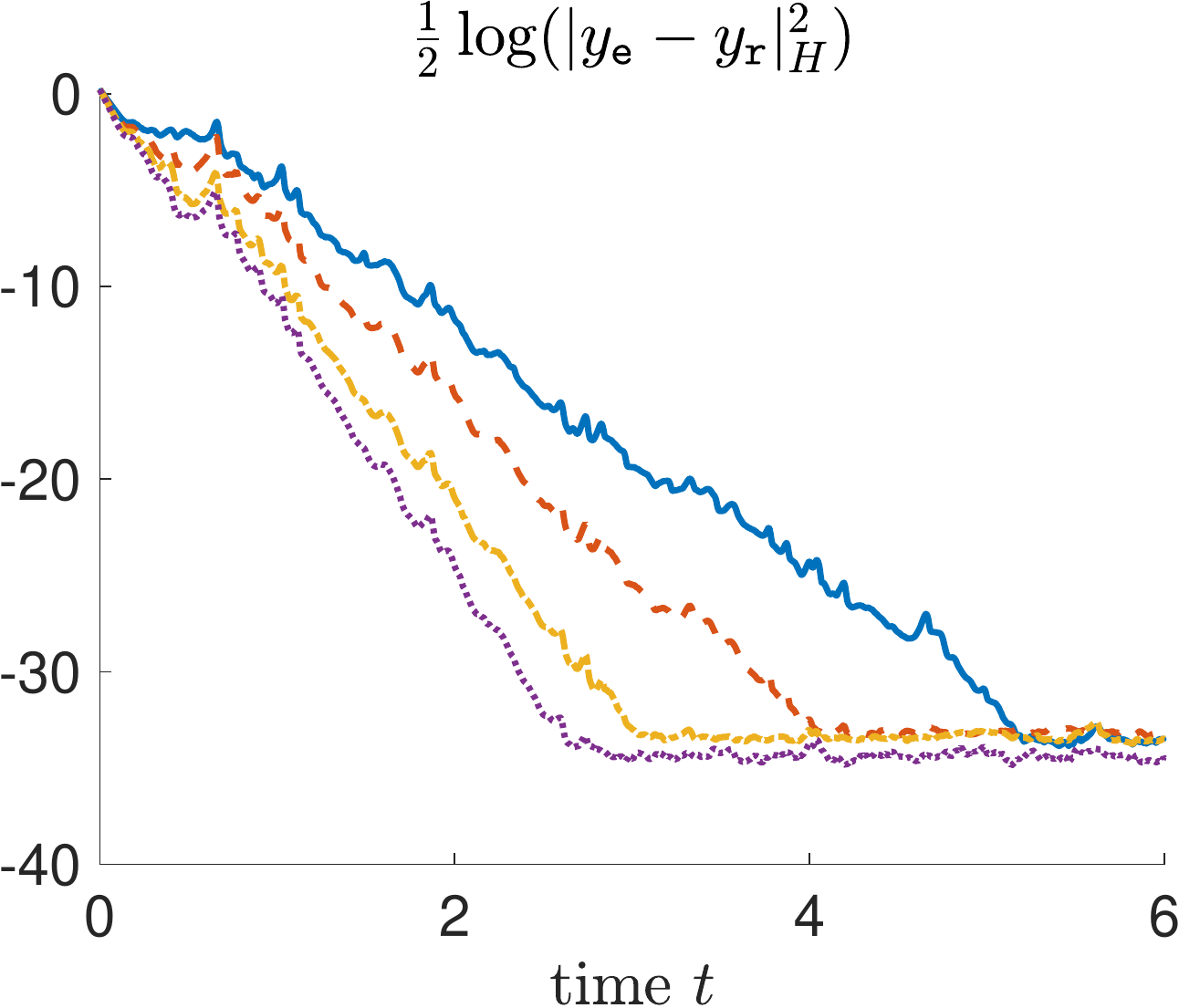}}
\caption{Estimate error  for larger~$\lambda\ge0$ in~\eqref{matrix_bfI-spe}. Model~\eqref{sys-y-num-fluid}--\eqref{sys-haty-num-fluid}.\newline}\label{Fig:lam_large-fluid}
\end{figure}

\begin{figure}[ht]
\centering
\subfigure%[]
{\includegraphics[width=0.32\textwidth,height=0.30\textwidth]{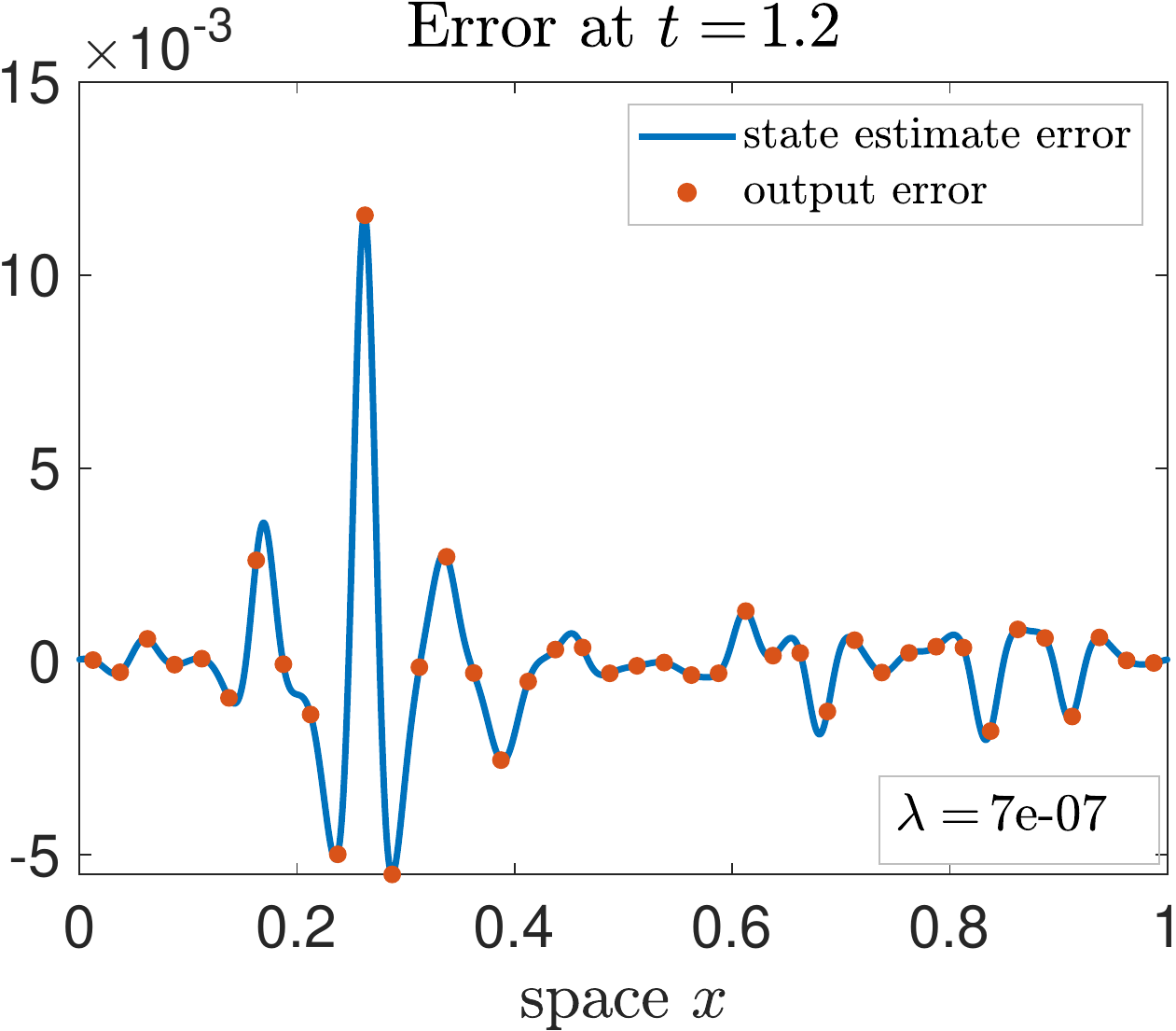}}
\subfigure%[.]
{\includegraphics[width=0.32\textwidth,height=0.30\textwidth]{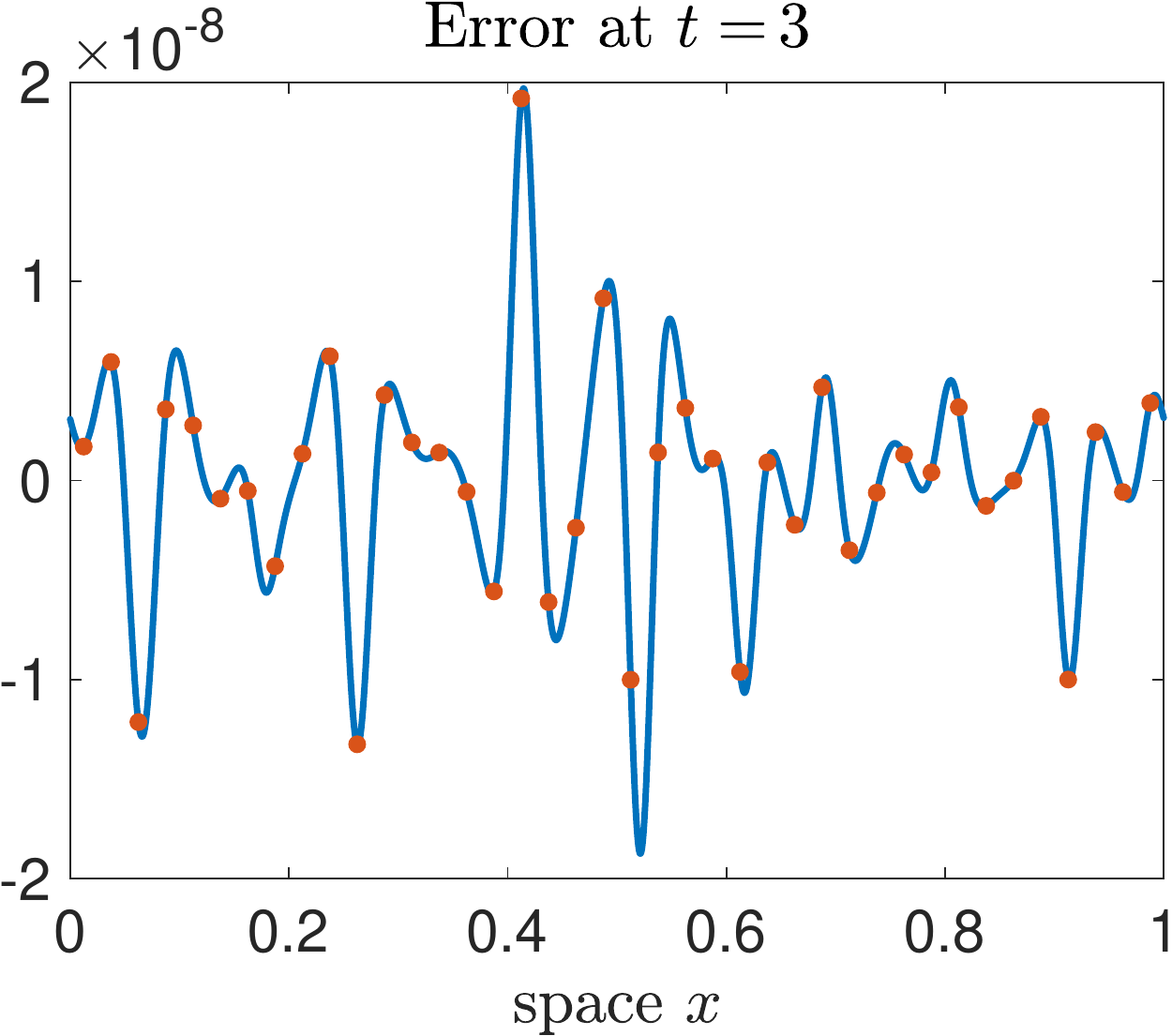}}
\subfigure%[]
{\includegraphics[width=0.32\textwidth,height=0.30\textwidth]{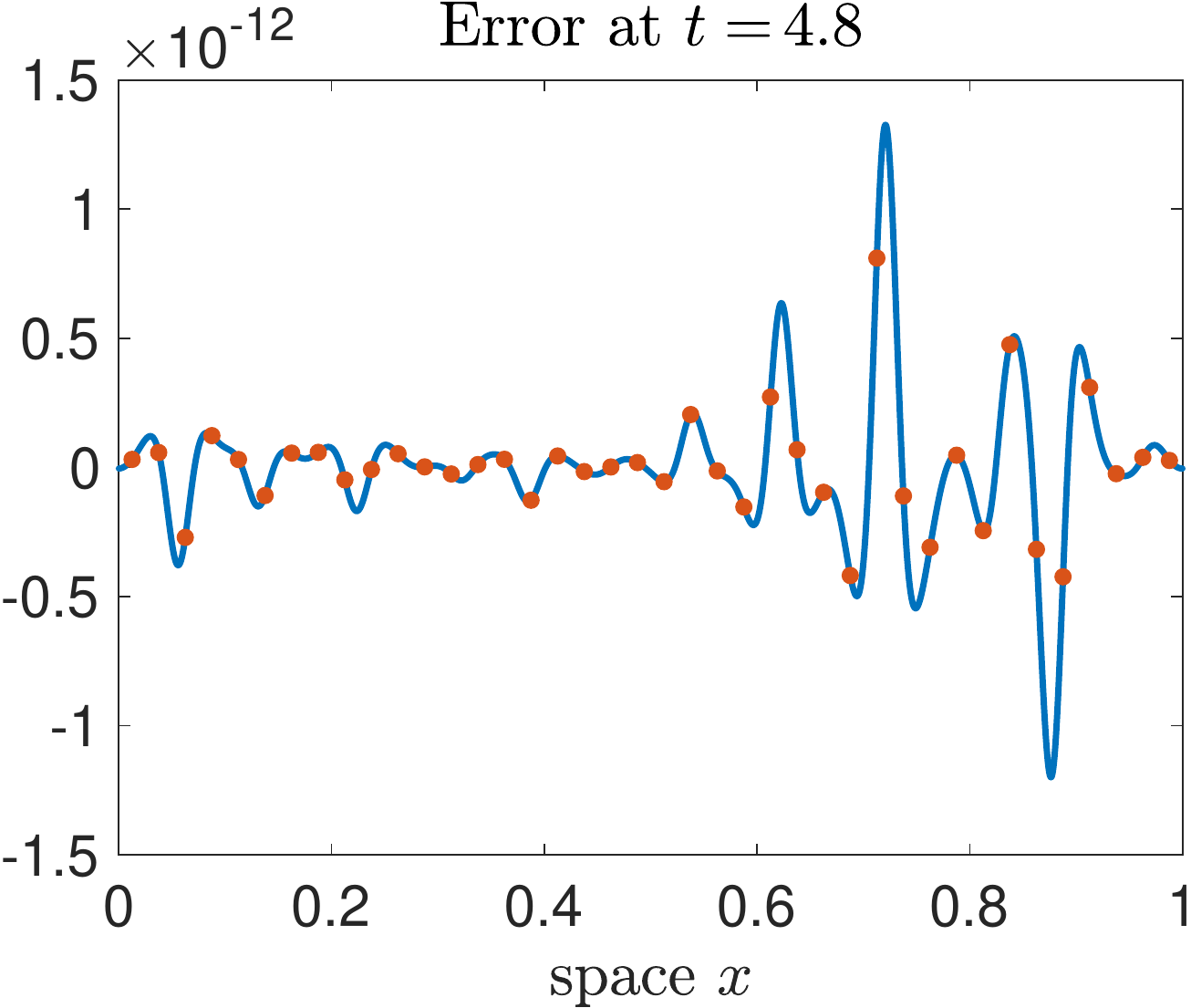}}
\caption{Snapshots of estimate and output error. Model~\eqref{sys-y-num-fluid}--\eqref{sys-haty-num-fluid}.}
\label{Fig:lam_large-fluid-snap}
\end{figure}

%%%%%%%%%%%%%%%%%%%%%%%%%%%%%%%%
\section{Final remarks}\label{S:finalremarks}
We proposed a Luenberger-type observer to estimate the state of parabolic-like equations from the output of the measurement of the state at a finite number of spatial points~$x^i\in\Omega$, $1\le i\le S_\sigma$. These measurements correspond to a finite number of sensors which are the delta distributions located at those spatial points. The result is in fact a corollary of an abstract general result where the set of sensors is a finite set of elements of a suitable Hilbert space, namely, the continuous dual~$V'$ of a regular enough Sobolev space~$V\subset\clC(\Omega)$. Thus the result can cover a more general class of sensors/measurements. We have performed  simulations for the case of the Kuramoto--Sivashinsky equation modeling flame propagation in one-dimensional spatial domain~$\Omega=[0,1)\subset\bbR$ under periodic boundary conditions. We observed the estimate given by the observer converges exponentially to the state of the equation.

The output injection operator is explicit, which makes it easy to implement in applications and enable us to obtain an estimate in real time. This operator involves an arbitrary positive definite matrix~$\Lambda\in\bbR^{S_\sigma\times S_\sigma}$, which can be taken arbitrary at the theoretical level. In the simulations we observed that the choice of~$\Lambda$ can play an important role on the tuning/choice of the observer parameter gain~$\lambda$. Future research could address the choice of~$\Lambda$, for example, either in higher-dimensional  physical spatial domains~$\Omega\subset\bbR^d$, $d\in\{1,2\}$, or under other types or boundary conditions.

The location of the sensors can also play a crucial role on the practicability of the observer in concrete problems, thus, this could also be a subject for future research.

State estimates are demanded in the  implementation of stabilizing feedback control applications, where the control input~$u(t)=K(y(t))$ depends on the entire state~$y(t)$ of the system. This entire state is unavailable, in general, and it is at this point that an estimate~$y_{\rm est}(t)$ provided by an observer can give us a way to compute an approximation~$u_{\rm est}(t)=K(y_{\rm est}(t))$ of such input control.
Hence, another subject of interest is the investigation of such inherent closed-loop systems (coupling an observer with a  feedback control system). This is, in general, a nontrivial problem for nonlinear systems, where the so-called {\em separation principle} does not hold. Note that the free dynamics of the error of the estimate provided by the observer depends on the targeted state~$y_\ttr$, e.g., see~$\fkN_{y_\ttr}$ in~\eqref{sys-z}. In the linear case where~$\clN=0=\fkN_{y_\ttr}$, we see that the dynamics of the error is independent of~$y_\ttr$, in this sense the estimation problem can be {\em separated} from the stabilization problem (cf.~\cite[syst.~(1.12)]{Rod21-aut}).

Finally, from the applications point of view, it would be  interesting to compare the observer we proposed with others we find in the literature. For example, in~\cite[Eq.~(8)]{AzouaniOlsonTiti14}, for an output~$w\in\bbR^{S_\sigma\times1}$ of point measurements the authors propose to use an output injection operator as
$-\lambda {\textstyle\sum\limits_{i=1}^{S_\sigma}} w_{i,1} \indf_{\omega_{S,i}}$. That is, comparing with the one we propose in~\eqref{sys-haty-o-Inj-delta}, for the case~$\Lambda=\Id$, we see that in~\cite{AzouaniOlsonTiti14} it is proposed to use the indicator functions~$\indf_{\omega_{S,i}}$ (for suitable subdomains~$\omega_{S,i}\subset\Omega$) instead of the functions~$A^{-1}\deltafun_{x^{S,i}}$. As mentioned in Remark~\ref{R:choiceLambda} during the numerical simulations we found that the choice of~$\Lambda$ is crucial for the practical performance of our observer, which is likely due to the large range of the eigenvalues of~$A$. Thus, we can guess that the observer proposed in~\cite[Eq.~(8)]{AzouaniOlsonTiti14} could have some advantages since it does not involve~$A^{-1}$. On the other hand, the class of systems in~\cite{AzouaniOlsonTiti14} require the solutions of the free dynamics to be globally defined for time (i.e., for all time~$t>0$) and, in particular, assume the apriori knowledge of the existence of a finite number of determining nodes (for the case of output point measurements). In this manuscript we consider a class of nonlinearities which allow the solutions of the free dynamics to blow up in finite-time, only the targeted solution corresponding to the state to be estimated is required to be defined for all~$t>0$.

 Finally, considering, for simplicity, the case~$\Lambda=\Id$, we can see that the use of~$A^{-1}$ in the output injection operator~$\fkI_S^{[\lambda,\Id]}w=-\lambda {\textstyle\sum\limits_{i=1}^{S_\sigma}} w_{i,1} A^{-1}\deltafun_{x^{S,i}}$, $w=\clZ_Sz$, is related to the fact that we are considering strong solutions and looking for an error decreasing in~$V=\rmD(A^\frac12)$-norm, namely, to have the monotonicity~$(\fkI_S^{[\lambda,\Id]}w,Az)_H=-\lambda\norm{w}{\bbR^{S_\sigma}}^2$. For one-dimensional models (i.e., evolving in spatial intervals~$\Omega\subset\bbR$), for which weak solutions are well defined (for initial states in~$H$ and for external/injection forces taking values in~$V'$), if we aim  for an error decreasing in~$H=\rmD(A^0)$-norm, then we can expect that we can omit~$A^{-1}$ and take~$\underline\fkI_S^{[\lambda,\Id]}w=-\lambda {\textstyle\sum\limits_{i=1}^{S_\sigma}} w_{i,1} \deltafun_{x^{S,i}}$ instead,  which again leads us  to the monotonicity~$\langle\underline\fkI_S^{[\lambda,\Id]}w,z\rangle_{V',V}=-\lambda\norm{w}{\bbR^{S_\sigma}}^2$. Note that for higher-dimensional models (i.e., evolving in spatial domains~$\Omega\subset\bbR^d$, $d>1$) we do not have that~$\deltafun_{x^{S,i}}$, thus we may still need to take~$\underline\fkI_S^{[\lambda,\Id]}w=-\lambda {\textstyle\sum\limits_{i=1}^{S_\sigma}} w_{i,1} A^{-\rho}\deltafun_{x^{S,i}}$ for some positive~$\rho$, so that~$A^{-\rho}\deltafun_{x^{S,i}}$). Of course, the details must be checked depending on the application/model we have at hand. In any case, assumptions on the nonlinearity to deal with weak solutions could be the analogue of Assumption~\ref{A:N} that we find in~\cite[Assum.~2.4]{KunRodWalter21} in the context of feedback stabilization, where essentially the triple~$(H,V,\rmD(A))$ in Assumption~\ref{A:N} is replaced by~$(V',H,V)$.

In real world applications, often measurements are subject to small errors and mathematical models are subject to disturbances/uncertainty. Therefore, the ``robustness'' of the observer that we propose under such noisy measurements and model disturbances is an interesting subject, which could be the focus of a future work. We refer the reader to~\cite[Eq.~(3)]{BessaihOlsonTiti15} where measurements are assumed to be perturbed by a random error.

%\medskip
\appendix
% \addcontentsline{toc}{section}{Appendix}
% \begin{center}
% {\sc --- Appendix ---}
% \end{center}
\section*{Appendix}
%to have the effect of a new section we have to reset the counters
\setcounter{section}{1}%to restart with A (~1 in Alph)
\setcounter{theorem}{0} \setcounter{equation}{0}
\numberwithin{equation}{section}

\subsection{}\label{Apx:balls-and-cells}
Recalling Remark~\ref{R:balls-and-cells} we show here that for~$m\in\bbN_+$ and~$p\in\bbN$, the number of monomials spanning the space of polynomials in the~$m$ variabes~$x_1,\dots,x_m$ with degree at most~$p$ is given by
\begin{equation}\label{monom-lep}
\#\fkM_m^{[\le p]}=\tfrac{(m+p)!}{m!p!}
\end{equation}
where as usual~$(n+1)!\coloneqq (n+1)(n!)$ is the factorial of the positive integer~$n+1$, and~$0!\coloneqq 1$.

We follow the so-called~``Balls\&Cells'' (or, ``Stars\&Bars'') argument popularized in~\cite{Feller68}. This argument is used in~\cite[Ch.~II, sect.~5]{Feller68} to show that the number of nonnegative integer solutions of the equation
\begin{equation}\notag
\kappa_1+\kappa_2+\dots+\kappa_m=k\in\bbN,
\end{equation}
corresponding to the number of monomials of degree ~$k$, is given by
\begin{equation}\label{monom=k}
\#\fkM_m^{[= k]}\coloneqq\tfrac{(m+k-1)!}{k!(m-1)!}.
\end{equation}
Note that we can write~$\#\fkM_m^{[\le p]}=\sum\limits_{k=0}^{p}\#\fkM_m^{[= k]}$. Now, the identity between the latter sum and the claimed quantity~$\frac{(m+p)!}{m!p!}$, as in~\eqref{monom-lep}, is left in~\cite[Ch.~II, sect.~12, Eq.~(12.8)]{Feller68} as an exercise, with several hints. For the convenience of the reader we show now this identity by Induction on~$p$ for a fixed~$m$.
As a base step, for~$p=0$, we have that~$\#\fkM_m^{[\le 0]}=\sum\limits_{k=0}^{0}\#\fkM_m^{[= k]}=\#\fkM_m^{[= 0]}=\frac{(m+0-1)!}{0!(m-1)!}=1=\frac{(m+0)!}{0!m!}$. Next, let~$q\in\bbN$ and, as Induction step, assume that~\eqref{monom-lep} holds with~$p=q$. Then, we find
\begin{align}
 \#\fkM_m^{[\le q+1]}=\sum\limits_{k=0}^{q+1}\#\fkM_m^{[= k]}=\#\fkM_m^{[\le q]}+\#\fkM_m^{[= q+1]}
 =\tfrac{(m+q)!}{m!q!}+\tfrac{(m+q)!}{(q+1)!(m-1)!}=\tfrac{(m+q+1)!}{(q+1)!m!},\notag
 \end{align}
 where we used~\eqref{monom=k} with~$k=q+1$. Therefore,~\eqref{monom-lep} holds with~$p=q+1$.

%%%%%%%%%%%%%%%%%%%%%%%%%%%
%%%%%%%%%%%%%%%%%%%%%%%%%%%

%%%%%%%%%%%%%%%%%%%%%%%%%%%%
%%%%%%%%%%%%%%%%%%%%%%%%%%%%

\bigskip\noindent
{\bf Aknowlegments.}
D. Seifu was supported by the State of Upper Austria and Austrian
Science Fund (FWF): P 33432-NBL, S. Rodrigues acknowledges partial support from the
same grant.

\bibliographystyle{plainurl}

\bibliography{KS_ObsRef}

\begin{thebibliography}{10}

\bibitem{AguiarHesp09}
A.~P. Aguiar and J.~P. Hespanha.
\newblock Robust filtering for deterministic systems with implicit outputs.
\newblock {\em Systems Control Lett.}, 58(4):263--270, 2009.
\newblock \href {http://dx.doi.org/10.1016/j.sysconle.2008.11.005}
  {\path{doi:10.1016/j.sysconle.2008.11.005}}.

\bibitem{AzouaniOlsonTiti14}
A.~Azouani, E.~Olson, and E.~S. Titi.
\newblock Continuous data assimilation using general interpolant observables.
\newblock {\em J. Nonlinear Sci.}, 24(2):277--304, 2014.
\newblock \href {http://dx.doi.org/10.1007/s00332-013-9189-y}
  {\path{doi:10.1007/s00332-013-9189-y}}.

\bibitem{BarRodShi11}
V.~Barbu, S.~S. Rodrigues, and A.~Shirikyan.
\newblock Internal exponential stabilization to a nonstationary solution for
  {3D} {N}avier--{S}tokes equations.
\newblock {\em SIAM J. Control Optim.}, 49(4):1454--1478, 2011.
\newblock \href {http://dx.doi.org/10.1137/100785739}
  {\path{doi:10.1137/100785739}}.

\bibitem{BessaihOlsonTiti15}
H.~Bessaih, E.~Olson, and E.~S. Titi.
\newblock Continuous data assimilation with stochastically noisy data.
\newblock {\em Nonlinearity}, 28(3):729--753, 2015.
\newblock \href {http://dx.doi.org/10.1088/0951-7715/28/3/729}
  {\path{doi:10.1088/0951-7715/28/3/729}}.

\bibitem{BuchotRaymondTiago15}
J.-M. Buchot, J.-P. Raymond, and J.~Tiago.
\newblock Coupling estimation and control for a two dimensional {B}urgers type
  equation.
\newblock {\em ESAIM Control Optim. Calc. Var.}, 21(2):535--560, 2015.
\newblock \href {http://dx.doi.org/10.1051/cocv/2014037}
  {\path{doi:10.1051/cocv/2014037}}.

\bibitem{DemengelDem12}
F.~Demengel and G.~Demengel.
\newblock {\em Functional Spaces for the Theory of Elliptic Partial
  Differential Equations}.
\newblock Universitext. Springer, 2012.
\newblock \href {http://dx.doi.org/10.1007/978-1-4471-2807-6}
  {\path{doi:10.1007/978-1-4471-2807-6}}.

\bibitem{Feller68}
W.~Feller.
\newblock {\em An Introduction to Probability Theory and Its Applications,
  Volume 1, 3rd Edition}, volume~1.
\newblock John Wiley \&Sons, 3rd edition, 1968.

\bibitem{FengMazzucato21}
Y.~Feng and A.~L. Mazzucato.
\newblock Global existence for the two-dimensional {K}uramoto--{S}ivashinsky
  equation with advection.
\newblock {\em Comm. Partial Differential Equations}, 79(2):279--306, 2021.
\newblock \href {http://dx.doi.org/10.1080/03605302.2021.1975131}
  {\path{doi:10.1080/03605302.2021.1975131}}.

\bibitem{GiacomelliOtto05}
L.~Giacomelli and F.~Otto.
\newblock New bounds for the {K}uramoto--{S}ivashinsky equation.
\newblock {\em Comm. Pure Appl. Math.}, 58(3):297--318, 2005.
\newblock \href {http://dx.doi.org/10.1002/cpa.20031}
  {\path{doi:10.1002/cpa.20031}}.

\bibitem{GoluskinFantuzzi19}
D.~Goluskin and G.~Fantuzzi.
\newblock Bounds on mean energy in the {K}uramoto--{S}ivashinsky equation
  computed using semidefinite programming.
\newblock {\em Nonlinearity}, 32(5):1705--1730, 2019.
\newblock \href {http://dx.doi.org/10.1088/1361-6544/ab018b}
  {\path{doi:10.1088/1361-6544/ab018b}}.

\bibitem{JadachowskiMeurerKugi13}
L.~Jadachowski, T.~Meurer, and A.~Kugi.
\newblock State estimation for parabolic {PDE}s with reactive-convective
  non-linearities.
\newblock In {\em Proceedings of the 2013 European Control Conference (ECC),
  Zurich, Switzerland}, pages 1603--1608, July 2013.
\newblock \href {http://dx.doi.org/10.23919/ECC.2013.6669588}
  {\path{doi:10.23919/ECC.2013.6669588}}.

\bibitem{JadachowskiMeurerKugi15}
L.~Jadachowski, T.~Meurer, and A.~Kugi.
\newblock Backstepping observers for linear {PDE}s on higher-dimensional
  spatial domains.
\newblock {\em Automatica J. IFAC}, 51:85--97, 2015.
\newblock \href {http://dx.doi.org/10.1016/j.automatica.2014.10.108}
  {\path{doi:10.1016/j.automatica.2014.10.108}}.

\bibitem{KangFridman19}
W.~Kang and E.~Fridman.
\newblock Distributed stabilization of {K}orteweg--{deVries}--{B}urgers
  equation in the presence of input delay.
\newblock {\em Automatica J. IFAC}, 100:260--263, 2019.
\newblock \href {http://dx.doi.org/10.1016/j.automatica.2018.11.025}
  {\path{doi:10.1016/j.automatica.2018.11.025}}.

\bibitem{KatzFridman22}
W.~Kang and E.~Fridman.
\newblock Finite-dimensional boundary control of the linear
  {K}uramoto--{S}ivashinsky equation under point measurement with guaranteed l2
  -gain.
\newblock {\em IEEE Trans. Automat. Control}, (to appear), 2022.
\newblock \href {http://dx.doi.org/10.1109/TAC.2021.3121234, IEEE}
  {\path{doi:10.1109/TAC.2021.3121234, IEEE}}.

\bibitem{KassamTrefethen05}
A.-K. Kassam and L.~N. Trefethen.
\newblock Fourth-order time-stepping for stiff {\sc pde}s.
\newblock {\em SIAM J. Sci. Comput.}, 26(4):1214--1233, 2005.
\newblock \href {http://dx.doi.org/10.1137/S1064827502410633}
  {\path{doi:10.1137/S1064827502410633}}.

\bibitem{Krogstad05}
S.~Krogstad.
\newblock Generalized integrating factor methods for stiff pdes.
\newblock {\em J. Comput. Phys.}, 203(1):72--88, 2005.
\newblock \href {http://dx.doi.org/10.1016/j.jcp.2004.08.006}
  {\path{doi:10.1016/j.jcp.2004.08.006}}.

\bibitem{KunRodWalter21}
K.~Kunisch, S.~S. Rodrigues, and D.~Walter.
\newblock Learning an optimal feedback operator semiglobally stabilizing
  semilinear parabolic equations.
\newblock {\em Appl. Math. Optim.}, 84(1):277--318, 2021.
\newblock \href {http://dx.doi.org/10.1007/s00245-021-09769-5}
  {\path{doi:10.1007/s00245-021-09769-5}}.

\bibitem{MarkowichTitiTrabelsi16}
P.~A. Markowich, E.~S. Titi, and S.~Trabelsi.
\newblock Continuous data assimilation for the three-dimensional
  {B}rinkman--{F}orchheimer-extended {D}arcy model.
\newblock {\em Nonlinearity}, 29(4):1292--1328, 2016.
\newblock \href {http://dx.doi.org/10.1088/0951-7715/29/4/1292}
  {\path{doi:10.1088/0951-7715/29/4/1292}}.

\bibitem{Meurer13}
T.~Meurer.
\newblock On the extended {L}uenberger-type observer for semilinear
  distributed-parameter systems.
\newblock {\em IEEE Trans. Automat. Control}, 58(7):1732--1743, 2013.
\newblock \href {http://dx.doi.org/10.1109/TAC.2013.2243312}
  {\path{doi:10.1109/TAC.2013.2243312}}.

\bibitem{MeurerKugi09}
T.~Meurer and A.~Kugi.
\newblock Tracking control for boundary controlled parabolic {PDE}s with
  varying parameters: Combining backstepping and differential flatness.
\newblock {\em Automatica J. IFAC}, 45:1182--1194, 2009.
\newblock \href {http://dx.doi.org/10.1016/j.automatica.2009.01.006}
  {\path{doi:10.1016/j.automatica.2009.01.006}}.

\bibitem{NicolaenkoSchTem85}
B.~Nicolaenko, B.~Scheurer, and R.~Temam.
\newblock Some global dynamical properties of the {K}uramoto--{S}ivashinsky
  equations: Nonlinear stability and attractors.
\newblock {\em Phys. D}, 16(2):155--183, 1985.
\newblock \href {http://dx.doi.org/10.1016/0167-2789(85)90056-9}
  {\path{doi:10.1016/0167-2789(85)90056-9}}.

\bibitem{OlsonTiti03}
E.~Olson and E.~S. Titi.
\newblock Determining modes for continuous data assimilation in {2D}
  turbulence.
\newblock {\em J. Stat. Phys.}, 113(5-6):799--840, 2003.
\newblock \href {http://dx.doi.org/10.1023/A:1027312703252}
  {\path{doi:10.1023/A:1027312703252}}.

\bibitem{Otto09}
F.~Otto.
\newblock Optimal bounds on the {K}uramoto--{S}ivashinsky equation.
\newblock {\em J. Funct. Anal.}, 257(7):2188--2245, 2009.
\newblock \href {http://dx.doi.org/10.1016/j.jfa.2009.01.034}
  {\path{doi:10.1016/j.jfa.2009.01.034}}.

\bibitem{RamdaniTucsnakValein16}
K.~Ramdani, M.~Tucsnak, and J.~Valein.
\newblock Detectability and state estimation for linear age-structured
  population diffusion models.
\newblock {\em ESAIM: M2AN}, 50(6):1731--1761, 2016.
\newblock \href {http://dx.doi.org/10.1051/m2an/2016002}
  {\path{doi:10.1051/m2an/2016002}}.

\bibitem{Rod20-eect}
S.~S. Rodrigues.
\newblock Semiglobal exponential stabilization of nonautonomous semilinear
  parabolic-like systems.
\newblock {\em Evol. Equ. Control Theory}, 9(3):635--672, 2020.
\newblock \href {http://dx.doi.org/10.3934/eect.2020027}
  {\path{doi:10.3934/eect.2020027}}.

\bibitem{Rod21-amo}
S.~S. Rodrigues.
\newblock Feedback boundary stabilization to trajectories for {3D}
  {N}avier--{S}tokes equations.
\newblock {\em Appl. Math. Optim.}, 84(2), 2021.
\newblock S1149--S1186.
\newblock \href {http://dx.doi.org/10.1007/s00245-017-9474-5}
  {\path{doi:10.1007/s00245-017-9474-5}}.

\bibitem{Rod21-sicon}
S.~S. Rodrigues.
\newblock Oblique projection exponential dynamical observer for nonautonomous
  linear parabolic-like equations.
\newblock {\em SIAM J. Control Optim.}, 59(1):464--488, 2021.
\newblock RICAM Report no.~2020-33.
\newblock \href {http://dx.doi.org//10.1137/19M1278934}
  {\path{doi:/10.1137/19M1278934}}.

\bibitem{Rod21-aut}
S.~S. Rodrigues.
\newblock Oblique projection output-based feedback exponential stabilization of
  nonautonomous parabolic equations.
\newblock {\em Automatica J. IFAC}, 129:109621, 2021.
\newblock \href {http://dx.doi.org/10.1016/j.automatica.2021.109621}
  {\path{doi:10.1016/j.automatica.2021.109621}}.

\bibitem{Rod21-jnls}
S.~S. Rodrigues.
\newblock Semiglobal oblique projection exponential dynamical observers for
  nonautonomous semilinear parabolic-like equations.
\newblock {\em J. Nonlin. Sci.}, 31:100, 2021.
\newblock \href {http://dx.doi.org/10.1007/s00332-021-09756-8}
  {\path{doi:10.1007/s00332-021-09756-8}}.

\bibitem{RodSeifu22-arx}
S.~S. Rodrigues and D.~Seifu.
\newblock Feedback semiglobal stabilization to trajectories for the
  {K}uramoto--{S}ivashinsky equation.
\newblock ArXiv:2205.13967v1 [math.OC], 2022.
\newblock \href {http://dx.doi.org/10.48550/arXiv.2205.13967}
  {\path{doi:10.48550/arXiv.2205.13967}}.

\bibitem{RodSturm20}
S.~S. Rodrigues and K.~Sturm.
\newblock On the explicit feedback stabilisation of one-dimensional linear
  nonautonomous parabolic equations via oblique projections.
\newblock {\em IMA J. Math. Control Inform.}, 37(1):175--207, 2020.
\newblock \href {http://dx.doi.org/10.1093/imamci/dny045}
  {\path{doi:10.1093/imamci/dny045}}.

\bibitem{SmyshlyaevKrstic05}
A.~Smyshlyaev and M.~Krstic.
\newblock Backstepping observers for a class of parabolic pdes.
\newblock {\em Systems Control Lett.}, 54(7):613--625, 2005.
\newblock \href {http://dx.doi.org/10.1016/j.sysconle.2004.11.001}
  {\path{doi:10.1016/j.sysconle.2004.11.001}}.

\bibitem{Temam01}
R.~Temam.
\newblock {\em {N}avier--{S}tokes Equations: Theory and Numerical Analysis}.
\newblock AMS Chelsea Publishing, Providence, RI, {reprint of the 1984}
  edition, 2001.
\newblock Date of access July 12, 2018.
\newblock URL: \url{https://bookstore.ams.org/chel-343-h}.

\bibitem{Wu74}
M.Y. Wu.
\newblock A note on stability of linear time-varying systems.
\newblock {\em IEEE Trans. Automat. Control}, 19(2):162, 1974.
\newblock \href {http://dx.doi.org/10.1109/TAC.1974.1100529}
  {\path{doi:10.1109/TAC.1974.1100529}}.

\end{thebibliography}

\end{document}